\documentclass[12pt
]{amsart}
\usepackage{amssymb,amsthm,verbatim,amsfonts}
\usepackage{graphicx}
\usepackage{enumitem}
\usepackage{color}

\newcommand{\dist}{\operatorname{dist}}
\newcommand{\area}{\operatorname{area}}
\newcommand{\diam}{\operatorname{diam}}
\newcommand{\reals}{\mathbb{R}}
\newcommand{\complex}{\mathbb{C}}
\newcommand{\disk}{\mathbb{D}}
\newcommand{\naturals}{\mathbb{N}}
\newcommand{\rationals}{\mathbb{Q}}
\newcommand{\integers}{\mathbb{Z}}
\newcommand{\rhp}{\mathbb{H}_r}
\newcommand{\lhp}{\mathbb{H}_l}
\newcommand{\uhp}{\mathbb{H}}
\newcommand{\classS}{\mathcal{S}}
\newcommand{\classB}{\mathcal{B}}

\newcommand{\Julia}{\mathcal{J}}
\newcommand{\Fatou}{\mathcal{F}}

\def\var1epsilon{\epsilon}
\newcommand{\Real}{\operatorname{Re}}

\theoremstyle{plain}                    
\newtheorem{thm}{Theorem}[section]
\newtheorem{cor}[thm]{Corollary}
\newtheorem{lem}[thm]{Lemma}
\newtheorem*{lem*}{Lemma}

\newcounter{ques}
   
\numberwithin{equation}{section}

\theoremstyle{remark}

\newtheorem*{remark*}{Remark}

\numberwithin{equation}{section}

\makeatletter
\renewcommand\subsection{\@startsection{subsection}{2}%
	\z@{.5\linespacing\@plus.7\linespacing}{-.5em}%
	{\normalfont\scshape}}
\makeatother

\begin{document}
\baselineskip=18pt


\title[Julia sets with dimension near one]{Speiser class Julia sets with dimension near one}

\begin{author}[S. Albrecht]{Simon Albrecht}
\address{ %
 Department of Mathematical Sciences,
University of Liverpool,
Liverpool, L69 7ZL, United Kingdom}
\email{Simon.Albrecht@liverpool.ac.uk}
\thanks{S. Albrecht is supported by the Deutsche Forschungsgemeinschaft, grant no. AL 2028/1-1}
\end{author}

\begin{author}[C.~J.~Bishop]{Christopher J. Bishop}
\address{
 Department of Mathematics,
Stony Brook University,
Stony Brook, NY 11794-3651 USA}
\email{bishop@math.stonybrook.edu}
\thanks{C. Bishop is partially supported by NSF Grant DMS 19-06259}
\end{author}

\date{January 20, 2020}


\begin{abstract}
For any $\delta >0$ we construct an entire function $f$ 
with three singular values  whose Julia set has Hausdorff
dimension at most $ 1+\delta$.   Stallard
proved  that the dimension must be strictly larger than $1$ 
whenever $f$ has a bounded singular set, but 
no examples with finite singular set and dimension
strictly less than $2$ were previously known.
\end{abstract}
\maketitle
\clearpage


\setcounter{page}{1}
\renewcommand{\thepage}{\arabic{page}}
\section{Introduction} \label{Intro} 

Suppose $f$ is an  entire function. 
The Fatou set $\Fatou (f)$ 
is the union of all open disks on which the iterates
 $f, f^2, f^3, \dots$  form a normal family and the
 Julia set $\Julia (f)$ is the complement of this set.
 In  1975  Baker \cite{MR0402044} proved that  if $f$ 
is transcendental (i.e., not a polynomial), then    the
 Fatou set has no unbounded, multiply connected components. 
This implies  the   Julia set contains a non-trivial 
continuum and
 hence has Hausdorff dimension  at least $ 1$, but it
 is difficult to build examples that come close to 
 attaining this minimum; constructing such 
 examples is the transcendental counterpart of  
 finding polynomial Julia sets with 
 dimension near 2 (e.g., \cite{MR2950763}, \cite{MR1626737}, 
 \cite{MR1650223}).
 For transcendental entire functions, finding
 ``large'' Julia sets is easier:
Misiurewicz  \cite{MR627790} proved that the Julia set
of  $f(z) = \exp(z)$ is the whole plane, and 
McMullen \cite{MR871679} gave explicit families 
where the Julia set is not the whole plane, 
but still has dimension $2$ (even  positive area).
Stallard \cite{MR1458228}, \cite{MR1760674}
proved  that   the Hausdorff dimension of  a
transcendental Julia set  can attain every value in
the interval $(1,2]$,
and the second author \cite{MR3787831} recently
constructed a transcendental Julia set with
dimension $1$, Baker's lower bound.

The singular set of an entire function $f$ is the closure
 of its critical values and finite asymptotic values 
(limits of $f$ along a curve to $\infty$)
 and will be denoted $S(f)$.  The Eremenko-Lyubich 
class $\classB$ consists of functions such that $S(f)$
 is a bounded set (such functions are also called  bounded-type).
The Speiser class $ \classS \subset \classB$  
consists of those
 functions for which $S(f)$ is a finite set.
These are important classes in transcendental dynamics
and it is an interesting problem to understand their 
differences and similarities. 
 For example, functions in $\classS$ can't have
 wandering domains, whereas those in $\classB$ can,
 \cite{Bishop-classS}, \cite{MR1196102}, \cite{MR857196}.
 Stallard's  examples with $1< \dim(\Julia) < 2$ are in the 
 Eremenko-Lyubich class,  
 and in this paper we show that
such  examples also exist in the Speiser
class.

\begin{thm}\label{dno thm} 
 $\inf \{ \dim(\Julia(f)): f \in \classS \} =1$.
\end{thm} 

Note that we do not claim that every dimension 
between $1$ and $2$ occurs; this remains an 
open problem.
Theorem \ref{dno thm}  is sharp in the sense
that Stallard \cite{MR1764934}
proved that $\dim(\Julia(f)) >1$ for any $f \in \classB$ 
(her result has  been extended beyond class
$\classB$ in several  papers, e.g., 
\cite{MR2480096}, \cite{MR2439666},  \cite{Waterman_VWdisks}).
Moreover, Rippon and Stallard \cite{MR2218773}
have shown that the 
packing dimension is always 2 for $f \in \classB$, so this 
holds for our examples as well. See
 \cite{MR3616046} for the  definitions of  Hausdorff
and  packing dimension and their basic properties.
Our examples all have exactly three singular values: $\pm 1$
each occur as critical values  for infinitely many critical 
points, and $0$ occurs as a critical value once, and as an
asymptotic value finitely often. 

Theorem \ref{dno thm} will be proven
using the quasiconformal folding construction
of the second author, which is a method of 
associating entire functions to certain infinite 
planar graphs. It was introduced in 
\cite{Bishop-classS}, and applied there  to construct 
various new examples, such as 
the Eremenko-Lyubich functions with wandering domains
mentioned above. 
More recently,   quasiconformal folding 
 has been used by
 Fagella, Godillon and   Jarque \cite{MR3339086},  
Lazebnik \cite{MR3579902},
 Osborne and Sixsmith \cite{MR3525384}, 
and  Rempe-Gillen  \cite{arc-like}  to construct 
other examples in the Speiser and
Eremenko-Lyubich classes. \cite{MR4023391}
gives an application to meromorphic dynamics.
The  proof of Theorem  \ref{dno thm}  uses  a variation 
of the  folding theorem that requires 
more  precise estimates than in earlier applications,
but that gives even greater control over the resulting function.
This should be useful for future problems.
Details of the folding  construction
 will be reviewed in Section \ref{review sec}.


The main idea is simple to explain: 
we will  build  an entire function 
$f \in \classS$ so that $f(0)=0$ is an 
attracting fixed point and so that there is a large 
disk $D(0,R) \subset \Fatou(f)$ 
that maps into itself.  Therefore, 
$ \Julia(f) \subset X = \bigcap_{k\in\naturals} X_k$ where
\[
X_k=\left\{ z\in\complex~:~|f^n(z)|
   \geq  R \text{ for all }
 n=1,2,\dots , k\right\}.
\]
Given  $\delta >0$, we construct $f  \in \classS $
so that $X_1$ 
can be covered by disks  $\{D_j\}$ so that
\begin{eqnarray} \label{initial sum}
\sum_j \diam\left(D_j\right)^{1+\delta}= M < \infty.
\end{eqnarray} 
Moreover, any disk $D=D(x,r)$ with 
$D \cap D(0,R) = \emptyset$, will satisfy  
\begin{equation}\label{main est}
\sum\diam\left(f^{-1}\left(D\right)\right)^{1+\delta}
\leq  \var1epsilon\cdot \diam(D)^{1+\delta},
\end{equation} 
for some $\var1epsilon < 1$, 
where the sum is over all  connected components 
of $f^{-1}(D)$.
By induction,  $X_k$ can be covered 
by a union of sets $S_j$ so that 
$\sum \diam(S_j)^{1+\delta} \leq \var1epsilon^{k} M \to 0$.
By definition, 
this implies $\dim(X) \leq 1+\delta$.

Since the tracts of our examples (the connected components of $\{z: |f(z)|>R\}$ for $R$ large) are all contained in half-strips (see Lemma 
\ref{tracts in strips}), it is easy to
verify directly  that our functions  have infinite order
 of growth:
$$  \rho(f)  = 
    \limsup_{|z| \to \infty} \frac {\log \log |f(z)|}
                  {\log |z|} = \infty.$$
This is necessary:  
 Bara{\'n}ski \cite{MR2464786} and Schubert
 \cite{SchubertPhD} independently proved 
that the Julia set of any finite-order Eremenko-Lyubich 
function has Hausdorff dimension $2$. 

Our examples  have Julia sets
that are  Cantor bouquets. See Section \ref{topology} for 
the precise definition and a quick sketch of the proof.
Although such sets are ``exotic'' in some ways (they are 
not locally connected),   they are 
fairly common among transcendental Julia sets, 
e.g., every finite-order, disjoint-type entire 
function has such a Julia set. Thus, although our
examples have novel metric properties (small dimension),
they are topologically ``ordinary'' among transcendental 
entire functions.

 We frequently use the  ``big-O'' notation:
if $f$ and $g$ are  non-negative real quantities
depending on common parameters, then  $f=O(g)$ means that
there is a constant $C>0$ so that $f\leq Cg$, independent
of the parameters. In particular, $f = O(1)$ means that $f$
is bounded, independent of the choice of parameters.
The notation $f \lesssim g$ means the same thing 
as $f=O(g)$ and $f \simeq g$ means that both 
$f=O(g) $ and $g=O(f)$ hold.


The first author would like
to thank the Stony Brook University for their support
and hospitality in Spring 2016.
The second author originally thought 
that Theorem \ref{dno thm}
would be a straightforward application of the QC folding 
method to approximating Stallard's examples, 
and he thanks the first author for  pointing out why
this was not the case.
Both authors would like to thank Lasse Rempe-Gillen  for 
a number of helpful insights and suggestions regarding the 
results in this paper. 
They are also extremely grateful to the anonymous 
referee whose thoughtful and  meticulous report  
greatly improved this paper.

\section{The Eremenko-Lyubich class 
 versus the  Speiser class } \label{EL class} 

In this section,  we sketch a  construction of 
Stallard's examples in $\classB$; this 
serves as a guide to  our 
construction in $\classS$, 
and  helps explain why the Speiser case
is more difficult. Our proof  is not Stallard's original
proof but has some similarities to it.

\begin{thm}  \label{EL thm}
 $ \inf\{ \dim(\Julia(f)): f \in \classB \} =1$.
\end{thm}

\begin{proof} 
By Baker's theorem mentioned earlier, the infimum is $\geq 1$, 
so we need only prove it is $\leq 1$.
Suppose $K> 1$,  let  $F(z)  = \exp(\exp(z-K))$  
and let 
$\Omega$ be the connected component of $F^{-1}
(\{z\in\complex:|z| > 2\})$ that lies inside the horizontal 
strip $S=\{x+iy\in\complex: |y| < \pi/2\}$. Note that $z \mapsto 
\exp(z-K)$ maps $\Omega$ conformally to the right 
half-plane $\{x+iy\in\complex: x > \log 2\}$.
Let $U =\{ z \in \Omega :\log 2 < \Real(\exp(z-K)) <   2 \}$, see the darker regions in Figure \ref{ELexample}.
Note that $U$ lies along the boundary of $\Omega$, and 
grows exponentially thin as we move to the right, so its 
total area is finite.
Also note that $\Omega$ lies to the right of the line
 $\{x+iy\in\complex: x =  K +\log 2\}$,
so if $K$ is large, $\Omega$ is disjoint from the closed unit disk 
around the origin, see Figure \ref{ELexample}.

\begin{figure}[htb]
\centerline{
\includegraphics[height=1.4in]{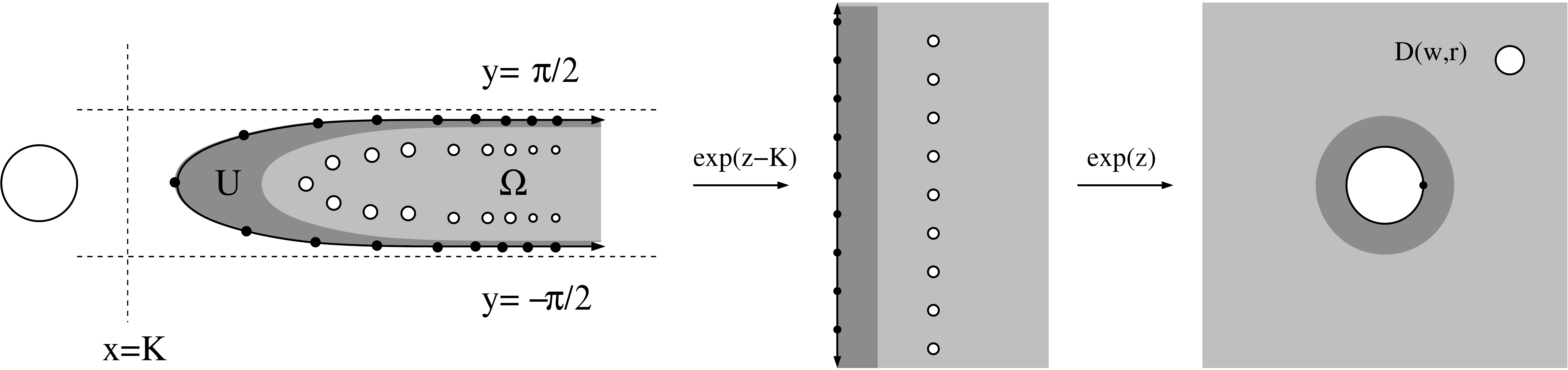}
}
\caption{ \label{ELexample}
     Preimages of a single disk under $F(z) = \exp(\exp(z-K))$.
}
\end{figure}

The pair $(\Omega, F)$ is a model in the sense of 
 \cite{MR3384512},  so by Theorem 1.1 of that paper, 
there is an $f \in \classB$ and a quasiconformal 
homeomorphism $\varphi$ of the plane  so that 
\begin{enumerate} 
\item  $|f| \leq 1$ on the complement of $\Omega$, 
\item $f \circ \varphi =F $ on $\Omega \setminus U$, 
\item $\varphi$ is conformal  except on $U$.  
\item $S(f) \subset \overline{\disk}$.
\end{enumerate} 
By condition (1), 
$f(\Omega^c) \subset \overline{\disk} \subset \Omega^c$, 
which implies $\Omega^c \subset \Fatou(f)$.
Thus, the Julia set of $f$ consists of points whose iterates 
stay in $\Omega$ forever.

Suppose  $D = D(w,r)$ with $\Real(w) \geq K$ and
$r \leq 4$ is a disk that hits  $\Julia(f)$.
Then  the   $f$-preimages of $D$ (i.e., the connected
components of $f^{-1}(D)$) correspond $1$-to-$1$
to  $F$-preimages via the map $\varphi$;
more precisely, $ f^{-1}(D) = \varphi(F^{-1}(D))$.
Set  $V=\{ z\in \Omega :4 < \Real(\exp(z-K))\}$.
It is easy to check that 
\[
\int_U\frac{dxdy}{|z-w|^2}\leq C<\infty
\]
for all $w\in V$ with a constant $C$ which is 
independent of $w$ and $r$. 
Because $\varphi$  is conformal off $U$, 
Lemma \ref{lem:bi-Lipschitz} of the current paper 
(applied to $A=U$ and $B=V$) implies that $\varphi$
is $L$-bi-Lipschitz on $V$ (where $L$ depends 
only on $C$ and the dilatation bound for $\varphi$)
and hence also on the
preimages of $D$ in question, at least if $K$ is large
enough.
 Hence, the diameter of a connected component of 
$f^{-1}(D)$ is at most $M$ times the diameter of
 the corresponding component of $F^{-1}(D)$, and
 $M$ is independent of the disk $D$ and the choice
 of the 
preimage component.

The $F$-preimages  of $D$ that are inside $\Omega$ 
can be understood in two steps: the inverses under 
$e^z$ consist of an infinite vertical ``stack'' of 
Jordan regions  $\{W_k\}$  each of diameter $O(r/|w|)$ and
  each containing a point of the form $\log|w| + 
i(2\pi  k +\arg(w))$, $k \in \integers$, see the 
center of Figure \ref{ELexample}.
The preimage of each $W_k$ is a region $U_k$ of diameter
\[
O\left( \frac { r}{|w|( \log |w| + 2 \pi k)} \right ),
\]
and hence,
\[
\sum_k\diam(U_k)^{1+\delta}\leq
\left(\frac {r}{|w|}\right)^{1+\delta}
\sum_k\frac 1{( \log |w|  + 2 \pi |k|)^{1+\delta} }
\leq\frac{C r^{1+\delta}}{\delta |w|^{1+\delta}
(\log |w|)^{\delta}}.
\] 
Using $|w|\geq K$, fixing $\delta>0$, and  taking
 $K$ large (depending on $M$ below), we get
\[
\sum \diam(U_k)^{1+\delta}
\leq  \frac {r^{1+\delta}}{2M^{1+\delta}},
\]
where the sum is over all preimages lying in $\Omega$.
Thus, 
\[
\sum \diam(f^{-1} (D) )^{1+\delta} 
\leq M^{1+\delta} \sum_k \diam(U_k)^{1+\delta}
\leq r^{1+\delta}/2 .
\]
This is  (\ref{main est}) with $\epsilon = 1/2$. We can also
cover $\Omega$ by disks 
$\{B_n\}=\{ D(n, 4)\}_{n=K}^\infty$, and 
summing over all preimages of all  these disks gives the sum 
\[
\sum_n\sum_k\diam(f^{-1}(B_n))^{1+\delta}
=O\left(\sum_n\frac 1{\delta |n|^{1+\delta}
  (\log K)^{\delta}}\right)<\infty,
\]
i.e.,  (\ref{initial sum}). 
Thus, the Julia set of $f$ has dimension $\leq 1+ \delta$, if $K$ is 
large enough.
\end{proof}



Our proof of   Theorem \ref{dno thm} is  significantly 
longer than the proof of  Theorem \ref{EL thm}  given above. 
Why?
The approximation theorem for the Eremenko-Lyubich 
class from  \cite{MR3384512}  has an analog for 
the Speiser class   \cite{MR3653246},
 but this result  does not satisfy 
the crucial condition (1) above; if  we  attempt to 
approximate the function $F$ on the tract $\Omega$ 
by a Speiser class function $f$, we may be forced to make 
$f$ large at points outside $\Omega$ and this 
 introduces many ``extra''  $f$-preimages  that are 
not associated to any $F$-preimages, and we have no way 
to control them.  
Instead, we have to replace the tract 
$\Omega$ above by a more complicated region, and 
replace  the approximation result 
from \cite{MR3384512} by an application of
the quasiconformal folding construction from
\cite{Bishop-classS}.

The folding construction  starts with
a model function $F$ that is holomorphic on each  
connected component
of the complement of an infinite, connected  planar graph $T$,
although it may be discontinuous across the edges of $T$.
(In general, the graph need not be a tree, but we shall 
still denote it by ``$T$'' instead of ``$G$''; in this 
paper the graph is an infinite
 tree except for a single closed loop).
 The function $F$ is  modified 
in a  certain neighborhood of  $T$, denoted $T(r)$ (see
next section),   to give a quasiregular 
function $g$ on the plane that equals $F$ outside $T(r)$,
 and then $g$ is converted to an entire 
function $f = g \circ \varphi^{-1}$ 
with a quasiconformal homeomorphism 
$\varphi$ given by the measurable 
Riemann mapping theorem.

Suppose   $\overline{T(2r)} \subset \Fatou(f)$.
Then the Julia set has a neighborhood $W$  disjoint from 
$T(2r)$ and  $f = g\circ \varphi^{-1} =  F \circ \varphi^{-1}$
on $W$.
 Moreover, we will show that 
$\varphi$ is bi-Lipschitz  outside $T(2r)$ 
(Lemma \ref{correction map bound}).
 This means 
that for a disk $D$,  connected components of 
$f^{-1}(D)$  can be associated via $\varphi$ to   
components of $F^{-1}(D)$ that have comparable 
size, and we will have good control of these
components by construction.
Thus, if we can build a model $F$  satisfying  
(\ref{initial sum})  and (\ref{main est}) with a small 
enough constant, then   we will get a Speiser class entire 
function   $f$ that also satisfies these conditions 
(with different constants) and this will  imply
Theorem \ref{dno thm}.

The difficult part of this plan is the claim that 
 $ \overline{T(2r)} \subset \Fatou(f)$.
In our examples, the Fatou set will contain a 
large  closed disk around the origin,
 and hence it will be enough to 
show that $f$ maps $T(2r)$ into this disk, i.e., that 
$f$  is bounded on $T(2r)$. This will reduce 
to showing that $F$ is bounded on $T(2r)$.
 The folding construction associates
two   positive weights, called the $\tau$-lengths,
 to each edge of the 
planar graph $T$  and it requires that the $\tau$-lengths
of all edges in $T$ 
 are uniformly bounded away from zero. 
Moreover, showing $F$ is bounded on $T(r)$
reduces to proving the  $\tau$-lengths   are 
uniformly bounded above.
Essentially all the work in this paper is devoted to 
building  a pair $(F,T)$  so that the function  $F$
satisfies
(\ref{initial sum})  and (\ref{main est})   and 
 $\tau$-lengths for the graph $T$ are
uniformly bounded  above and away from zero.
Once we have done this, 
the proof will  proceed 
as  in the Eremenko-Lyubich case described earlier. 

In Section \ref{review sec} we review 
quasiconformal folding; 
in Sections \ref{Step 1 tree}-\ref{adding vertices}
 we define 
the graph $T$ (and this determines $F$); in Sections 
\ref{review metric}-\ref{comparable tau}
 we give estimates of 
hyperbolic distances that culminate in the desired 
 upper and lower 
bounds for $\tau$-length;  
in Sections \ref{T(r) finite area}-\ref{bi-Lip sec}
we estimate the correction map $\varphi$; and in 
Sections \ref{prove 1.1}-\ref{topology} we finish the 
proof of Theorem \ref{dno thm} and discuss the topology of 
the Julia set.


\section{Quasiconformal  folding } \label{review sec} 

Quasiconformal folding is a technique for  constructing
 transcendental entire functions with  good 
control on both the singular values and the geometric 
behavior of the function.
Here we will review some definitions and results from
\cite{Bishop-classS}; consult that paper for further 
details and proofs.

Let $f$ be a transcendental entire function with no finite
 asymptotic values and with only two critical values $\pm 1$. 
Then $T=f^{-1}([-1,1])$ is an unbounded, infinite tree and all
 components of $\Omega=\complex\setminus T$ are unbounded,
 simply connected domains. We can 
choose a map $\tau$ which is
 conformal from each component of $\Omega$ onto the right
 half-plane $\rhp=\{x+iy:x>0\}$, and  so that
 $f=\cosh\circ \, \tau$ on $\Omega$. 

The idea of quasiconformal folding is to reverse this procedure. 
We start with an unbounded, infinite, locally finite 
tree $T$ which fulfills certain mild geometric conditions. 
Furthermore, for each component $\Omega_j$ of 
$\Omega=\complex\setminus T$ let $\tau_j$ map $\Omega_j$
 conformally onto $\rhp$ and let $\tau:\complex\setminus T\to\complex$ 
be given by $\tau=\tau_j$ on $\Omega_j$.
 The map $g$ given by $g=\cosh\circ \, \tau$ is then holomorphic 
off $T$. In general, $g$ is  not 
continuous across $T$, but  it is possible to change $g$ 
in  a  neighborhood 
\[
T(r)=\bigcup_{\text{edges }e\text{ of }T}\{z\in\complex~:~\dist(z,e)<r\cdot\diam(e)\}
\]
of $T$ so that it becomes continuous on the whole plane 
(the new function is quasi-regular on the plane).
See Figures \ref{Nbhd2} and \ref{Nbhd3}.


\begin{figure}[htb]
\centerline{
\includegraphics[height=.75in]{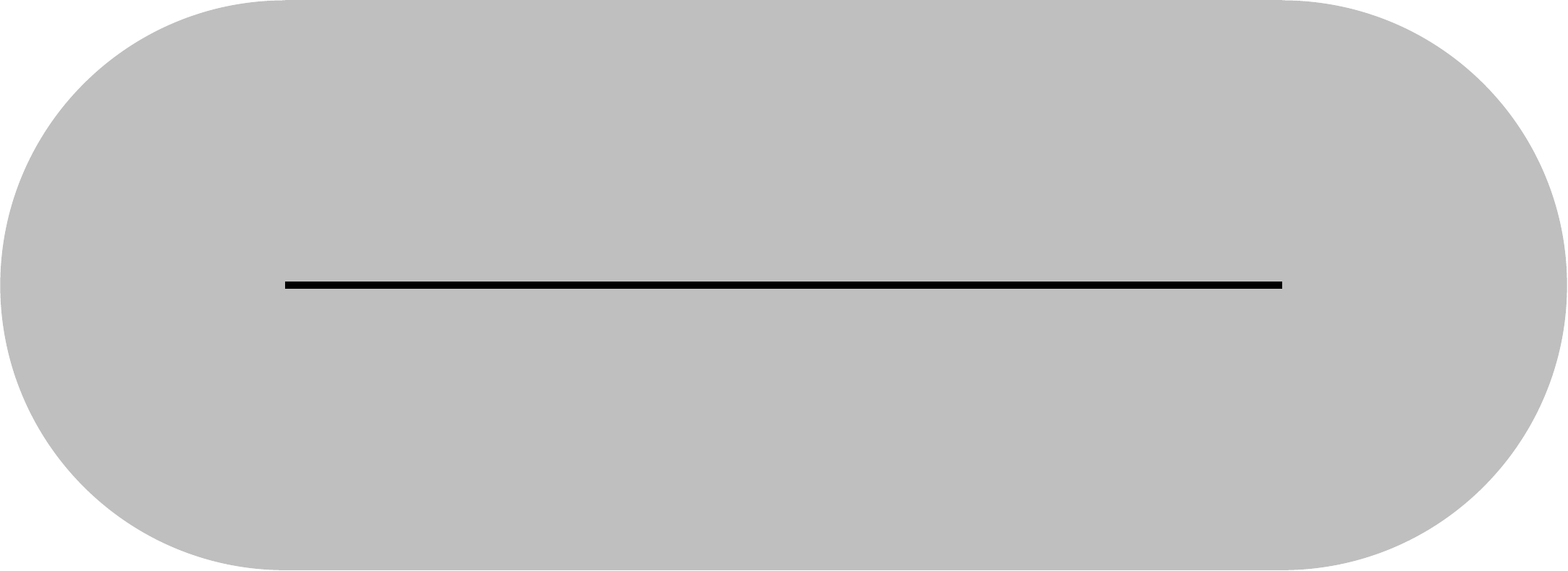}
}
\caption{ \label{Nbhd2}
    The  $r$-neighbourhood of an arc $\gamma$ is the set given by
    $ \gamma(r) = 
    \{ z: \dist(z, \gamma) < r \cdot \diam(\gamma) \}$ .
}
\end{figure}

\begin{figure}[htb]
\centerline{
\includegraphics[height=2.0in]{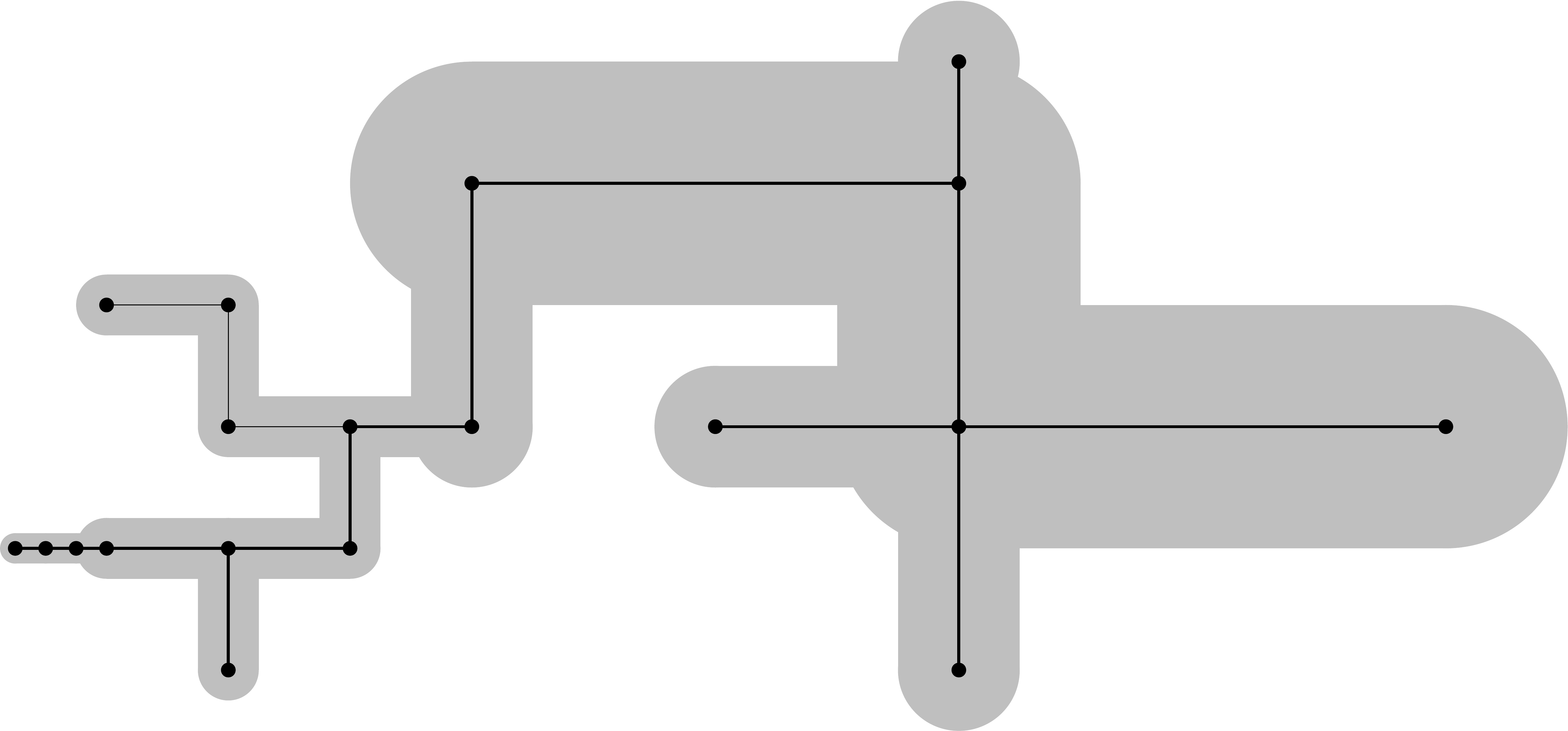}
}
\caption{\label{Nbhd3}
The neighbourhood of a graph is the union of neighborhoods
 of the individual edges.
}
\end{figure}

Every edge $e$ of $T$  has two ``sides'' that are each 
mapped by $\tau$  to an interval on $ i \reals =\partial\rhp$. 
The $\tau$-size of $e$ is the minimum of the lengths 
of the two  image intervals. 
For each component of $\Omega = \complex \setminus T$, the 
sides of $T$ on $\partial \Omega$ correspond, via $\tau$, to a 
locally finite  partition of $ \partial \rhp$ into 
intervals. The folding theorem requires that there is
a positive lower bound for the lengths of these intervals, 
usually taken to be $\pi$ (in many cases, we can 
replace $\tau$ by a positive multiple of itself, so the 
precise  lower bound is not critical). The folding theorem
also  requires that 
adjacent  intervals in these partitions have comparable lengths. 
This  follows from 
\emph{bounded geometry}, i.e.,  a planar graph has bounded 
geometry if (this is a slight strengthening of the conditions 
originally given in \cite{Bishop-classS}):
\begin{enumerate}[label=(\arabic*)]
\item the edges of $T$ are $C^2$ arcs with uniform bounds;
\item the  union of edges meeting at a vertex are a bi-Lipschitz 
	image of a star $\{z: z^n \in [0,r]\}$, 
	i.e., a  union of $n$ equally spaced, equal length 
       radial  segments  meeting at $0$, with $n$ uniformly 
       bounded;
\item for non-adjacent edges $e$ and $f$, $\diam(e)/\dist(e,f)$ 
is uniformly bounded.
\end{enumerate}
Note that (2) implies adjacent edges have comparable lengths and 
that they meet at an angle bounded uniformly away from zero.



\begin{thm}\label{thm:folding_only_R}
Suppose that $T$ has bounded geometry and every edge
 has $\tau$-size $\geq\pi$. Then there are $r >0, K>1$,
 an entire function $f$ and a $K$-quasiconformal map
 $\varphi$ so that $g = f\circ\varphi=\exp \circ\tau$
 off $T(r)$. The constants $r$ and  $K$ only
 depend on the bounded geometry constants of $T$. 
The only critical values of $f$ are $\pm 1$ and $f$
 has no finite asymptotic values.
\end{thm}


\begin{figure}[htb]
\centerline{
\includegraphics[height=1.5in]{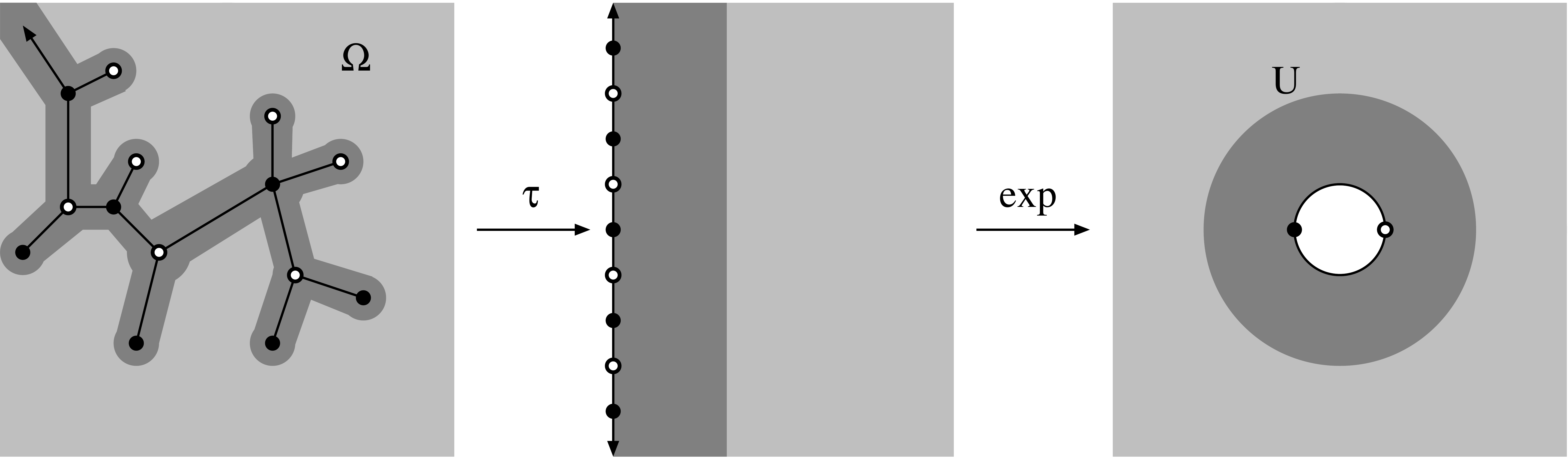}
}
\caption{\label{Pcovers6}
The quasiregular function $g$ equals the holomorphic
function $\exp \circ \tau$ away from the tree (light gray)
and is defined near the tree  (dark gray) to be 
continuous and have bounded dilatation. The dilatation 
of the quasiconformal correction map is supported in
this neighborhood of the tree.
}
\end{figure}

This is Theorem 1.1 of  \cite{Bishop-classS}, but we
have modified the statement given there slightly. 
Here we have used ``$\exp$'' in place of 
``$\cosh$''; this is a harmless change, as 
explained in Section 7 of \cite{MR3653246}. 
We can deduce a little more if we impose another 
geometric restriction on our bounded geometry tree:

\begin{lem}\label{lem:T(r)_bound 1}
Suppose that $T $ has bounded geometry and that $r$ is 
	as in Theorem \ref{thm:folding_only_R}.
Assume that for
 every edge $e$ of $T$, the neighbourhood
$T_e(4r)=\{z\in\complex~:~\dist(z,e)<4r\cdot\diam(e)\}$
only intersects edges whose length is comparable
to the length of $e$ with constant $M$.
Then there exists an $\var1epsilon>0$, only depending 
on  $r$, the bounded geometry constants of $T$,  and on $M$,
so that for every point $z\in T(2r)$ there exists 
some edge $e'$ so that the harmonic measure of $e'$ 
with respect to $z$ is at least $2\var1epsilon$.
\end{lem}

\begin{proof}
Let $\Omega$ be one of the components of the complement 
of the tree, let $e$ be an edge
on the boundary of $\Omega$, and let $z\in T_e(2r)\cap\Omega$.
 Let $e'$ be the edge on the boundary of $\Omega$ which
 is closest to $z$. Then $e'$ intersects $T_e(4r)$ and hence, 
by assumption, 
it has diameter comparable to the diameter of $e$.
Let $d = \dist(z,e') = O(\diam(e'))$ and let $w \in e'$ be a
point closest to $z$.
See the left side of  Figure \ref{HMproof}.
 By bounded geometry, there is a radius $r$ 
that is comparable to $\diam(e')$ so that the disk $D(w,r)$ only
hits $e'$ or edges adjacent to $e'$, of which there are only a
bounded number.  By the Beurling projection theorem
(e.g., Theorem II.9.2 or Exercise II.10 of \cite{MR2450237})
the part of $\partial \Omega$ in $D(w,r)$ has harmonic
measure with respect to $z$  that is uniformly bounded
away from zero.  Since only a  uniformly bounded number of edges
hit this disk, one of them must have harmonic measure 
uniformly bounded away from zero, as desired.
\end{proof} 

\begin{figure}[htb]
\centerline{
\includegraphics[height=2.0in]{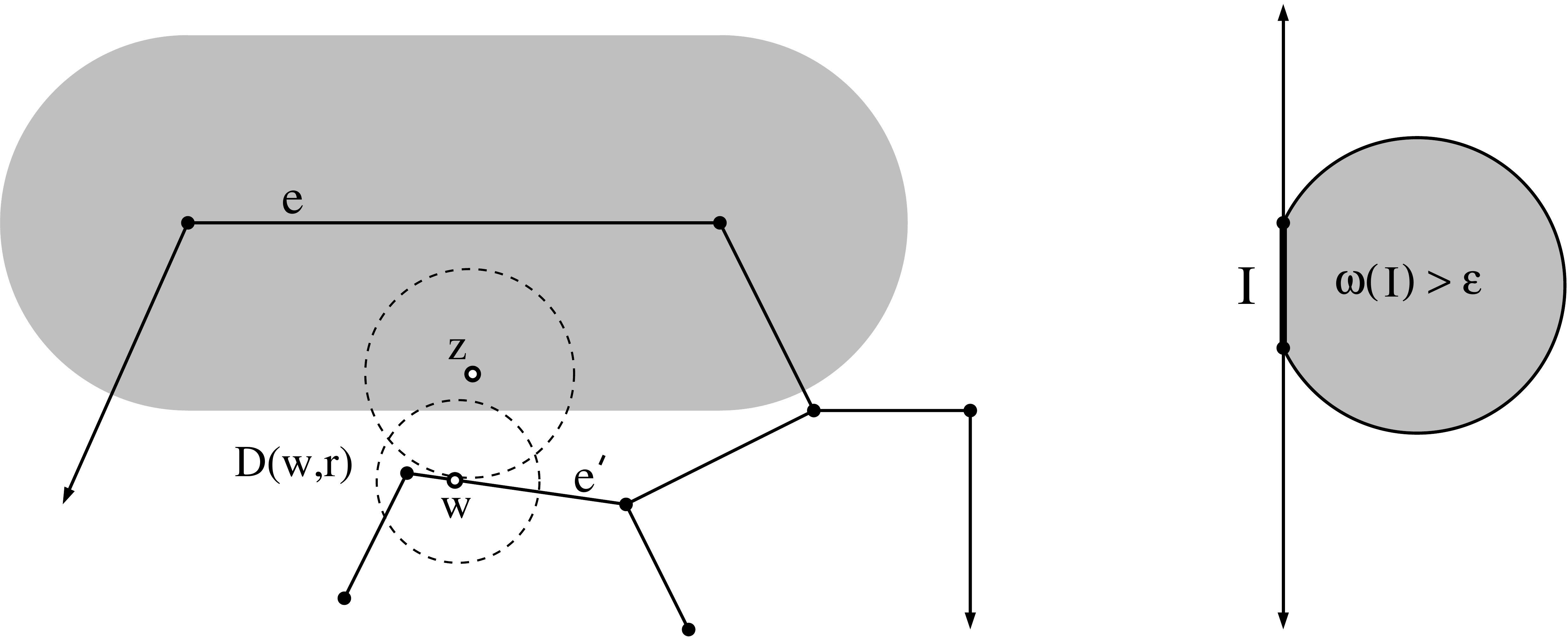}
}
\caption{\label{HMproof}
The proofs of Lemmas \ref{lem:T(r)_bound 1} and 
\ref{lem:T(r)_bound}. The left 
side shows why there is an  individual edge with 
large harmonic measure. The right side shows the 
region of large harmonic measure  for an interval
on the boundary of a half-plane.
}
\end{figure}

\begin{lem} \label{lem:T(r)_bound}
Suppose that $T$ has bounded geometry and that $r$ is 
as in Theorem \ref{thm:folding_only_R}
and Lemma \ref{lem:T(r)_bound 1}.
Assume $T$ satisfies  the assumption in Lemma 
\ref{lem:T(r)_bound 1}.   Suppose also that 
  the $\tau$-lengths of all edges
are bounded above. Then there exists $M< \infty$
 so that $\tau(T(2r))\subset\{x+iy\in\complex: 0<x<M\}$.
In other words, $|g| \leq e^M$ on $T(2r)$.
\end{lem} 

\begin{proof} 
Let $z\in T(2r)$ and  
let $e'$ be the tree edge given by Lemma \ref{lem:T(r)_bound 1}.
Since $e'$ has harmonic measure at least $2 \var1epsilon$ 
with respect to $z$ in $\Omega$ (the complementary component 
of $T$ that contains $z$), one of the two sides of 
$e'$ has harmonic measure at least $\epsilon$ with respect
to $z$ in $\Omega$. Call this side $s'$.
By the conformal invariance of harmonic measure,
  $\omega(\tau(z),\tau(s'),\rhp)
\geq\var1epsilon$.
In the half-plane, the set
of points at which a boundary interval $I$ has harmonic measure
$\geq \epsilon$ is the intersection of a disk of radius 
$\simeq |I|/\epsilon$  with the half-plane. 
See the right side of Figure \ref{HMproof}.
Therefore $\tau(z)$ lies in this region.
Since the $\tau$-lengths of all sides 
are bounded above,  this implies $\tau(z)$ is 
within a bounded distance of the imaginary axis, i.e., 
$\tau(T(2r))$ is  contained in a vertical strip of 
uniformly bounded width.
\end{proof}

If one wants a 
function that has  finite asymptotic values, 
or that has  critical points with high order, then the 
folding construction described above
needs to be changed a bit. The tree $T$
is replaced by an unbounded, connected, locally 
finite graph. Each of the components $\Omega_j$ of 
$\Omega=\complex\setminus T$ is one of  three 
types (D, L or R), and each is mapped to a 
corresponding ``standard'' domain (disk, left half-plane 
or right half-plane)  by a conformal map $\tau$.  Each 
standard domain is then mapped by  an associated 
holomorphic map $\sigma$.
More precisely, the cases are:
\newline$\bullet$ D-component: 
Here, $\Omega_j$ is a bounded domain 
whose boundary is a Jordan curve which consists of $d$ edges.
D-components are quasiconformally mapped to $\disk$ so that 
the $d$  vertices of the component map to  $d$-th roots of unity.
This is followed by the map 
$\sigma(z) = z^d$.
This gives a critical value at $0$.
\newline$\bullet$ L-component: $\Omega_j$ is an
unbounded Jordan domain 
which is quasiconformally  mapped to the left half-plane $\lhp$.
This is followed by $\sigma(z) = \exp(z)$, 
which maps $\lhp$  to $\disk$ and gives the asymptotic value $0$. 
\newline$\bullet$ R-component: These are the components which were used 
in the first theorem. Here, $\Omega_j$ is unbounded but not 
necessarily a Jordan domain. Each $\Omega_j$ is 
mapped onto $\rhp$ as before,
and $\sigma(z)$ is   $\exp(z)$. 

\begin{thm}\label{thm:folding}
Let $T$ be a bounded geometry graph and suppose $\tau$ is
 conformal from each complementary component of $T$ to  the 
	corresponding
 standard domain (i.e. $\disk$, $\lhp$ or $\rhp$). Assume that
\begin{itemize}
\item D and L-components only share edges with R-components;
\item on D-components with $n$ edges, $\tau$ maps the 
vertices to $n^{th}$ roots of unity;
\item on L-components, $\tau$ maps sides  to intervals
of the form $[2 \pi k i, 2 \pi (k+1)i]$;
\item on R-components, the $\tau$-sizes of all edges are $\geq 2\pi$.
\end{itemize}
Then there exist $r>0$, $K>1$, an entire function $f$ and a $K$-quasiconformal 
map $\varphi$ of the plane so that $f\circ\varphi=\sigma\circ\tau$ 
off $T(r)$. The constants $r$ and $K$ only depend on the 
bounded geometry constants of $T$. Also, $S(f) =\{ \pm 1\}$;
plus $\{0\}$ if any D or L-components occur. 
\end{thm}

As before, 
if the $\tau$-sizes of the edges are bounded above
by some constant $C$, we get that the image 
of the part of $T(2r)$ that  lies in the
R-components is mapped into a vertical strip whose
width only depends on the bounded geometry constants and on $C$.
The construction on the D and L-components can be modified 
to give singular values other than $0$, but we 
will not need this variation here.

\section{The ``trunk'' of the  tree} \label{Step 1 tree} 

The graph to which we apply the folding theorem  will 
be  built in two steps. In this section, we present the first step:
we construct a graph, called the ``trunk'', 
 that divides the plane into 
$4N+1$ connected components ($1$ D-component, $2N$ 
L-components, $2N$ R-components),
where $N$ is a positive integer larger than $2$.
In later sections, we will add ``branches''
(line segments) to the trunk, in order to get 
the ``bounded $\tau$-length'' condition.

The first of the $4N+1$ components is the disk
$D(0, r_N)$ where
\[
r_N=\frac 1{2 \sin(\pi/(2N))}+N-1 
 =  N\left(1+\frac 1{\pi}\right)-1 
+O\left(\frac 1N\right).
\] 
This is the  D-component and will be denoted $D_0$.
Let $\theta=\frac\pi N$ and define
\[
z_k=(r_N-N+1)\exp\left({i\left(\frac\theta 2+k\theta\right)}\right)
\]
for $0\leq k\leq  2N-1$.
With $S_0=\{x+iy~:~x>0,|y|<1/2\}$ and 
$S_k=\exp(i k\theta) S_0$, the point $z_k$
is the unique non-zero intersection point
between $\partial S_k$ and $\partial S_{k+1}$, see Figure \ref{BuildTree}.
 The L-components $L_k$ are then given by
\begin{align*}
L_k&=\left\{z\in\complex~:~|z-z_k|>N,\ k\theta<\arg (z-z_k)<(k+1)\theta\right\}\\
&=\left\{z_k+re^{i\phi}~:~r>N,\ k\theta<\phi<(k+1)\theta\right\}.
\end{align*}
Note that these are truncated sectors that are disjoint, 
are  unit distance apart, and are unit distance from the
D-component, see Figure \ref{BuildTree}.  The 
vertex of the $k$-th sector is $z_k$ (not the origin).
One L-component is shown in light gray in Figure \ref{BuildTree}.

Finally, we construct the R-components.
 The complement $\Omega$ of the union of  the D-component 
and L-components can be split into $2N$ 
congruent connected components  $\{R_k\}$
by the  radial segments 
\[
\left\{ z\in\complex:r_N\leq |z|    \leq r_N+1,\ \arg(z)  
= \frac\theta 2+k\theta \right\}.
\]
Denote by $R_0$ the component in the right half-plane
which is symmetric with respect to the real axis.
The component $R_k$ for $1\leq k<2N-1$ is then the
component which is just $R_0$ rotated by $\exp(ik\theta)$.
Note that $R_k$ has a boundary arc on the circle of
radius $r_N$ around the origin, two radial boundary
arcs on the segments just defined, two circular arc
boundary edges on the circles of radius $N$ around
the points $z_{k-1}$ and $z_k$ (using 
$z_{-1}=z_{2N-1}$), and two sides that are infinite,
parallel rays distance one apart,  
see Figure \ref{BuildTree}. 
Figure \ref{Frame5} shows the overall structure of 
the trunk graph.

Note that the five finite sides of the $R_k$ all have lengths
comparable to $1$ and all the angles are approximately 
$90^\circ$, even as $N \to \infty$. These facts will
help prove  that our graph has  bounded geometry 
with constants independent of $N$.

\begin{figure}[htb]
\centerline{
\includegraphics[height=2.5in]{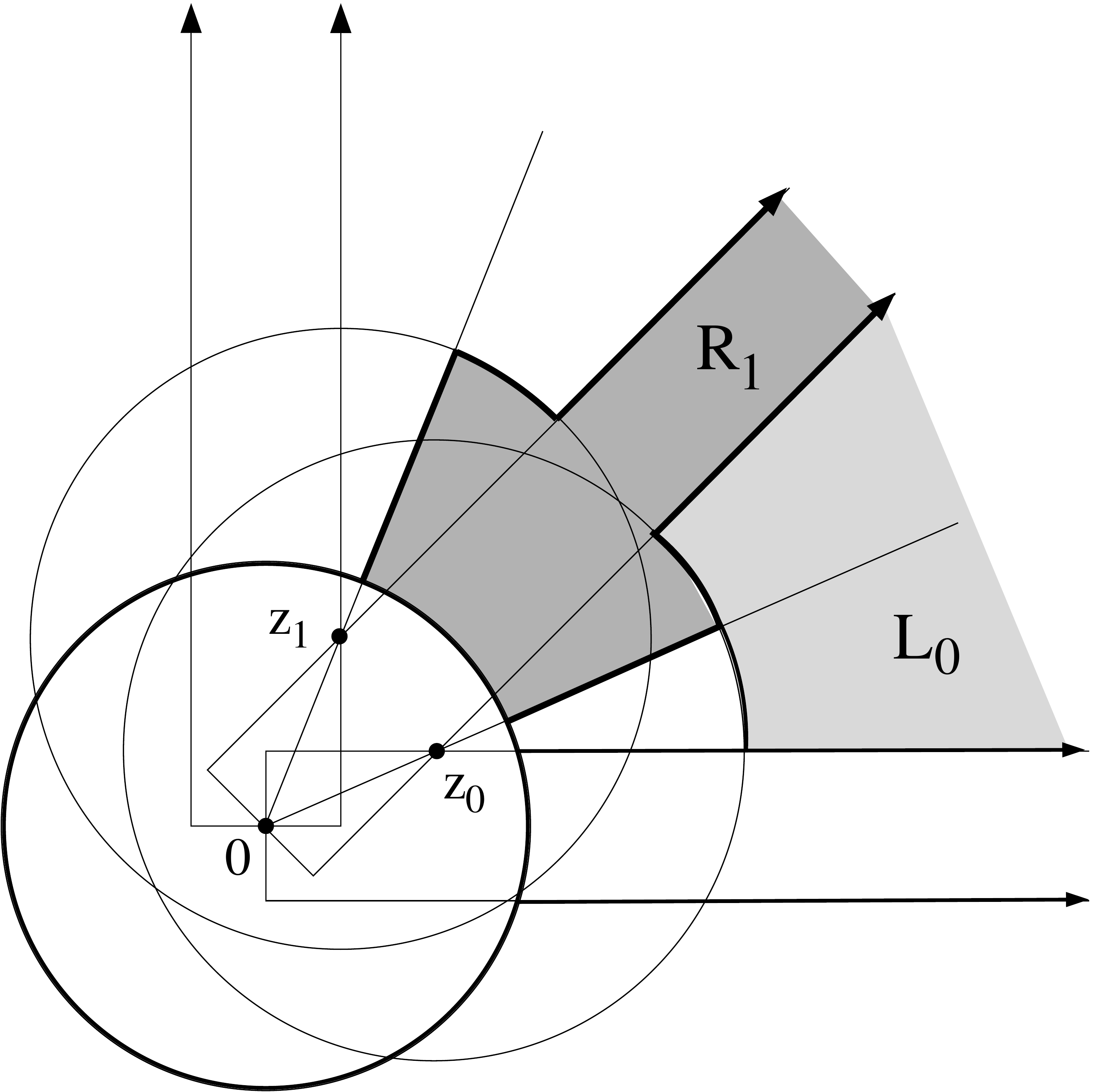}
}
\caption{ \label{BuildTree}
   Consider  $2N=8$ unit width half-strips rotated
   evenly around the circle. The D-component is 
   a disk of radius $r_N$. The L-components  are 
   truncated sectors with vertices at the points 
   where the half-strips intersect (one is shown 
   in light gray). The  remainder of the plane 
   is divided into the R-components 
   (one is shown in darker gray).
}
\end{figure}

\begin{figure}[htb]
\centerline{
\includegraphics[height=2.5in]{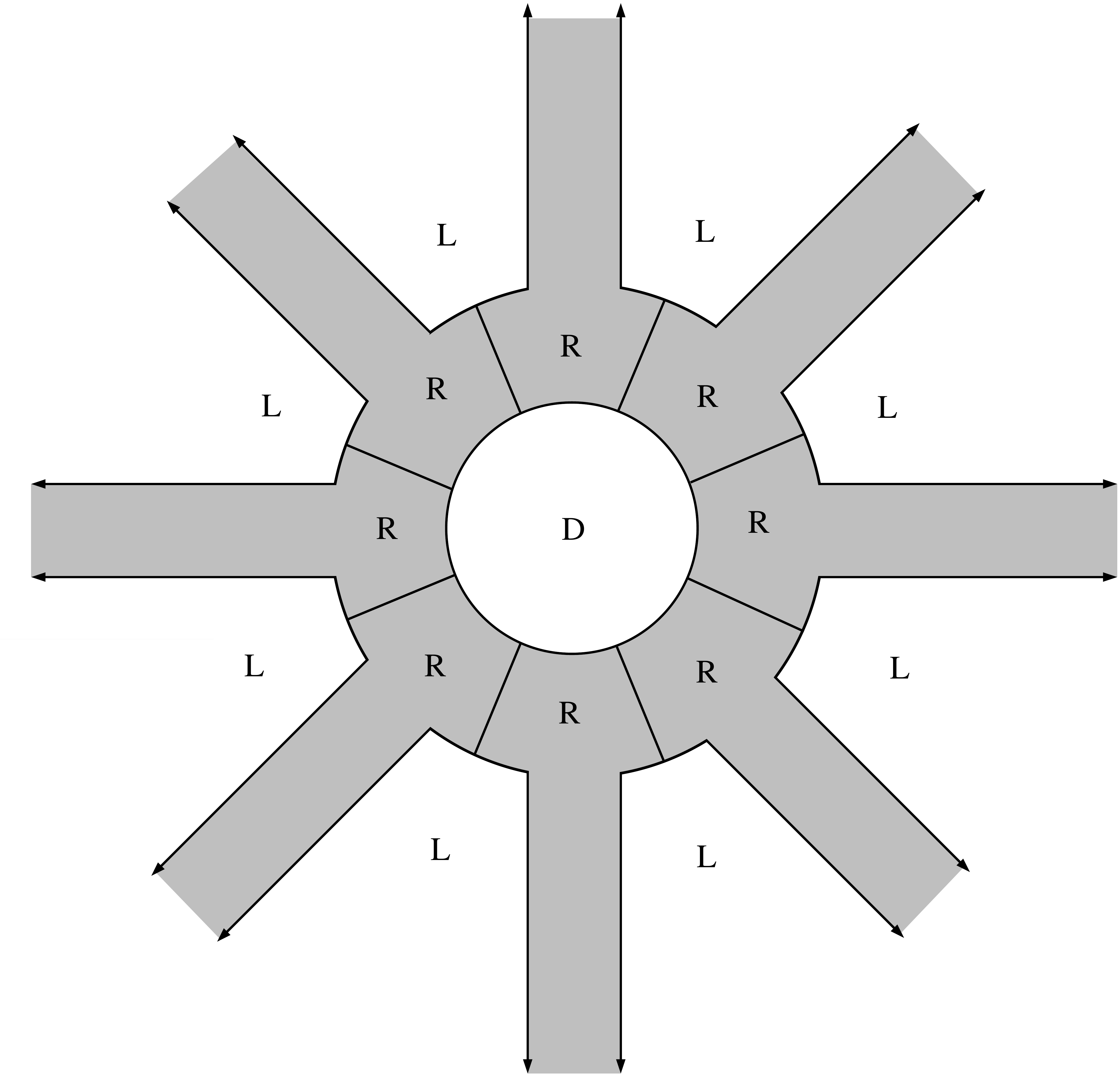}
}
\caption{ \label{Frame5}
   The ``trunk'' consisting of 
   one D-component,  $2N$  L-components that 
   are truncated sectors, and an equal number of 
   R-components that are essentially 
   half-strips.
}
\end{figure}

\section{Conformal partition of the L-components}
\label{L component sec} 

A conformal partition of an unbounded Jordan domain
$\Omega$  is a collection of points $S$ 
on $\partial \Omega$ so that $\tau(S) \subset \partial H$ are 
evenly spaced, where $H$ is a half-plane and 
$\tau:\Omega \to H$ is 
a conformal map  taking $\infty$ to $\infty$. 
For the L-components defined in Section 
\ref{Step 1 tree}  such a partition can 
be  explicitly computed, and we record the
computation in this section.

\begin{figure}[htb]
\centerline{
\includegraphics[height=1.25in]{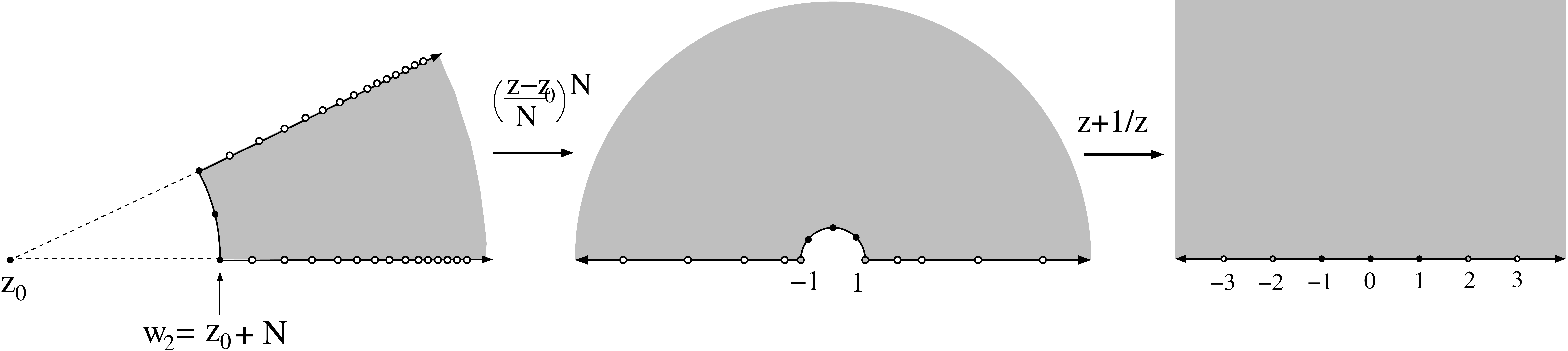}
}
\caption{ \label{TauForLcomp}
   The L-components can be mapped to a half-plane by an 
   explicit map (a rescaling and a power followed by the Joukowsky map),
 and thus the conformal partition is 
   also given by explicit points.
}
\end{figure}

The L-component $L_0$ is conformally 
 mapped to the upper half-plane   by 
\[
z \mapsto\left(\frac 1 N (z-z_0)\right)^N + \frac 1{\left(\frac 1 N (z-z_0)\right)^N},
\]
the composition of a linear rescaling, a power,
 and the Joukowsky map $ z \mapsto z + 1/z$,
 which conformally maps  $\uhp\setminus \overline{\disk}$
 to $\uhp$, see Figure \ref{TauForLcomp}. 
The inverse map is given by 
\[
w \mapsto z_0+ N\left(\frac w2 + 
\sqrt{\left(\frac w2\right)^2-1}\right)^{1/N}
=z_0+ N\left(\frac 12\left(w + 
\sqrt{w^2-4}\right)\right)^{1/N}.
\]
Let $w_n$ for $n \in \integers$ be the points on $\partial L_0$ that correspond under this map to the points of $\integers$ on the boundary of the upper half-plane. For $n\geq 2$ we define $t_n$ by the equation
\[
w_n = z_0+N + t_n.
\]
For $n \geq 2$, let $\Delta_n =t_{n+1}-t_n$.
For convenience, we set $\alpha = 1/N$ and define 
$\phi_N(n) = n^{2 \sqrt{\alpha}} + 2 \log n$. Note that 
for fixed $n \geq 1$ this decreases as $N$ increases.

\begin{lem} \label{adjacent estimates}
Suppose notation is as above. Then
\begin{equation}\label{eqn:t_n}
      t_n =   N(n^{  \alpha } - 1)  + O\left(n^{\alpha -2}\right)  
\end{equation}
\begin{equation}\label{eqn:t_n 2}
      t_n  \leq N(n^{\alpha}-1)\leq N n^\alpha\quad\text{and}\quad   t_n\leq N(n^{\alpha}-1)\leq   \phi_N(n) ,
\end{equation}
\begin{equation}\label{eqn:Delta_n} 
      \Delta_n =  n^{  \alpha -1} + 
              \frac{\alpha-1}{2} n^{\alpha-2} +  O\left(n^{ \alpha -3}\right),
\end{equation} 
\begin{equation}\label{eqn:Delta diff}
   \Delta_{n}-\Delta_{n+1} 
    = (1-\alpha)n^{\alpha-2} +  O( n^{\alpha-3} ) 
    = \Delta_n \left ( \frac {1-\alpha}{n} + O(n^{-2})\right ) .
\end{equation} 
The big-$O$ estimates hold as $n \nearrow \infty$ 
and the constants in these inequalities 
do not depend on $N$.
\end{lem} 

\begin{proof}
By definition (recall $\alpha = 1/N$), if $n\geq 2$, then
\begin{align*}
w_n &= z_0+ N \left(\frac 12 \left(n + \sqrt{n^2-4}\right)\right)^{\alpha} 
= z_0+ N n^{\alpha} \left(\frac 12 
    \left( 1+ \sqrt{1-4n^{-2}}\right)\right)^{\alpha} \\
&= z_0+ N n^{\alpha} \left(1 + O\left(n^{-2}\right)\right)^{\alpha} 
= z_0+ N n^{\alpha} \left(1 + O\left(\frac 1N n^{-2}\right)\right) \\
&= z_0+ N n^{\alpha} + O\left(n^{\alpha-2}\right). 
\end{align*}
The constants in the big-O's hold as $n\nearrow \infty$
	and they do not depend on $N$
since $\alpha=1/N\leq 1/2<1$
 (in fact, one can easily check that 
the constant $4$ works). 
The  equality in \eqref{eqn:t_n} is   immediate. 
It is clear that the ``O'' term in the last 
line above is negative, so 
$t_n \leq N(n^\alpha-1).$ This gives  the first part of 
(\ref{eqn:t_n 2}).
To prove the second part of (\ref{eqn:t_n 2}), we consider 
two cases depending on whether $n$ is less than or greater 
than $N^{\sqrt{N}}$.
For  $1 \leq n \leq N^{\sqrt{N}}$ we have 
\[
t_n \leq N( n^\alpha -1) = N (\exp( \alpha \log n ) -1) 
  \leq  N(1+ 2\alpha \log n -1) = 2\log n.
 \]
Here we have used  the facts from calculus that 
$\alpha \log n \leq \sqrt{N} (\log N )/N \leq 1$ for $N \geq 3$
and  $e^x \leq 1 +  2 x$ for $0 \leq x \leq 1$. 
For $n > N^{\sqrt{N}}$,  since $\alpha \leq \sqrt{\alpha}$, we get
\[
t_n \leq N n^\alpha =  N n^{-\sqrt{\alpha}}
        n^{\alpha + \sqrt{\alpha}}
 \leq N  \cdot N^{-\sqrt{N}/\sqrt{N}}  n^{2\sqrt{\alpha} } 
= n^{2 \sqrt{\alpha}}.
\]
Since we have upper bounds on two disjoint intervals that cover
all $ n \geq 1$, we know  $t_n$ is less than the sum of these 
two estimates.
This is the second part of (\ref{eqn:t_n 2}).

To compute the gaps $\Delta_n$ between the points $w_n$, 
we can omit the additive factor $z_0$ and consider the function
\[
f(w) = N\left(\frac 12\left(w + \sqrt{w^2 -4}\right)\right)^{\alpha}.
\]
A calculus exercise shows  (recall $\alpha N=1$):
\begin{align*}
f'(w)=&\left(\frac 12\left(w + \sqrt{w^2 -4}\right)\right)^{\alpha -1}\cdot\frac 12\left(1+w\left(w^2-4\right)^{-1/2}\right)\\
&=w^{\alpha -1}\left(\frac 12 +\frac 12\sqrt{1-4w^{-2}}\right)^{\alpha-1}\cdot\frac 12\left(1+w\left(w^2-4\right)^{-1/2}\right)\\
&=w^{\alpha -1}\left(1 + O\left(\frac 1{w^2}\right)\right)^{\alpha-1}\left(1+O\left(\frac 1{w^2}\right)\right)\\
&=w^{\alpha -1} + O\left( w^{\alpha-3}\right).
\end{align*} 
tends to zero as $w \to  \infty$. A similar computation 
shows that
\[
f''(w)=(\alpha-1)w^{\alpha-2}+O(w^{\alpha-4}).
\]
Using Taylor series, we see that
\begin{eqnarray*}
 \Delta_n = f(n+1) - f(n) 
&=& \int_n^{n+1} f'(t) dt,  \\
&=& \int_n^{n+1} \left[ n^{\alpha-1} + (\alpha-1)(t-n) n^{\alpha-2}
         +O(n^{\alpha -3}) \right] dt,  \\
&=& n^{\alpha-1} + \frac 12 (\alpha-1)n^{\alpha-2}
         +O\left(n^{\alpha -3}\right) ,   
\end{eqnarray*} 
which is (\ref{eqn:Delta_n}).
Finally,  using the mean value theorem 
 gives (\ref{eqn:Delta diff}):
\begin{eqnarray*} 
 \Delta_n - \Delta_{n+1} 
&=&  \left[n^{\alpha-1} +\frac{\alpha-1}{2} n^{\alpha-2} + O\left(n^{\alpha-3}\right)\right] \\
&& \qquad \qquad 
-\left[(n+1)^{\alpha-1} + \frac{\alpha-1}{2} (n+1)^{\alpha-2} + O\left(n^{\alpha-3}\right)\right] \\
&=&  \left[n^{\alpha-1}-  (n+1)^{\alpha-1}\right]  
   +\frac 12(\alpha-1)\left[n^{\alpha-2} -(n+1)^{\alpha-2}\right] 
+ O\left(n^{\alpha-3}\right) \\
&=&  (1-\alpha) n^{\alpha-2}  + O\left(n^{\alpha-3}\right) .
\qedhere
\end{eqnarray*}
\end{proof}

Since $\{\phi_N\}$ is decreasing in $N$,  $\{t_n\}$
has an  upper bound independent of $N$, and this  
bound improves if $N$ is large, e.g.,  if $N \geq 16$, then
$ t_n \leq \phi_{16}(n) =   \sqrt{n} + 2 \log n.$

\begin{cor}\label{cor:delta_n_powers}
With notation as above, if $0 < \delta \leq 1$
and  $N > 2\left( 1 +\frac 1 \delta\right)$, 
then 
\[
\sum_{n=1}^\infty \Delta_n^{1+\delta} = O\left(\frac 1 \delta\right),
\]
where the constant does not depend on $N$.
\end{cor} 

\begin{proof}
After some arithmetic, we see that  the  hypothesis
$N> 2(1+1/\delta)$  is equivalent to 
$
(\alpha-1)(1+\delta)
 <  -1 -  \frac \delta 2,
$
and after some calculus, \eqref{eqn:Delta_n} implies 
\[
\sum_{n=1}^{\infty} \Delta_n^{1+\delta} 
= O\left(\sum_{n=1}^{\infty} n^{-1-\frac\delta 2}\right)
=O\left(\frac 1 \delta\right). \qedhere
\]
\end{proof} 

\section{The ``branches'' of the tree} \label{branch sec}

We modify the components $R_k$ by adding line segments 
perpendicular to the boundary as illustrated in Figure 
\ref{Rcomp12}. We describe the construction of the 
R-components only for the component $R_0$ intersecting 
the positive real axis; the other R-components will
all be rotations of this one.

\begin{figure}[htb]
\centerline{
\includegraphics[height=2.0in]{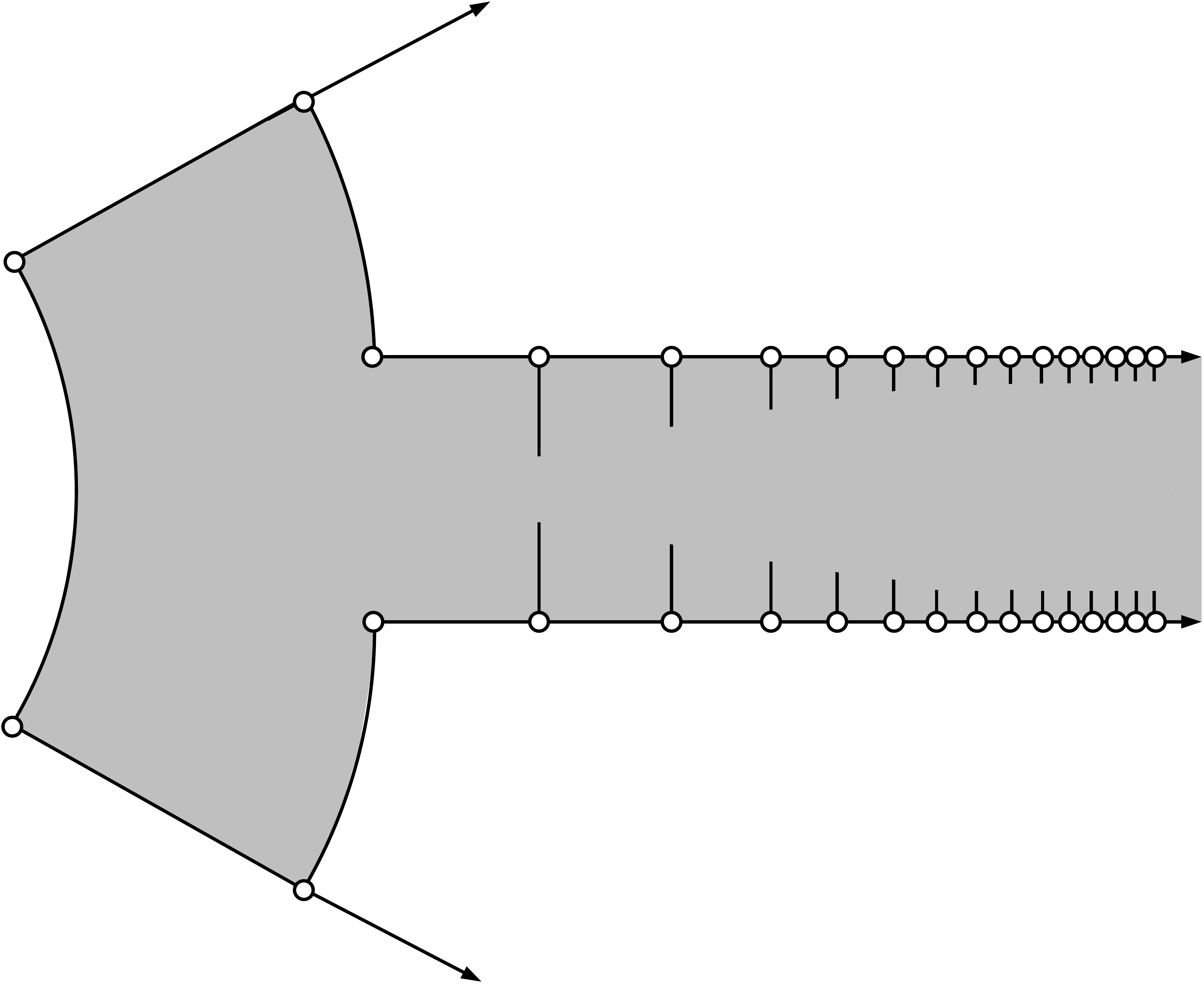}
}
\caption{ \label{Rcomp12}
An R-component and the conformal partition points 
coming from  the two adjacent L-components. The spikes get
 shorter and closer together  near
$\infty$, but they do not accumulate at any finite point.
}
\end{figure}

\begin{figure}[htb]
	\centerline{
		\includegraphics[height=3in]{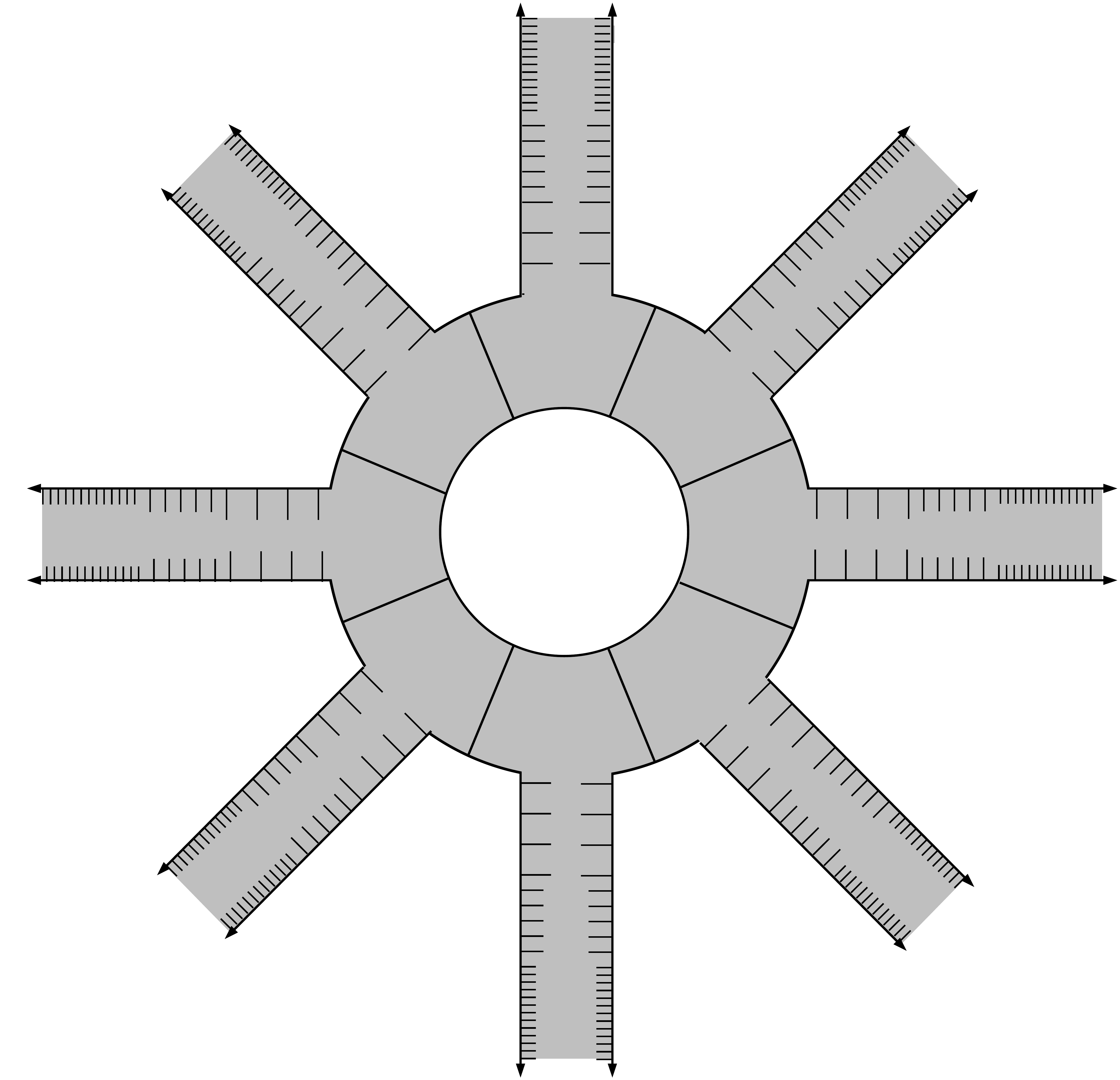}
	}
	\caption{ \label{fig:complete_graph}
The final shape of the graph with all the slits attached.
The vertices on the  slits (defined in Section \ref{adding vertices})
are  too close together to see  at this scale.
	}
\end{figure}

We define the
modified R-component $\Omega_0 \subset R_0$ by 
removing vertical slits  that are  attached to the 
top and bottom edges of $R_0$  at the partition points
 of the adjacent L-components; these points were described in 
Section \ref{L component sec}.
The region between two adjacent slits will 
be informally referred to as a ``tower''; it is 
the trapezoid defined by the two slits and the 
connecting segment on the boundary of $R_0$.
 The slits will be chosen
 so that the  domain $\Omega_0$ is symmetric with respect
 to the real line. 
 Thus, it suffices to define the length of the slit attached to the point $w_n = z_0 + N + t_n$ on the top edge of $\Omega_0$ (the top edge is the horizontal ray starting at $w_2 = z_0 + N$). 

The segment attached at the partition point $w_n$
 is denoted $\lambda_n$  and has length
\begin{eqnarray} \label{defn y_n}
y_n = \min\left\{\frac 14,\Delta_n
           \left(t_n+\frac 1\pi\log\Delta_n\right)\right\}.
\end{eqnarray}
If $y_n < 1/4$, then
by Lemma \ref{adjacent estimates} we have 
\begin{eqnarray} \label{t_n equality}
t_n
=\frac {y_n}{\Delta_n} - \frac{\alpha-1}{\pi}\log n+O\left(n^{\alpha-2}\right)
=\frac {y_n}{\Delta_n} - \frac{\alpha-1}{\pi}\log n+O\left(\frac 1 n\right)
\end{eqnarray} 
because $\alpha -2 = \frac 1 N -2 < -1$. 
We will later interpret this equation 
as an  equality (up to a bounded
additive factor) between two  hyperbolic distances
in $\Omega_0$; see Corollary \ref{equal hyper dist}.
 This approximate equality
 will imply that the desired $\tau$-length
upper and lower bounds in Theorem \ref{thm:folding} hold.
See Lemma \ref{vertex placement}. 

\begin{lem}
We have $y_n <  1/4$ for $ n \geq n_0$, with $n_0$  
independent of $N$. 
\end{lem}

\begin{proof} 
From Lemma \ref{adjacent estimates} we know that
\[
N(n^{\alpha}-1)   \leq  \phi_N(n) = n^{2\sqrt{\alpha}} + 2\log n.
\]
Since $\{\phi_N\}$ is decreasing (and 
$\frac 1{20} + \frac 2{\sqrt{20}}-1 < - \frac 12$),
 if $N \geq 20$,  then
\begin{eqnarray*}
y_n = \Delta_n\left(t_n+\frac{1}{\pi}\log  \Delta_n \right)
&=& 
 O\left(  n^{\alpha-1} ( \phi_{N}(n)  + \log n )\right)\\
&=& 
 O\left( n^{\alpha+2 \sqrt{\alpha} -1}  + n^{\alpha-1}  \log n \right)\\
&=& 
 O\left(  n^{-1/2  }  \right). \qedhere
\end{eqnarray*}
\end{proof} 

Thus, only finitely many segments will have
length $1/4$ and this number is bounded independent of $N$. 
Let $\var1epsilon_n = \min\{\frac 14, C_1/\sqrt{n}\}$, for 
$n \geq 2$,  where
$C_1$ is chosen so that $y_n \leq \var1epsilon_n$ for all 
$n \geq 2$.
Define $\var1epsilon(t_n) = \var1epsilon_n$   and define 
$\var1epsilon(t)$ for $t_n \leq t \leq t_{n+1}$  by
linear extension. This function is  continuous and  decreasing,
and we have $\Omega_1\subset\Omega_0\subset R_0$ where
\[
\Omega_1=\left\{(t+\Real z_0+N)+iy: t >0, 
 |y|< \frac 12-\var1epsilon(t)\right\}.
\]
Recall that $t_2=0$ so that $\var1epsilon$ is in fact defined on $[0,\infty)$.

\begin{lem} \label{epsilon lemma} 
If $N \geq 20$, then 
$\int_{0}^\infty \var1epsilon(t) dt < \infty $
with a bound that is  independent of $N$.
\end{lem}

\begin{proof}
Choose $n_0$ so that
 $\var1epsilon_n=C_1/\sqrt{n}$ for $n\geq n_0$. By
 Lemma \ref{adjacent estimates},
\begin{align*}
\int_0^\infty \var1epsilon(t)dt
&\leq\sum_{n=2}^\infty\var1epsilon(t_n)\Delta_n
=\sum_{n=2}^{n_0-1}\frac{\Delta_n}{4}
+\sum_{n=n_0}^\infty\frac{C_1\Delta_n}{\sqrt{n}}\\
&\leq  C_2+ C_3\sum_{n=n_0}^\infty n^{\alpha-\frac 32}
\leq  C_2+ C_3\sum_{n=n_0}^\infty  n^{-\frac 43}<\infty.
\qedhere
\end{align*}
\end{proof} 

This will be used in Corollary \ref{cor:distance in R comp} 
to approximate hyperbolic distance in $\Omega_0$.

\section{$\{y_n\} $ is almost convex}
\label{convexity sec}

Imagine that we connect the endpoints
of adjacent vertical slits by segments to form an 
infinite polygonal path. We want to verify that this
path is ``not far'' from being convex.
 More precisely, define 
$s_n = (y_{n} - y_{n+1})/\Delta_n$; this is the
 slope of the segment connecting the endpoints of 
the $n^{\text{th}}$ and $(n+1)^{\text{st}}$ slits.
The quantity $(s_n-s_{n+1})/\Delta_n$ is thus a
type of second derivative, and can be thought of 
as the curvature of the polygonal path. The path will 
be convex down if all these numbers are negative.
This is true, but tedious to prove; we will 
give an easier estimate that is sufficient for 
our needs.
 Recall that $\phi_N(n)=n^{2\sqrt{\alpha}} + 
2 \log n$, $\alpha = 1/N$.

\begin{lem} \label{convexity lemma} 
Suppose $N \geq 24 $. There is a $n_0>0$, independent 
of $N$,  so that for $n \geq n_0$,  $s_n$ is decreasing to zero, and 
\[
|s_{n} - s_{n+1}|   
=O\left(\phi_{16}(n)/n^2\right) 
=  O\left(   n^{2\sqrt{1/16} -2} \right) 
    +O\left(n^{-2} \log n\right) 
=O\left(n^{-3/2}\right),
\]
with constants independent of $N$. In particular, $|s_{n} - s_{n+1}|=o\left(n^{\alpha-1}\right) =o\left(\Delta_n\right)$.
\end{lem}
\begin{proof}
We simply compute using the definitions. Recall  
that $t_{n+1} = t_n + \Delta_n$. For notational convenience, 
let $\kappa_1 = (1-\alpha)/\pi$, $\kappa_2 = \frac{1-\alpha}2$,
and let $\phi= \phi_{16}$. 
If $n$ is large enough, we get from \eqref{defn y_n}
\begin{align*}
s_n&=\frac {y_{n}-y_{n+1} }{\Delta_n}=\frac {\Delta_n(t_n+ \frac 1\pi \log \Delta_n)-\Delta_{n+1}(t_{n+1}+\frac 1 \pi \log \Delta_{n+1})  }{\Delta_n} \\
&=t_n + \frac 1\pi \log \Delta_n - \left(1-  \frac{\Delta_n-\Delta_{n+1}}{\Delta_n}\right)\left(t_{n}+\Delta_n +\frac 1 \pi \log \Delta_{n+1} \right)  \\
&= - \Delta_n + \frac 1\pi (\log  {\Delta_n} - \log {\Delta_{n+1}})  + \left(\frac{\Delta_n-\Delta_{n+1}}{\Delta_n}\right)\left(t_{n}+\Delta_n +\frac 1 \pi \log\Delta_{n+1} \right).
\end{align*}
Now use Lemma \ref{adjacent estimates} to  note that
\[
\log \Delta_n 	= \log\left(n^{\alpha-1}\left(1+\frac{\alpha-1}{2n} 
	+ O\left(n^{-2}\right)\right)\right) 
	=  \left(\alpha-1\right) \log n - \frac{\kappa_2}{n} + O\left(n^{-2}\right).
\]
Using this we get 
\begin{align*} 
s_n&=-\Delta_n  -  \kappa_1  \log  \frac n{n+1}  
	+ O(n^{-2})   \\
	& \qquad \qquad  + \frac{\Delta_n-\Delta_{n+1}}{\Delta_n}
	  \left[t_{n}+\Delta_n -\kappa_1 \log(n+1)+   O( n^{-1}) \right] \\
&=  -\Delta_n  +  \frac {\kappa_1}{n}  
+   \left(\frac {\kappa_2}{n} + O(n^{-2})\right)
\left[ t_n + \Delta_n  - \kappa_1 \log(n+1)+   O( n^{-1}) \right] +O(n^{-2}) \\
&=  -\Delta_n+   \frac{\kappa_1}{n}    
+   \kappa_2  \frac {t_n}{n}
- \kappa_1 \kappa_2 \frac { \log(n+1)}{n} +O\left(\phi(n)/n^{2}\right) ,
\end{align*}
where we have used $\Delta_n/n = O\left(n^{\alpha-2}\right) =O\left(\phi(n)/n^2\right)$. This shows that $s_n \to 0$.
	
We now want to estimate $|s_n - s_{n+1}|$.
The big-O term is already the correct size, 
and taking differences preserves this. 
Next, note   we have already shown in (\ref{eqn:Delta diff})  that 
\[
\Delta_n -  \Delta_{n+1} = O\left(n^{\alpha-2}\right) = O\left(\phi(n)/n^2\right).
\]

Since we assumed $N \geq 24$ we have   $\alpha 
= 1/N < 1/4$, so  using \eqref{eqn:t_n 2} and \eqref{eqn:Delta_n}
and ignoring the multiplicative factor we get
\begin{align*} 
\frac {t_n}{n} - \frac {t_{n+1}}{n+1}&={t_n} \left(\frac 1n - \frac 1{n+1}\right) +  \frac { t_n- t_{n+1}}{n+1}  \\
&=O\left(  {t_n}  n^{-2} \right) + O\left(\frac {\Delta_n}{n+1}\right)\\
&= O\left(\phi(n)/n^2\right).
\end{align*} 
Similar arguments complete the proof by showing 
\begin{eqnarray*} 
\frac  1 n- \frac 1 {n+1} 
&=&   O\left( \frac 1{n^2}\right) = O\left(\phi(n)/n^2\right),
\end{eqnarray*} 
\begin{eqnarray*} 
\frac {\log (n+1)}{n} - \frac {\log (n+2)}{n+1} 
&=&   O\left(\frac {\log n}{n^2}\right) = O\left(\phi(n)/n^2\right).
\qedhere
\end{eqnarray*} 
\end{proof} 

Next we derive a geometric consequence: 

\begin{cor} \label{tangent lemma} 
For $n > n_0$ (as in the previous lemma) the endpoints of 
the $n^\text{th}$ and $(n+1)^{\text{st}}$ slits are on 
the boundary of an open disk $D_n$  in $\Omega$ centered 
on the axis of $\Omega$. Furthermore, there is a disk 
$B_n$ whose radius is bounded below independently of $N$
 and $n$ containing the endpoints of the $n^{\text{th}}$ 
and $(n+1)^{\text{st}}$ so that no other endpoint lies
 within $\overline{B_n}$. The boundary of $B_n$ intersects
 all slits within unit distance of the $n^{\text{th}}$ 
slit in either direction.
\end{cor} 

\begin{proof}
Note that a disk which is centered on the 
real line and completely contained in the R-component
 has radius at most $\frac 12$. In particular, the 
curvature of such a disk is at least $2$. 
This means that slope of the circle (considered 
as the graph of a function) changes
by  at least  $\simeq \Delta$ over a horizontal 
distance  $\Delta$ (the change is smallest 
when the interval is centered around the center 
of the disk), see Figure \ref{PathVsCircle}.
On the other hand, 
we can choose $n_0$ in the Lemma \ref{convexity lemma} so that
the curvature of the polygonal path through the 
endpoints of the slits with indices larger than $n_0$
is much less than $2$. Thus if
we move from the $(n+1)$st tip
to the right, the path can never intersect the boundary 
of this disk again. If it did, then there would be
a segment of the path whose slope 
is less than or equal to the slope of
$\partial D$ at the point of intersection,  and 
this is impossible by  observations above.
 A similar argument works to the 
left of the $n$-th tip. 
This proves the existence of  the disk $D_n$.

\begin{figure}[htb]
\centerline{
\includegraphics[height=1.50in]{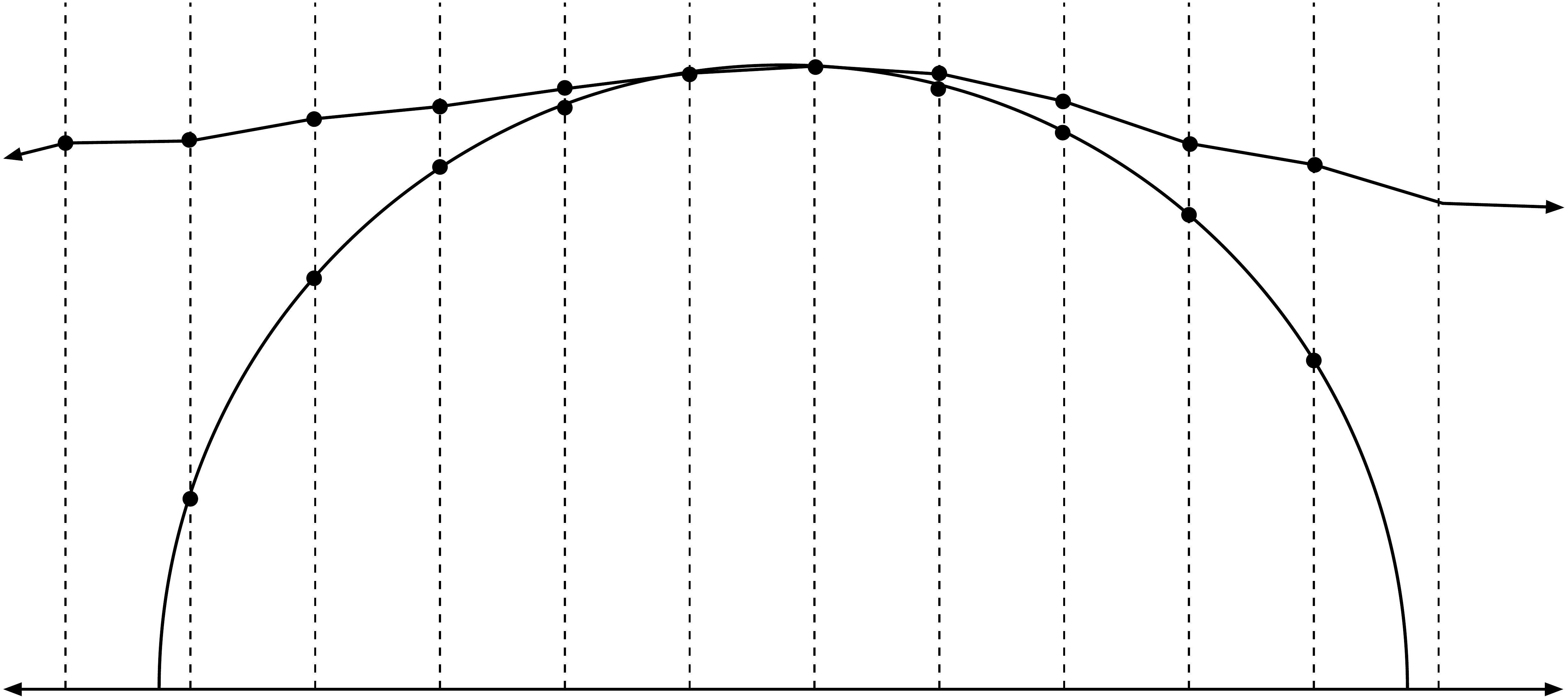}
}
\caption{ \label{PathVsCircle}
Existence of $D_n$.
By definition one segment of our path forms 
a chord of the disk  $D_n$. The slope of the 
circle changes by at least a fixed  amount 
proportional to the horizontal 
distance from this chord, and the slopes
of the path segments change at most 
by a much smaller multiple of this distance.
Thus the path cannot intersect the disk except 
along the given segment. 
}
\end{figure}

We simply define $B_n$
to be the largest disk the boundary of which passes
through the endpoints of the $n^{\text{th}}$ and 
$(n+1)^{\text{st}}$ slits and lies above the polygonal
path unlike $D_n$ which lies below that path, see
Figure \ref{TangentDisk}. Since the curvature of the
polygonal path is bounded above, we get an upper bound
for the curvature of the boundary of $B_n$ and thus
a lower bound for the radius. We can ensure that the
boundary of this disk intersects all slits within unit
distance of the $n^{\text{th}}$ slit in either 
direction by increasing $n_0$ and thus increasing the
minimal radius of $B_n$ if needed.	
\end{proof}

\begin{figure}[htb]
\centerline{
\includegraphics[height=2.0in]{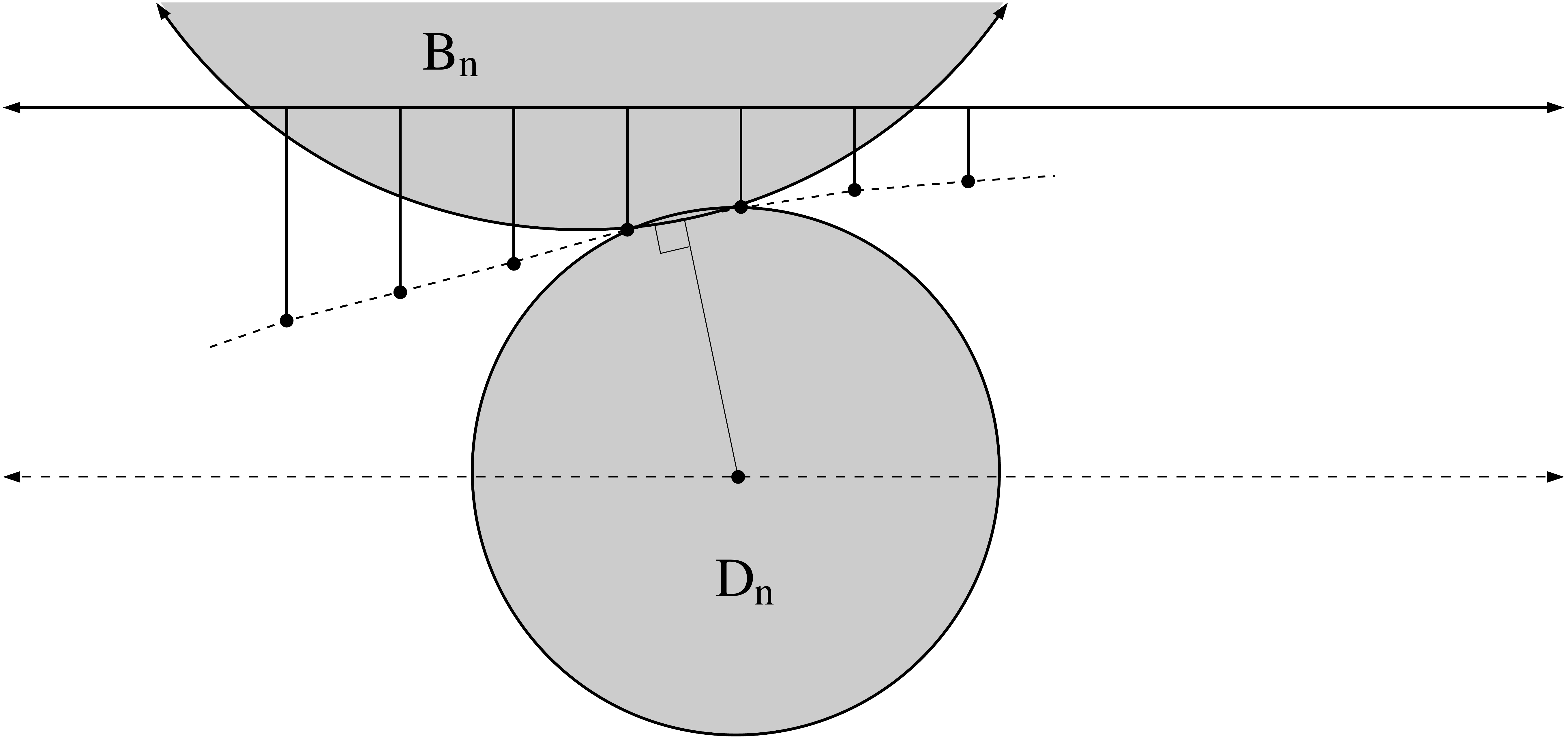}
}
\caption{ \label{TangentDisk}
Existence of $B_n$.
The dashed curved connecting the slit tips lies between 
two disks with radius bounded uniformly away from zero.
Later we will use  this to
estimate the hyperbolic distance from the axis to points near 
the tips of the slits.
}
\end{figure}

 A stronger result holds: 
 $s_n \searrow 0$.
This implies that the polygonal curve is convex and hence
the disk $B_n$ in Lemma \ref{tangent lemma} can be taken
to be a half-plane.  However, verifying this 
seems to require long and tedious computations, so  we omit the proof,
since the weaker condition above  is sufficient for our purposes.

\section{The definition of $T$}
\label{adding vertices}

We finally come to the definition of the graph $T$. 
In the previous section we have built a graph by 
defining the ``trunk'' and adding ``branches''. As a planar
set, this is $T$, but to make $T$ a graph we have to specify
the vertices. The vertices along the trunk have already been 
given; in this section we define the vertices on the 
branches. Once this is done, $T$ has been completely 
specified.

Divide the slit $\lambda_n$  attached to the trunk 
at position $w_n$ into disjoint segments of 
length $\Delta_n$, except the last segment at the end of the 
slit closest to the axis, which has length between $\Delta_n$ 
and $2\Delta_n$.
 By (\ref{defn y_n}),  $\lambda_n$ is divided into 
\[
m_n=\left\lfloor \frac{y_n}{\Delta_n}\right\rfloor
  = \left\lfloor 
           t_n+\frac 1\pi\log\Delta_n
     \right\rfloor.
 \] 
many pieces.  Note, 
\begin{eqnarray*}
\frac{y_{n+1}}{\Delta_{n+1}}-\frac {y_n}{\Delta_n}
  = \left(t_{n+1} + \frac 1 \pi  \log \Delta_{n+1}\right)
    - \left(t_{n} + \frac 1 \pi  \log \Delta_{n}\right)
  = \Delta_n + \frac 1 \pi  \log  \frac{\Delta_{n+1}}{\Delta_n}
  \to 0
\end{eqnarray*}
so  $|m_{n+1}-m_n| \leq 1$ for large $n$.
  
Next, we add  $\lfloor\exp( \pi k)\rfloor$ vertices to the
 $k$-th segment (the one that is distance $k  \cdot \Delta_n$
 from  $w_n$, the point where the slit is attached 
to the trunk). Thus, the spacing  between vertices 
in the $k$-th segment is $ \simeq \Delta_n\exp(- \pi k)$. We do 
something slightly different in the last segment at
 the end of the slit. Instead of adding vertices 
evenly spaced in that last interval, we use a square
 root to place the points, as in Figure \ref{Root},
 so that the spacing changes from $\exp(-\pi m_n) \Delta_n$ (away 
from the tip) to $\exp(-\pi m_n/2) \Delta_n$ (at the tip).
 This implies that all these edges have comparable $\tau$-length.
  
\begin{figure}[htb]
\centerline{
\includegraphics[height=1.75in]{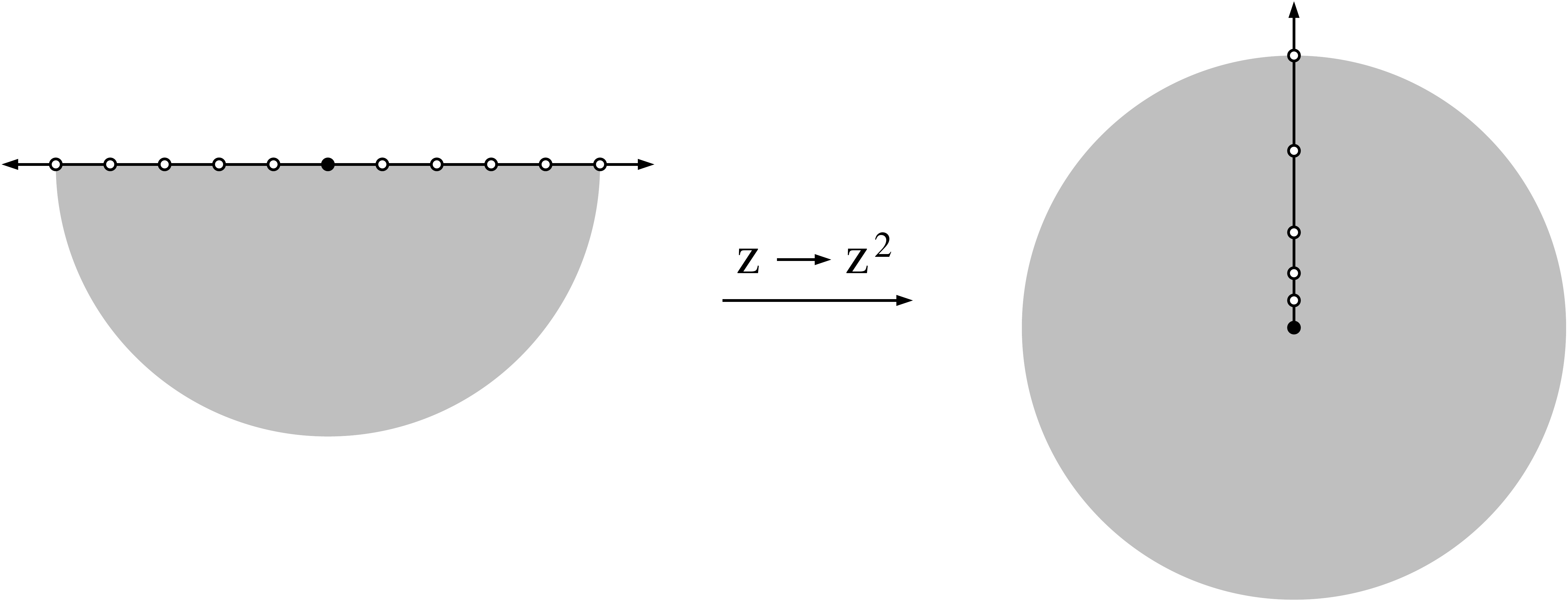}
}
\caption{ \label{Root}
Placing points near the tip of a slit.
 The even spacing on the left is mapped to an uneven 
spacing on the right that gives all the sides comparable 
 $\tau$-length.
}
\end{figure}


\begin{lem}\label{lem:bounded_geometry}
	$T$ has bounded geometry with constants independent of $N$.
\end{lem}

\begin{proof}
The graph was constructed  so this would be true, but 
we briefly review the different conditions.
First, all the edges are either straight line segments or
circular arcs with bounded curvature, independent of $N$;
the angles between adjacent edges are $\pi/2$ or
 $\pi$ and the valence of the vertices is at most $3$; 
and adjacent edges have comparable lengths by construction.
From these facts we easily deduce  that
 (1) and (2) of the bounded geometry conditions hold. 

It remains to show that for non-adjacent edges $e$ and $f$
the ratio $\diam(e)/\dist(e,f)$ is uniformly bounded.
But again, this holds by construction. The edges have
either a distance of at least $1/2$ (when they lie on 
different sides of the symmetry line of $\Omega$) or
 they lie on the same side of the symmetry line. In this
 case, they lie either on opposite sites of one of 
the ``towers'' or on sides of different towers even 
further apart.
But then the diameter of the edges is  bounded by the width
of the tower, which is a lower bound for the distance 
between the edges.
If the edges lie on the same slit, then the boundedness 
follows  since adjacent edges have uniformly comparable length and
there is at least one edge in between the two edges under consideration.
\end{proof}

We also need to know that our graph satisfies the condition in 
Lemma \ref{lem:T(r)_bound}:  

\begin{lem}  \label{verify T(r) bound}
For every edge $e$ of $T$, and  any $r>0$ the neighbourhood
$T_e(4r)=\{z\in\complex~:~\dist(z,e)<4r\cdot\diam(e)\}$
only intersects edges whose length is comparable
to the length of $e$.
\end{lem} 

We leave the (simple) proof to the reader.

\section{Estimates for the hyperbolic metric}
\label{review metric}

Now that we have defined $T$ (and hence the model 
function $F$), our next goal is to prove that it
satisfies the $\tau$-length upper and lower bounds
discussed earlier. This reduces to careful estimates
of hyperbolic length within the R-components.
We start with a review of some basic facts that can be found, 
for example, in \cite{MR2450237}. 
The hyperbolic
 length of a  (Euclidean) rectifiable curve in 
the unit disk $\disk$  is given by integrating
\[
\frac {ds}{1-|z|^2},
\]
along the curve. In the upper half-plane $\uhp$ 
 we integrate $ {ds}/{2y}$.
Note that this definition differs by a 
factor of $2$ from that given  in some sources, e.g., 
\cite{Beardon-Rectangles-II},  which contains 
estimates similar to the ones we will derive below.

The hyperbolic distance between two points is given 
by taking the infimum of all hyperbolic 
lengths of paths connecting the points. In the disk, 
minimizers (hyperbolic geodesics) are either diameters 
of the disk or subarcs of circles  perpendicular to
the unit circle. In the upper half-plane, the hyperbolic
geodesics are either vertical rays or semi-circles centred on the real line. 
The hyperbolic distance between two points $z,w$ is given by 
\[
\rho(z,w) = \frac 12 \log \frac{1+T(z,w)}{1-T(z,w)},
\]
where 
\[
T(z,w) =  \left |\frac  {z-w}{1-\overline{w}z} \right |,
\qquad 
T(z,w) =  \left |\frac  {z-w}{z-\overline{w}} \right |,
\]
for  $\disk$ and $\uhp$, respectively. 
For the disk, taking $z=0$, $w = r  = 1-\epsilon> 0$
gives 
\begin{eqnarray} \label{R epsilon}
   R= \rho(0,r) = \frac 12 \log \frac {1+r}{1-r} 
          = \frac 12 \log \frac {2-\epsilon}{\epsilon}.
\end{eqnarray}
Therefore, $ \epsilon= 2/(1+\exp(2 R))$, which gives 
$ \exp(-2 R) \leq \epsilon \leq 2 \exp(-2 R).$

 Koebe's estimate (e.g., Theorem I.4.3 of \cite{MR2450237})
 says that if $\varphi$ is conformal from $\disk$ to
 a simply connected domain $\Omega$, then
\[
\frac 14 |\varphi'(z)| (1-|z|^2)\leq \dist(\varphi(z), \partial \Omega)\leq  |\varphi'(z)| (1-|z|^2).
\]

The hyperbolic metric $\rho =\rho_\Omega$ on a simply 
connected planar domain $\Omega$  is defined by transferring
 the hyperbolic metric on  $\disk$ by  a conformal map 
(the choice of the map makes no difference). The 
quasi-hyperbolic metric on $\Omega$  is defined by integrating 
\[
d \widetilde \rho  =\frac {ds}{\dist(z, \partial \Omega)}.
\]
Koebe's estimate implies these two metrics are comparable
to within a factor of $4$, 
\[
\rho_\Omega(z,w) \leq \widetilde \rho_\Omega(z,w)  
\leq 4 \cdot \rho_\Omega(z,w).
\]
In particular, if $L$ is a line segment
 and $ \dist(L, \partial \Omega) \simeq \diam(L)$, 
then $L$ has hyperbolic length comparable to $1$.

Another domain for which we can explicitly compute the 
hyperbolic metric is the infinite strip 
$S =\{ (x,y): |y|< 1/2\}$. This is conformally 
mapped to $\rhp$ by $\exp(\pi z)$, so a simple computation shows that 
the hyperbolic  metric on $S$ is given by $d \rho = \frac \pi 2 
ds /\cos(\pi y)$. From this we can deduce that  
for $|y| < 1/2$, we have 
\begin{eqnarray} \label{vert est}
 \rho_S(x, x+iy) 
   =  \frac \pi 2 \int_0^{|y|} \frac {ds}{\cos(\pi s)}
   = \frac \pi 2 \log \frac 1{\frac 12 -|y|} + O(1).
\end{eqnarray}
   The geodesics in $S$ are a little 
difficult to draw, but some of them can be well 
approximated by a polygonal arc as follows:

\begin{lem} \label{shortcut}
If $z= x+i0$, $w = s+it$, $|t|< 1/2$, and $p=s+i 0$,  then 
$$ \rho_S(z,w) 
=  \rho_S(z,p)+\rho_S(p,w) +   O(1)
= \frac \pi 2 |x-s| + \frac \pi 2 \log \frac 1{\frac 12 -|t|} +O(1). $$ 
\end{lem}

\begin{proof}
One direction is obvious by the triangle inequality and
(\ref{vert est}).
To prove the other direction, we can use 
conformal invariance to replace $S$ by $\disk$, 
set $p=0$ and assume $z$ is on the segment from $0$ to $i$
and $w$ is on the segment from $0$ to $1$. See Figure \ref{Shortcut}.
The geodesic  $\gamma$ from $z$ to $w$ must hit the Euclidean ball 
$D=D(0, \sqrt{2}-1)$. Note that $D \cap \gamma$ has
hyperbolic length 
$O(1)$  and the segments of $\gamma\setminus D$ (if any)
are longer than $\rho(w, D) = \rho(w,0) -O(1)$ and $\rho(z,0)
= \rho(0,z) - O(1)$. This proves the result.
\end{proof} 

\begin{figure}[htb]
\centerline{
\includegraphics[height=1.5in]{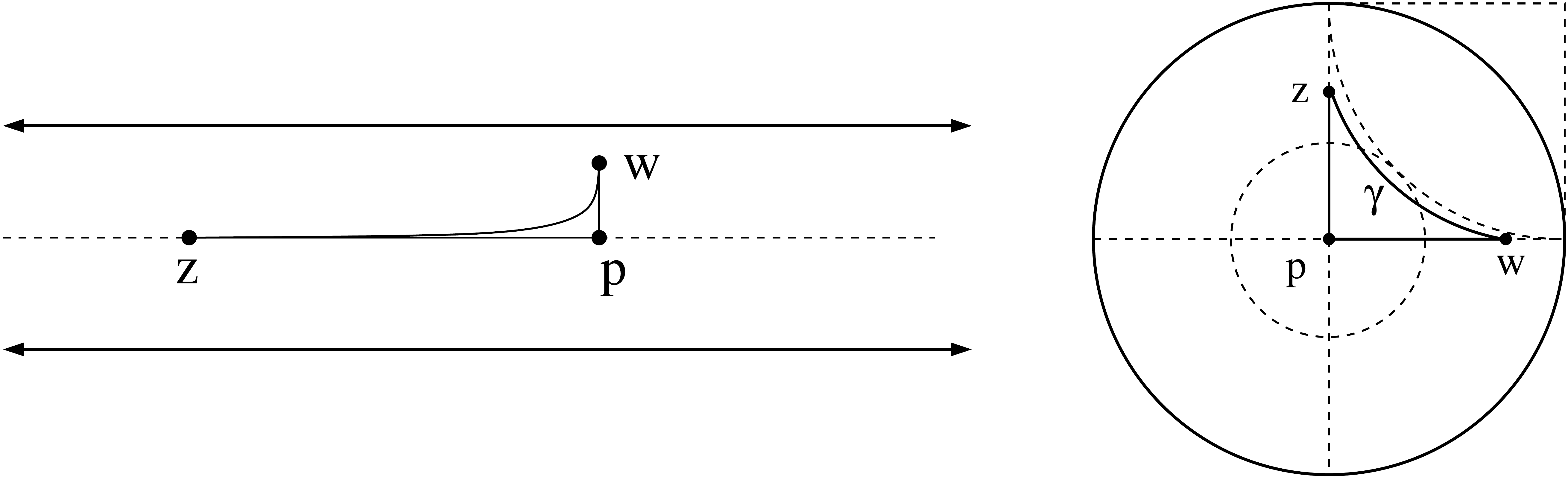}
}
\caption{ \label{Shortcut}
Proof of Lemma \ref{shortcut}. The geodesic
is only $O(1)$ shorter than the polygonal path.
}
\end{figure}



\begin{lem} \label{sharp schwarz}
If $W_1\subset W_2$ are simply connected 
domains, $z \in W_1$, and $R = \rho_{W_2}
(z, W_2 \setminus W_1)  \geq c >0$,  then 
$$\rho_{W_2}(z) \leq \rho_{W_1}(z)
 \leq (1+O( e^{-2R}))\rho_{W_2}(z),$$ 
where the constant in the ``big-O'' only depends on $c$.
\end{lem}

\begin{proof}
The left inequality is a well known consequence of the Schwarz lemma.
To prove the right side, we may assume using conformal invariance that 
$W_2 = \disk \subset W_1/(1-\epsilon)$ where $R$ and $\epsilon$
are as in (\ref{R epsilon}). Thus  
$$\rho_{W_2} (0)
\leq  \rho_{W_1}(0)/(1-\epsilon)
\leq  \rho_{W_1}(0)(1+O(\epsilon)) 
\leq  \rho_{W_1}(0)(1+ O( \exp(-2R))),$$
if $\epsilon < 1$ or $ R > 0$ uniformly.  
\end{proof} 

For an even more precise version of Lemma \ref{sharp schwarz}, 
see Proposition 3.4 of \cite{Helena_Lasse_Absence}.

\begin{lem}\label{verify tau cond}
Suppose  $I,J$ are disjoint intervals  on the boundary 
of $\rhp$ and suppose $\gamma_I, \gamma_J$ are the hyperbolic 
geodesics  that have the same endpoints as $I$ and $J$ 
respectively. Suppose $\gamma$ is the geodesic that 
connects the center of $I$  to $\infty$ 
($\gamma$ is a horizontal ray). 
Let $z$ be the point on $\gamma$ that  is closest to 
$\gamma_J $ with respect to the hyperbolic 
metric in the right half-plane. 
If $\rho(z, \gamma_J) = \rho(z, \gamma_I) +C$, 
then $|I| \simeq |J|$ with a multiplicative 
factor depending only on $C$.
\end{lem} 

\begin{proof}
  See Figure \ref{TauEst}.
Map  $\rhp$
 to  $\disk$   by a M{\"o}bius transformation 
$\tau$ that sends the center of $I$ to 
$-1$, $z$ to $0$ and $\infty$ to $1$.
It is easy to check the images of $I$ and $J$
have comparable length  on the circle, and 
that the image of $J$ is bounded away from $1$. 
This  implies $I$ and $J$ have comparable length
on $i \reals$ because the derivative of $\tau^{-1}$
has comparable  absolute values on $I$ and $J$.
\end{proof}

\begin{figure}[htb]
\centerline{
\includegraphics[height=1.75in]{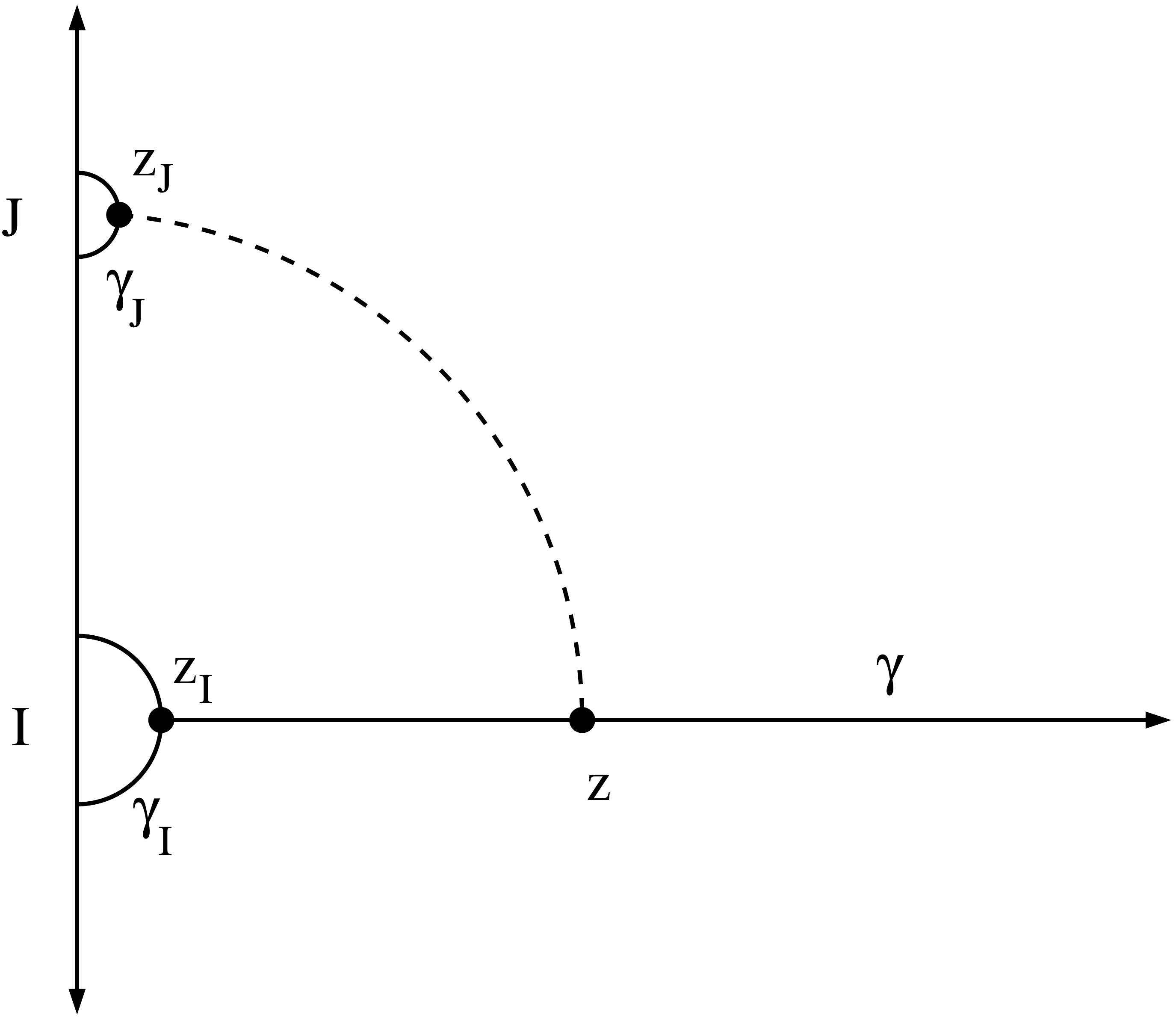}
$\hphantom{xxxxxx}$
\includegraphics[height=1.75in]{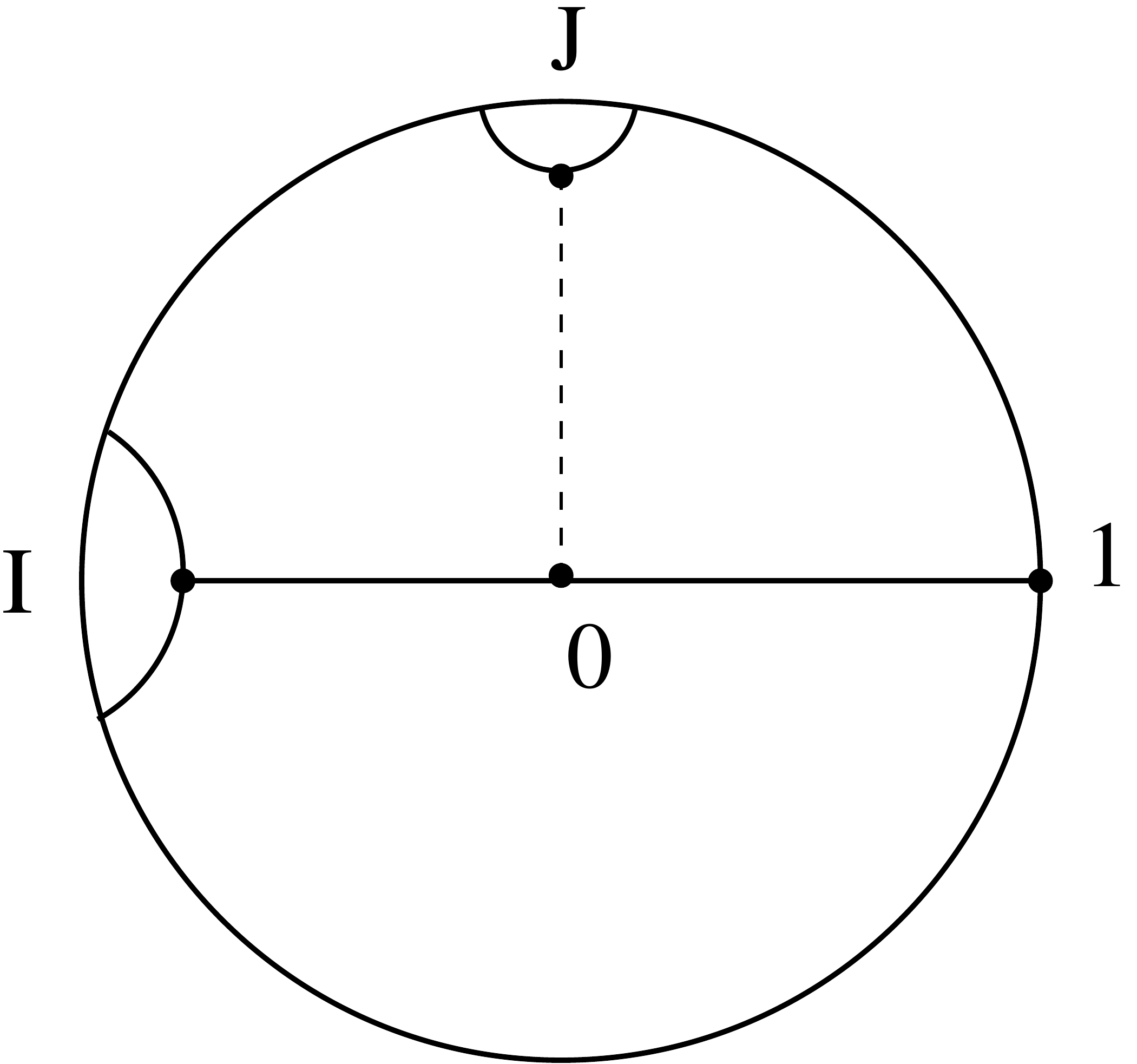}
}
	\caption{ \label{TauEst}
 Proof of Lemma \ref{verify tau cond}. 
	}
\end{figure}


\begin{lem} \label{square lemma}
Suppose $Q$ is an open square with center $q$, 
and $\Omega \supset Q$ is a 
simply connected region such that $\partial \Omega$ 
contains the top and bottom sides of $Q$ and $\Omega$
	contains the left and right sides of $Q$. Then 
$\Omega\setminus Q$ has two connected components, 
	separated by $Q$. Suppose  that 
$a,b$ are in different components. Let $\gamma$ be the 
hyperbolic geodesic in $\Omega$ connecting $a$ and $b$. 
Then $\rho_{\Omega}(\gamma,q) = O(1)$.
\end{lem}

\begin{proof}
Let $I$ and $J$ be the top and bottom sides of $Q$. Let 
$\sigma_I, \sigma_J$ be the hyperbolic geodesics for $Q$ 
joining the endpoints of $I$ and $J$ respectively, and 
let $\gamma_I$, $\gamma_J$ be the hyperbolic geodesics 
for $\Omega$ joining the same pairs of points, see Figure \ref{SquareProof}.

It is a
standard fact that $\sigma_I$ are exactly the points $z$ in 
$Q$ where $I$ has harmonic measure $1/2$ (the solution 
of the Dirichlet problem with boundary values $1$ on $I$
and zero on $\partial Q \setminus I$).
Since $Q  \subsetneq \Omega$, 
the maximum principle for harmonic functions implies that  
the harmonic measure of $I$ in $\Omega$ will be 
$  > 1/2$  at each point of  $\sigma_I$,
 and therefore $\gamma_I$ is separated
from $I$ in $\Omega$ by $\sigma_I$. Similarly for $\gamma_J$ 
and $\sigma_J$.
See Figure \ref{SquareProof}.

Finally,  distinct hyperbolic geodesics in a simply 
connected domain either intersect once or not at all 
and any intersection is a crossing (this 
is obvious in the disk or half-plane model). Hence, $\gamma$
does not intersect either $\gamma_I$ or $\gamma_J$ (it could 
not connect $a$ to $b$ if it crossed either curve  only once).
Therefore it crosses $Q$ traveling between $\sigma_I$ and 
$\sigma_J$ and hence comes within $O(1)$ of the center 
point $q$ (a simple argument leads to the explicit 
estimate  $\rho_\Omega (\gamma, q) 
\leq \rho_Q(\sigma_I,q) = \frac 12 \log \sqrt{2}/(2-\sqrt{2})
\approx .4407$).
\end{proof} 

\begin{figure}[htb]
\centerline{
\includegraphics[height=2.0in]{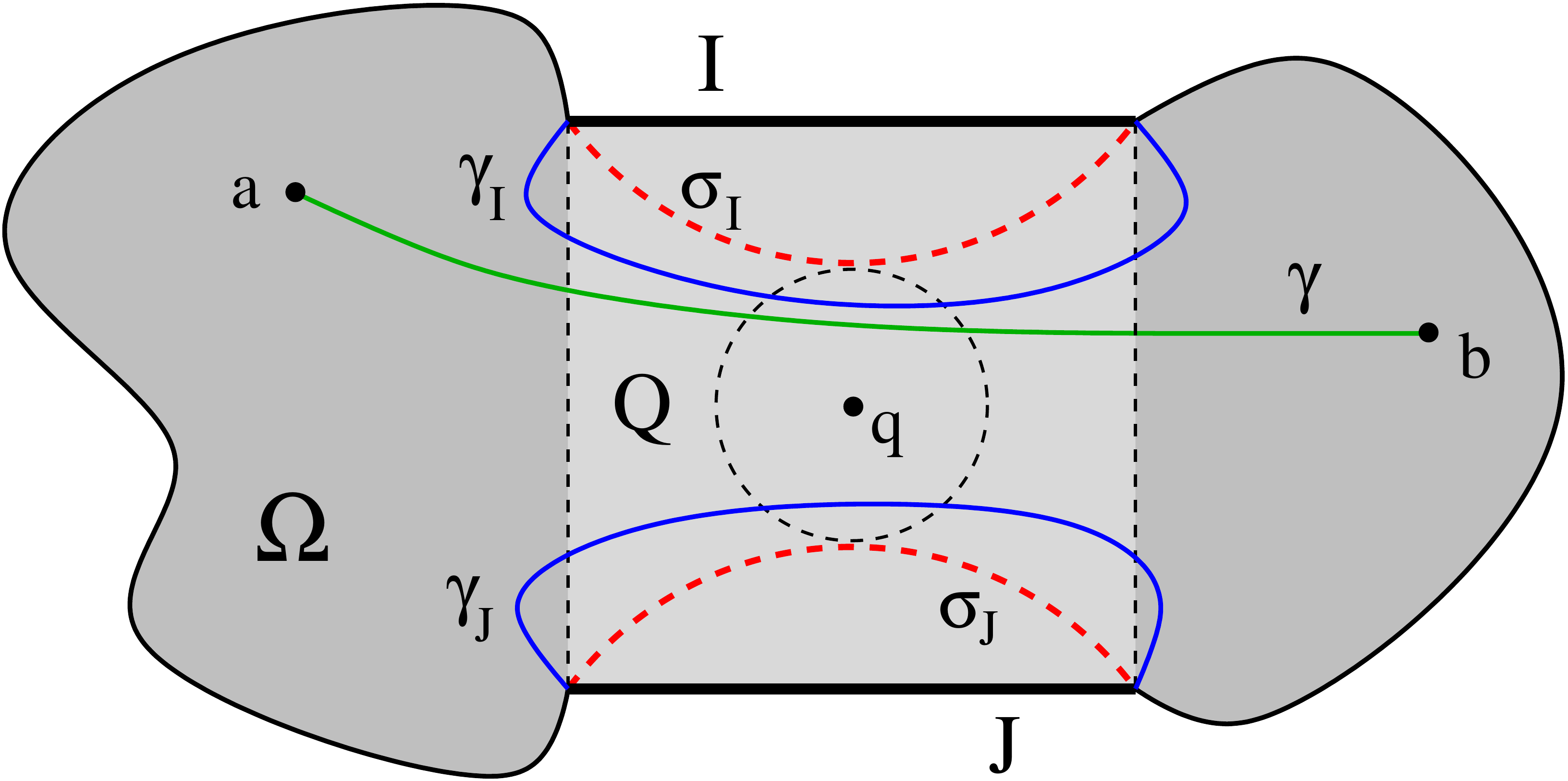}
}
\caption{ \label{SquareProof}
A geodesic  $\gamma$ for $\Omega$ 
crossing $Q$ horizontally  passes 
between the geodesics $\sigma_I$, $\sigma_J$
for $Q$ corresponding to the top and
bottom of $Q$. Thus $\gamma$ comes 
 within a fixed distance of the center $q$.
}
\end{figure}

\section{The hyperbolic metric in approximate rectangles} 

\begin{lem} \label{rectangle bound}
Fix $t >0$, set $R=\{ (x,y): 0 < x < t, |y|< 1/2\}$
and suppose $ \Omega \supset R$ is 
a simply connected domain 
whose boundary contains the
top and bottom sides of $R$. Then
$$
\rho_\Omega(a,b)  = \frac \pi 2 (b-a) + O(1)
$$
whenever $0 < a < b < t$ and $a ,(t-b)$ are 
both bounded away from zero.
\end{lem}

\begin{proof}
First consider $\Omega = R$.
Note that $ R \subset S =\{ (x,y): |y|< 1/2\}$, so
$d\rho_R \geq d\rho_S$ by one direction 
of Lemma \ref{sharp schwarz}.
On  the interval $[a,b]$, the other direction gives
\begin{eqnarray*}
 d\rho_R(x) &\leq&
  (1+ O(\exp(-2\rho_S(x,S \setminus R))))
               \cdot d\rho_S(x)  \\
&\leq& (1+ O(\exp(- \pi \min(x,t-x)))) \cdot 
    \frac \pi 2 dx .
\end{eqnarray*}
Integrating from $a$ to $b$ gives 
$ \rho_R(a,b) = \frac \pi 2 (b-a) +O(1).$

For a general $\Omega$, we repeat this argument with $S$ 
replaced by $\Omega$ to get 
$$ d\rho_R(x)/ (1+ O(\exp(-2\rho_\Omega(x,\Omega \setminus R))) ) 
 \leq d \rho_\Omega(x) \leq d \rho_R(x),$$
for $a < x < b$.
By Koebe's theorem 
$\rho_\Omega(x,\Omega \setminus R) \simeq \min(x,t-x),$
so  integrating  gives 
$$ \rho_\Omega(a,b) = \rho_R(a,b)+ O(1) = \frac \pi 2 (b-a) 
+O(1). \qedhere
$$
\end{proof}



Recall the definitions of $\Omega_0$ and $\Omega_2$ from 
Section \ref{branch sec}  ($\Omega_0$ was defined in the 
second paragraph and $\Omega_1$ just before Lemma \ref{epsilon lemma}).

\begin{lem}
Suppose $\epsilon(t)$ and $\Omega_1$ are as in 
Lemma \ref{epsilon lemma} and $1< a < b$. Then 
$\rho_{\Omega_1}(a,b) = \frac \pi 2 (b-a) +O(1).$ 
\end{lem} 

\begin{proof}
The lower bound is obvious since $\Omega_1 \subset S$.
To prove the other direction, 
we first claim that for $x > 1$,  
\begin{eqnarray} \label{dist est}
 \rho_S(x,S \setminus \Omega_1) =  \frac \pi 2\log \epsilon(x) + O(1).
\end{eqnarray}
For small $x$ this is trivial, and for large $x$ 
we deduce from Lemma \ref{shortcut} that 
$$ \rho_S(x,S \setminus \Omega_1) =  O(1) + \frac \pi 2 
          \inf\{  |t-x| - \log  \epsilon(t) : t > 1\}.$$
Since $\frac {d}{dt}\epsilon(t) =  {C_1}/{2t} <   1$ for 
$ t \geq C_1/2$, we see the infimum  is attained at $t=x$,
which is equivalent to the claim  (\ref{dist est}).
Thus,
\begin{eqnarray*}
d\rho_\Omega(x)   
&\leq & \left(1+ O\left(\exp\left( - \pi  \log \frac 1{\epsilon\left(x\right)}\right)\right)\right) 
   d \rho_S(x)  \\
&\leq&  \frac \pi 2 \left(1+ O\left(\exp\left( - \pi  \log \frac 1{\epsilon(x)}\right)\right)\right) ds.
\end{eqnarray*}
Integrating from $a$ to $b$ proves the lemma.
\end{proof} 

Since $\Omega_1 \subset \Omega_0 \subset R_0$, the following is now
immediate

\begin{cor}\label{cor:distance in R comp}
In the R-component $\Omega_0$, 
if  $\Real z_0+N+1\leq s<t$,  then 
\[
\rho_{\Omega_0}(s,t)=\frac\pi 2 (t-s)+O(1).
\]
\end{cor}

\section{The hyperbolic metric in approximate half-planes}

In this section, we prove that a domain 
that ``looks like'' a half-plane has a 
hyperbolic metric that approximates  the hyperbolic
metric on the half-plane.

\begin{lem} \label{covering lemma} 
Suppose $n$ is a positive integer and  suppose  
$ -1 = z_1 < z_2 < \dots < z_n =1$ are $n$ points
in $[-1,1]$. Let $I_j = (z_j, z_{j+1})$ and assume 
these intervals all have comparable lengths, say 
$ 1/n \leq |z_{j+1} -z_j|  \leq  4/n$ for $j=1, \dots, n-1$.
Let $\Omega$ be the complex plane with these  $n$ points
removed. The hyperbolic distance (in $\Omega$) 
between $a = i$ and $b= i/n $ satisfies 
\[
\rho_\Omega(a,b)  =  \frac 12 \log   n + O(1) = 
\rho_{\uhp}(a,b) + O(1) .
\]
\end{lem}

\begin{proof}
By symmetry, the  segments $\reals \cap \Omega$ are
 hyperbolic geodesics in $\Omega$ and therefore they
 lift to hyperbolic geodesics in the upper half-plane
 under the covering map from the upper half-plane to
 $\Omega$. We can choose the covering map so that the
 points $\pm 1, \infty$ map to themselves. Let $\Omega'$ be the
 preimage of $\uhp$ under the covering map as shown
 in Figure \ref{CoveringMap}. The points $a,b$  lift
 to points $c,d$ and $\rho_\Omega(a,b) = \rho_\uhp(c,d)$.
 The point $a$ gives comparable harmonic measure in
 $\uhp$ to the three intervals $(- \infty, -1], [-1,1],
 [1, \infty]$. Thus, the point $c$ gives comparable
 harmonic measure to the two vertical rays in $\partial 
\Omega'$ and to the arc (union of semicircles centred 
on the real line) of $\partial \Omega'$ joining these rays. 
This implies that the imaginary part of $c$ is comparable to $1$.
	
\begin{figure}[htb]
\centerline{
\includegraphics[height=1.5in]{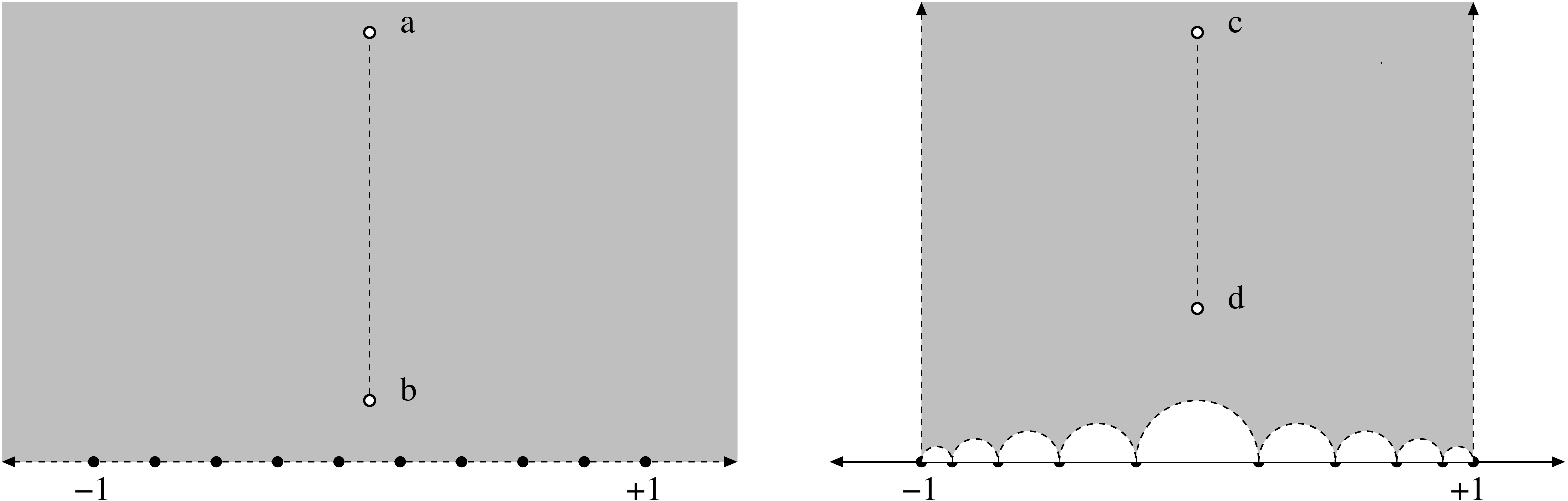}
}
\caption{ \label{CoveringMap}
The punctured plane $\Omega$ is covered by the
upper half-plane $\uhp$. Under the covering map,
$\uhp \subset \Omega$ (left) has a preimage
bounded by hyperbolic geodesics (right). 
}
\end{figure}
	
Let $I_k$ be a component of $\Omega \cap \reals$ 
whose closure contains $0$. By assumption this interval
has Euclidean length comparable to $1/n$ and hence it
has harmonic measure comparable to $1/n$ in $\uhp$
with respect to the point $a$. Similarly, for the 
intervals $I_{k-1}$ and $I_{k+1}$ on either side of $I_k$.
Therefore, the circular arcs in $\partial \Omega'$ 
corresponding to these three intervals have  comparable
harmonic measure (in the upper half-plane) 
with respect to the point $c$. Thus the Euclidean
diameters of these circular arcs are comparable to $1/n$. 
	
The point $b$ gives harmonic measure comparable to $1$ 
to the segment $I_k$, and thus $d$ gives the same  
 harmonic measure to the corresponding circular arc
 in $\partial \Omega'$. Therefore, the imaginary part 
of   $d$ is comparable to the diameter of this arc, i.e., $\simeq 1/n$. 
Thus, the hyperbolic distance between $c$ and $d$ in $\uhp$ is 
$\rho_\uhp(c,d) 
	= \frac 12 \log n + O(1).$
\end{proof}

\section{All $\tau$-lengths are comparable to $1$}
 \label{comparable tau}

We now come to the central estimate of the construction.
Consider an R-component $\Omega$  
that 
 is  symmetric with respect to the
 real line and the upper and lower horizontal sides
 of $\Omega$ are distance $1/2$ from the real axis,
 that is $\Omega=\Omega_0$. Choose a basepoint 
$A=\Real(z_0)+N+2$ in $\Omega_0$ on the real line. 
By symmetry, the hyperbolic geodesic from $A$ to  
$\infty$ is the horizontal ray from $A$ to $+\infty$.
 Let $B$ be a point on this ray to the right of $A$ 
(later, we will only need to consider points $B$ 
sufficiently far to the right).
See Figure \ref{Rcomp4}
The following 
 is immediate from  Corollary \ref{cor:distance in R comp}:

\begin{lem}\label{lemma A B}
$\rho_{\Omega_0}(A,B)=\frac \pi 2(B-A)+O(1)$.
\end{lem}

Assume that  $B$ is 
located so the vertical segment from $B$ connects 
it to a midpoint of one of the horizontal edges $I$
of $\partial\Omega_0$, see Figures \ref{Rcomp4} and 
\ref{Rcomp5}. Recall that $|I|$ denotes the length 
of $I$ and let $E$ be the point below the center of
$I$ and distance $|I|/2$  from  $I$. Let $J$ be the
shorter vertical side of the tower with top $I$. 
Let $D$ be the point distance  $|J|-|I|$ below $E$.
Let $C$ be the point  distance $3|I|$  below  $D$,
see Figure \ref{Rcomp5}.

\begin{figure}[htb]
\centerline{
\includegraphics[height=1.5in]{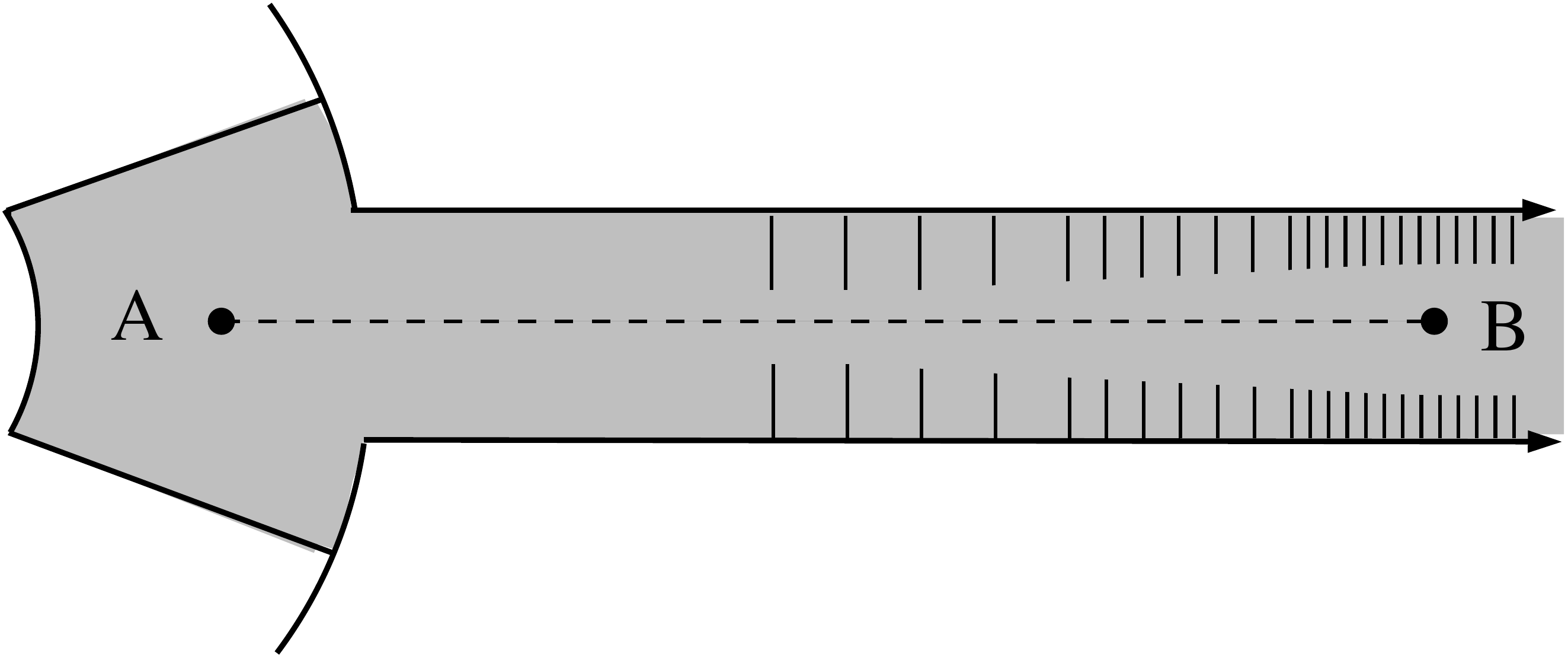}
}
\caption{ \label{Rcomp4}
Illustration of the R-component $\Omega_0$ and the placement of the points $A$ and $B$.
}
\end{figure}

\begin{figure}[htb]
\centerline{
\includegraphics[height=2.0in]{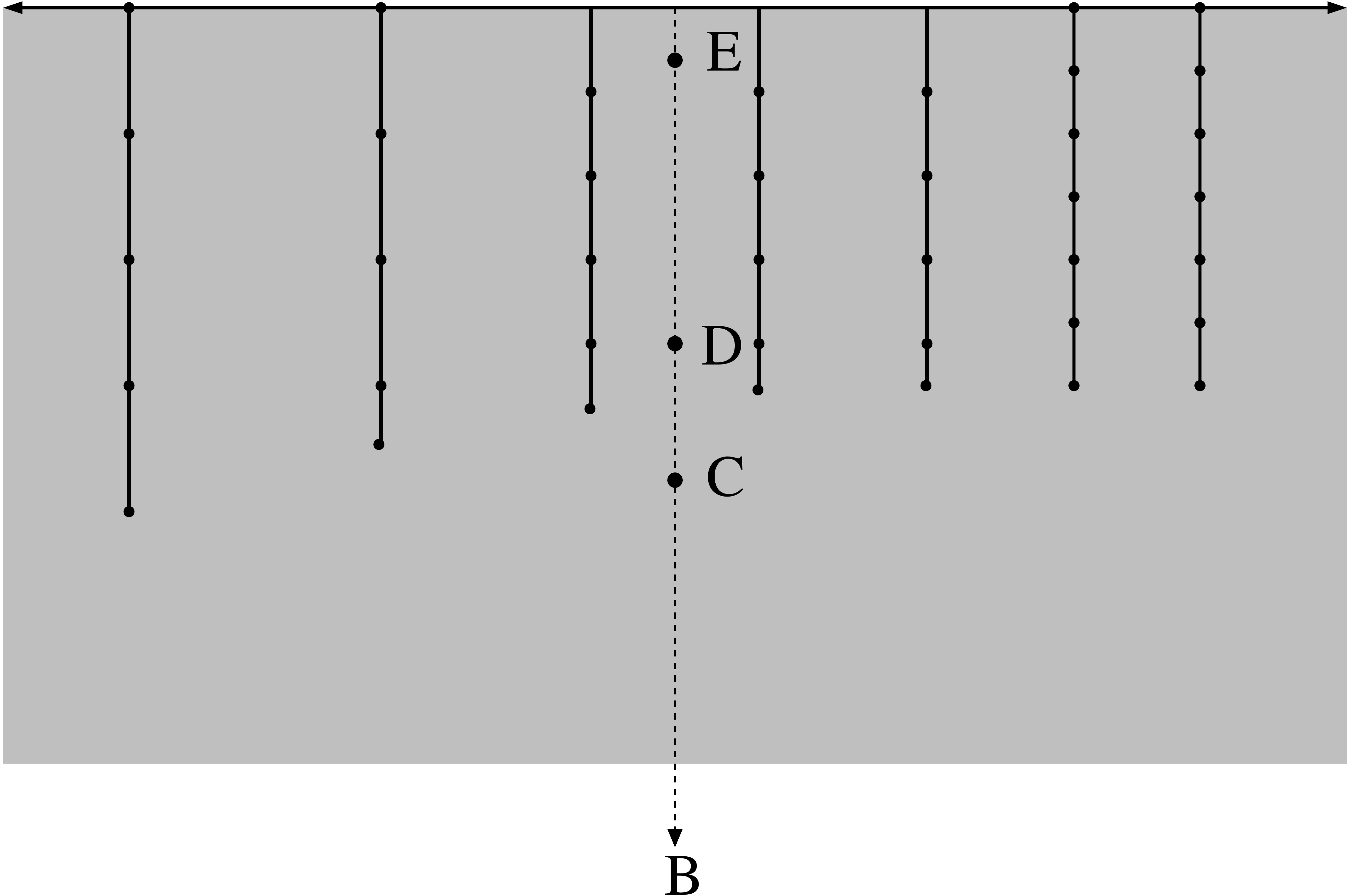}
}
\caption{ \label{Rcomp5}
We will estimate
$\rho_{\Omega_0}(B,E)$ up to an additive
factor  by breaking it into three pieces: $[BC]$ is 
estimated by Lemma \ref{lemma B C},  $[DE]$ is estimated
by Lemma \ref{lemma D E}, and $[CD]$ has bounded  
hyperbolic length by Lemma \ref{lemma C D}.
}
\end{figure}

\begin{lem} \label{lemma B C}
	$\rho_{\Omega_0}(B,C) =  \frac 12  \log \frac 1{|I|} +O(1)$.
\end{lem} 

\begin{proof}
By Corollary \ref{tangent lemma}, there is a disk $D_n$
so that $\{ B,C\} \subset D_n \subset \Omega_0$,  
with $\dist(B,\partial D_n) \simeq 1$ 
and  $\dist(C,\partial D_n) \simeq |I|$.
 Thus,  $  \rho_{\Omega_0}(B,C) \leq \rho_{D_n} (B,C)
=\frac 12 \log \frac 1{|I|}     +O(1).$ 

To prove the other direction, 
consider the arc on  $\partial B_n$ given by
 Corollary \ref{tangent lemma}.
Let  $\{z_j\}$ be the points 
where this arc intersects the vertical slits
in  $\partial\Omega_0$ and 
let $U = \complex \setminus \{z_j\}$. Then $\Omega_0 
\subset U$, so $\rho_{\Omega_0}(B,C) \geq \rho_U(B,C)$. 
The points $\{z_j\}$ are about $|I|$ apart and there are 
$m \simeq 1/|I|$ such points.
We can use a 
M\"{o}bius transformation to map  $B_n$
to the upper half-plane, and 
 so that $B$ and $C$ map to points at height 
$\simeq 1$ and $ \simeq |I|$, respectively.
The distance between the $\{z_j\}$ is distorted, 
but only boundedly, so the gaps are still $\simeq |I|$.
By  Lemma \ref{covering lemma} and conformal invariance 
of hyperbolic distances, we deduce
$$
\rho_{\Omega_0}(B,C) \geq 
\rho_{U}(B,C)  \geq  \frac 12 \log \frac 1{|I|} +O(1).
\qedhere
$$
\end{proof} 

\begin{lem}\label{lemma C D}
With notation as above,	$\rho_{\Omega_0}(C,D)=O(1)$.
\end{lem}
\begin{proof}
This is immediate from Koebe's estimate, 
 since $C$ and $D$ are connected 
by a segment in $\Omega_0$ whose Euclidean length 
is comparable to its distance from $\partial \Omega_0$.
\end{proof}

\begin{lem} \label{lemma D E} 
 $\rho_{\Omega_0}(D,E)= \frac \pi 2 \frac {|J|}{|I|} +O(1).$
\end{lem} 

\begin{proof}
This is immediate from Lemma \ref{rectangle bound}.
\end{proof}

\begin{lem} \label{dist to gamma}
Let $\gamma$ denote the hyperbolic geodesic from 
$B$ to $E$ in the R-component. Then
$\rho_{\Omega_0}(\gamma, D) =O(1) $ and 
$\rho_{\Omega_0}(\gamma, C) =O(1) $. 
\end{lem} 

\begin{proof}
The  claim for $D$ follows immediately 
from Lemma \ref{square lemma}. For $C$ it follows 
from this and Lemma \ref{lemma C D}.
\end{proof}

\begin{cor} \label{cor B E}
With notation as above,
\begin{equation}
\rho_{\Omega_0}(B,E)=\rho_{\Omega_0}(B,C)
+\rho_{\Omega_0}(C,D)+\rho_{\Omega_0}(D,E)+O(1)
\end{equation}
\end{cor}

\begin{proof}
Using the triangle inequality, we immediately get
``$\leq$''.
To obtain the opposite inequality, 
let $\gamma$ be the hyperbolic geodesic in $\Omega_0$
connecting $B$ and $E$ and let $z_C$, $ z_D$ be the points 
on $\gamma$ closest to $C$ and $D$ respectively. 
Then by Lemmas \ref{dist to gamma} and \ref{lemma C D}   
\begin{eqnarray*}
\rho_{\Omega_0}(B,E)
&=& \rho_{\Omega_0} (B,z_C)+\rho_{\Omega_0}(z_D,E)+
 \rho_{\Omega_0}(z_C,z_D)  \\
&\geq& \rho_{\Omega_0}(B,C)-\rho_{\Omega_0}(C,z_C) 
              +\rho_{\Omega_0}(D,E)-\rho_{\Omega_0}(D, z_D) \\
&&  \qquad + \rho_{\Omega_0}(z_C,C)  + \rho_{\Omega_0}(C,D)  + \rho_{\Omega_0}(D,z_D)  \\
&\geq & \rho_{\Omega_0}(B,C)+\rho_{\Omega_0}(C,D)
           +\rho_{\Omega_0}(D,E)  + O(1). \qedhere
\end{eqnarray*}
\end{proof} 

\begin{lem} \label{equal hyper dist} 
With notation as above, 
$\rho_{\Omega_0}(A,B) = \rho_{\Omega_0} (B,E) +O(1)$ .
\end{lem}

\begin{proof}
By Corollary \ref{cor B E} it is enough to prove 
$$
\rho_{\Omega_0}(A,B) = 
\rho_{\Omega_0}(B,C)+\rho_{\Omega_0} (C,D)
+\rho_{\Omega_0}(D,E)+O(1) .
$$
Note that $B=\Real z_0+N+t_n+\frac 12\Delta_n$
 for some $n$, $A=\Real z_0+N+2$, 
 and by Lemmas \ref{lemma A B}
 and \ref{adjacent estimates}
\[
\rho_{\Omega_0} (A,B)=\frac \pi 2(B-A)+O(1)=\frac \pi 2t_n+O(1).
\]
But by Lemmas \ref{lemma B C}, \ref{lemma C D},
 \ref{lemma D E}, and \ref{adjacent estimates}, we have
\begin{eqnarray*} 
\rho_{\Omega_0}(B,C)+\rho_{\Omega_0} (C,D)+\rho_{\Omega_0} (D,E)
&=&\frac 12\log\frac 1{|I|}+\frac{\pi}{2}\frac{|J|}{|I|}+O(1)\\
&=&-\frac 12\log\Delta_n+\frac \pi 2\frac{y_{n+1}}{\Delta_n}+O(1).
\end{eqnarray*}
Since $y_{n+1} = y_n + O(1)$, we get
\begin{eqnarray*}
&&\rho_{\Omega_0}(B,C)+\rho_{\Omega_0}(C,D)+\rho_{\Omega_0}(D,E) \\
&& \qquad  \qquad 
=-\frac 12\log\Delta_n+\frac \pi 2\frac{\Delta_{n}\left(t_{n}
	+\frac 1\pi\log\Delta_{n} +O(1) \right)}{\Delta_n}+O(1)\\
&&\qquad \qquad= \frac \pi 2t_n+O(1).
\qedhere
\end{eqnarray*}
\end{proof}

We defined  $\{y_n\}$ as we did
so that (\ref{t_n equality}) would hold, 
knowing that it would lead to Corollary 
\ref{equal hyper dist}, and thus to: 

\begin{lem} \label{vertex placement}
	The $\tau$-lengths for $T$ are all comparable to $1$.
\end{lem}

\begin{proof} 
Because of Lemma \ref{equal hyper dist} 
 and   Lemma \ref{verify tau cond}, 
the edges of the R-component shared with a neighboring 
L-component all have $\tau$-sizes comparable to $1$ 
(the $\tau$-lengths for the L-component are 
equal to $1$ by definition). 

Next, we will show the same  is true for 
the edges on the vertical slits. 
Consider a vertical  tower $R$ of width 
$\Delta_n$  in the modified R-component
between two adjacent vertical slits, see Figure 
\ref{Vertical}. Divide the tower into squares as shown.
Let $v_k$ be the center of the $k$-th square. By Lemma
\ref{rectangle bound}
the hyperbolic distance between $v_1$ and $v_k$ is
$\frac \pi 2k +O(1)$. Therefore, the image points 
in the right half-plane are the same hyperbolic distance apart.
We verified above that the image of the first 
square is $\simeq 1$, so the $k$-th square has 
image with diameter $\simeq \exp( k \pi)$.

\begin{figure}[htb]
\centerline{
\includegraphics[height=2.5in]{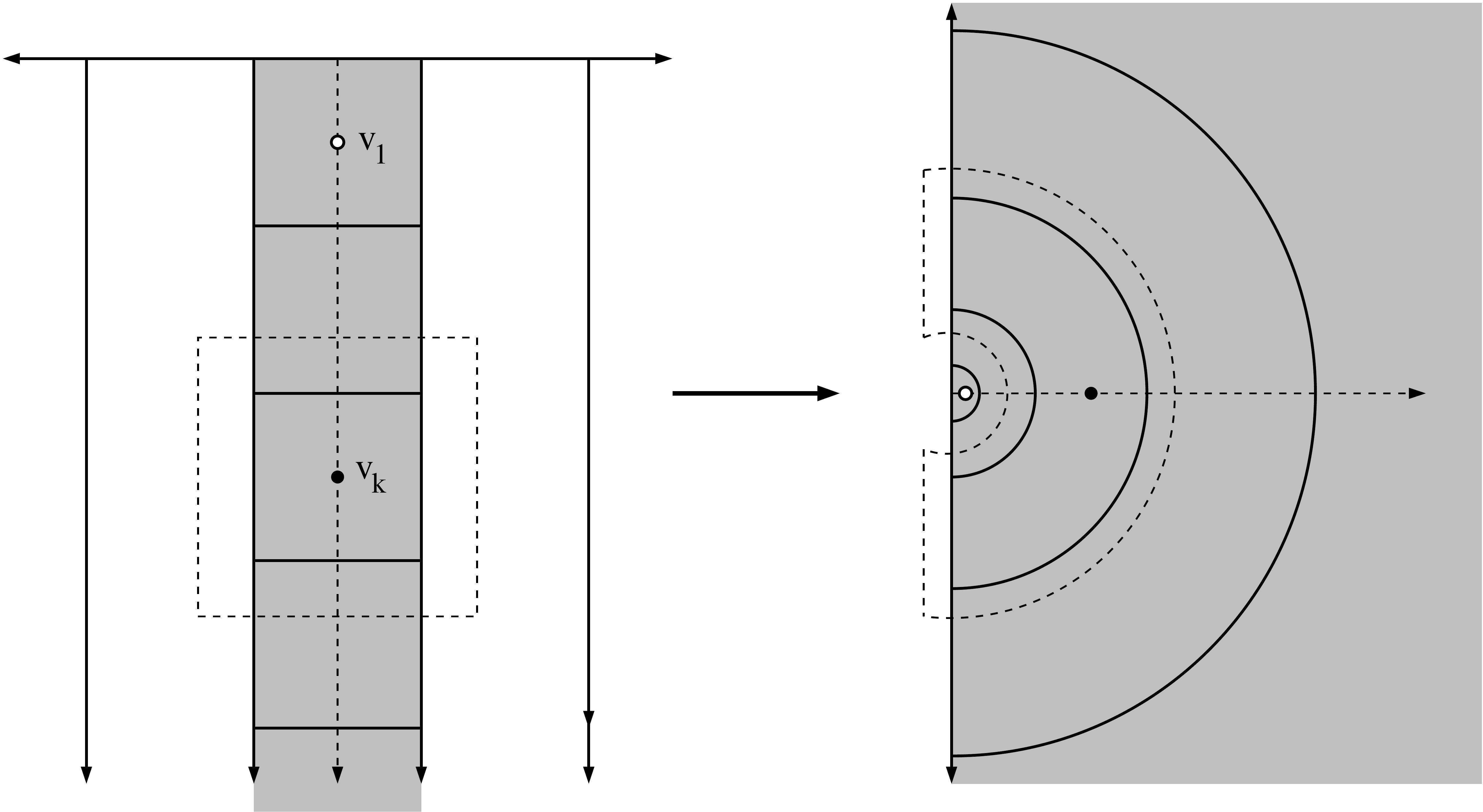}
}
\caption{ \label{Vertical}
Each tower is divided into squares, 
whose $\tau$-images are approximate 
half-annuli whose Euclidean
diameters grow exponentially. In particular, the 
$\tau$-lengths of the squares grow exponentially. 
	}
\end{figure}

Moreover, for each $k \geq 2$, the $k$-th square can
 be expanded by a factor of $3/2$ (dashed box in Figure
 \ref{Vertical}) and  the conformal map $\tau$ 
 can be extended to this larger box by
 reflection. Thus, by Koebe's estimate
$|\tau'|\simeq \exp(\pi k)/\Delta_n$ uniformly on
the $k$-th square. Since the side lengths of $T$ inside
the $k$-th square are $\simeq \Delta_n \exp(-\pi k)$ 
by construction, they all have $\tau$-lengths $\simeq 1$.

A slightly different argument is needed for vertices on the 
last segment of each vertical slit 
(the one near the tip), that were defined using a 
square root map. This was done precisely so that 
$\tau$ would map these edges of $T$ to intervals 
on $\partial \rhp$ that have comparable lengths to each
other, and the 
argument above shows these lengths are all $\simeq 1$.
\end{proof}

We have proven  uniform lower and upper bounds for the
 $\tau$-sizes of all edges. By  replacing  $\tau$ by $c\cdot\tau$  
 for some $c>0$, if necessary,  we  may assume every 
$\tau$-length is at least $\pi$ and at most $O(1)$.
The previous lemma,  Lemma \ref{lem:T(r)_bound} 
and Lemma \ref{verify T(r) bound}  imply:

\begin{cor} \label{tau image}
With  $T$,   $\tau$ and $r$ as above, 
there is a $M < \infty$ (independent of $N$)
 so that 
$  \tau(T(2r))\subset\{x+iy:0< x \leq M\}$.
\end{cor}

\section{ $T(r)$ has finite  area} \label{T(r) finite area}

The $\tau$-length bounds derived above give good control of
the model function $F$. We now start to  control the 
quasiconformal correction map $\varphi$. Recall that 
$T(r)$ is an open neighborhood of the tree $T$ defined
in Section \ref{review sec}. Since $\varphi$ is 
conformal off $T(r)$, its dilatation $\mu$ is supported on $T(r)$, 
and we will prove this set is ``small''. We  begin
by showing it has finite area. 

\begin{lem}\label{lem:sum_edges_in_tower}
Let $n\geq 2$ and consider the tower attached to the
 edge of length $\Delta_n$. Let $\{J_k\}$ be the 
collection of edges on the vertical sides of the 
tower. Then there exists a constant $C>0$, 
independent of $N$ and $n$, so that
for any $0 < \delta \leq 1$ we have
\[
\sum_k (\diam J_k)^{1+\delta}\leq  \frac C \delta
\cdot\Delta_n^{1+\delta}.
\]
\end{lem}

\begin{proof}
In the construction, we 
subdivided the $n$-th ``branch''  
into $m_n=\lfloor y_n/\Delta_n\rfloor$
many subintervals $\{J_k\}$. All but one 
of these (the ``tip'')  has length  $\Delta_n$; 
consider these first.
The  $k$-th one is  divided
into $\simeq \exp(\pi k)$  smaller 
edges $\{J_{j,k}\}$ of equal 
 length $\simeq \Delta_n \exp(-\pi k)$. Summing 
over these gives 
 \begin{eqnarray*}
\sum_{j,k}(\diam J_{j,k})^{1+\delta}
&\leq&  \sum_{k=1}^{m_n-1} \exp(\pi k) 
      (\Delta_n \exp(-\pi k))^{1+\delta} \\
& \leq & \Delta_n^{1+\delta}\cdot 
          \sum_{k=0}^{\infty} \exp(-\pi k \delta)\\
& =& \Delta_n^{1+\delta} \cdot O\left(\frac 1 \delta\right). 
\end{eqnarray*}

The last  interval (i.e., the ``tip'') 
has length $\leq 2 \Delta_n$ and 
is divided into $\simeq \exp(\pi m_n)$ intervals $J_k$, the 
longest of which  has length
 $\simeq   \Delta_n \exp(-\pi m_n/2)$.
Thus 
\begin{eqnarray*}
 \sum \diam(J_k)^{1+\delta} 
  &\leq& (\Delta_n \exp(- \pi m_n/2))^\delta \sum \diam(J_k) \\
 &\leq & (\Delta_n \exp(- \pi m_n/2))^\delta (2 \Delta_n) \\
  &\leq& 2 \cdot \Delta_n^{1+\delta}.
\end{eqnarray*}
Combining both cases we get  the lemma.
\end{proof}

\begin{lem}\label{lem:T(r)_has_finite_Lebesgue_area}
For each $R$-component $\Omega$, the Lebesgue area of
 $\Omega\cap T(r)$ is finite with a bound independent of $N$.
\end{lem}

\begin{proof}
The part of $T(r)$ outside the towers is easy to bound.
Consider what happens in each tower with base $I_k$,
which has length $\Delta_k$. The points of $T(r)$ 
in the $k$-th tower that are between distance $j\Delta_k$
and $(j+1)\Delta_k$ from the $L$-component are within
$O(r  \Delta_k \exp(-j))$ of  a vertical side of the
tower. Thus, the part of $T(r)$ inside this part of
the tower has area bounded by $O(r  \Delta_k^2 \exp(-j))$.
Summing over $j$ shows the total area of $T(r)$ within
each tower is $O(r\Delta_k^2)$. By  
Corollary \ref{cor:delta_n_powers} we have  
$\sum_k \Delta_k^2 < \infty$, so  the lemma follows, see Figure \ref{TRinTower}.
\end{proof} 

\begin{figure}[htb]
\centerline{
\includegraphics[height=2.0in]{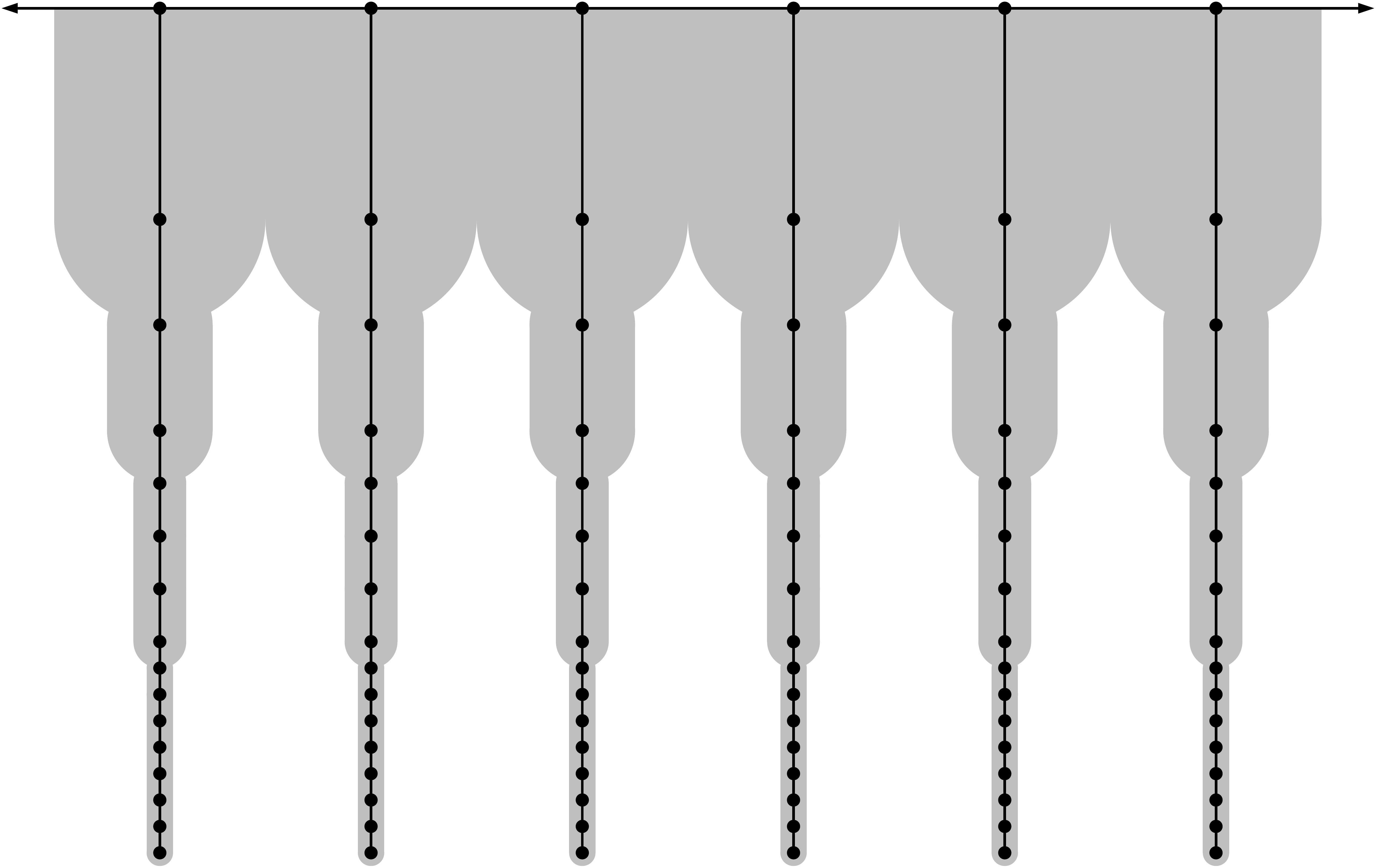}
}
\caption{ \label{TRinTower}
The neighborhood $T(r)$ narrows exponentially in each tower 
as we move away from the $L$-component, so the total area 
of $T(r)$ in each R-component is finite (and bounded independent of $N$).
}
\end{figure}

It follows that the total area of $T(r)$ is $O(N)$.

\section{ Logarithmic area estimates  for $T(r)$} \label{area ests}

 The logarithmic area of 
a planar set $E$ with respect to a point $w$ is 
$$ \int_E \frac {dxdy}{|z-w|^2}.$$
We will show the  logarithmic area of $T(r)$ 
is uniformly bounded for any base point 
outside $T(2r)$.   
This will give a bi-Lipschitz estimate 
on $\varphi$ in the next section.

\begin{lem} \label{area of T}
Suppose $r>0$ and suppose
   $ w \in \Omega\setminus T(2r)$. Then 
\[
\int_{T(r)}\frac {dxdy}{|w-z|^2} < C < \infty,
\]
where $C$ is independent of  $w$ and  $N$. 
\end{lem} 
\begin{proof}
Let $r_0 = \dist(w, \partial T(2r))$ and $D = D(w, r_0)$.
Let $r_n = r_0 2^n$ and set $A_n(w) =\{z~:~r_{n-1} \leq  |z-w|\leq 
 r_n \}$ (we will drop the $w$ when the center of the 
annulus  is clear from context). 
Note that  $T(r)=\bigcup_{n=1}^\infty T(r)\cap A_n$.
Furthermore, for  $z\in A_n$ we have 
$|w-z|^2\geq r_{n-1}^2\simeq\area(A_n)$. 
 Thus it suffices  to show 
\begin{eqnarray} \label{annulus sum}
\sum_{n=1}^\infty \frac {\area(T(r)\cap A_n)}{\area(A_n)} < \infty, 
\end{eqnarray} 
with a bound that is independent of $w$ and  $N$. 
Note that each term in the series is bounded by $1$.
First assume that $w$ is located in one of the ``towers''
along the horizontal edges of the R-component $\Omega$ 
(the remaining case will be dealt with later). 
Suppose this tower has top edge $I$ (the edge shared
 with a the adjacent L-component) and shorter vertical side $J$.
It is convenient to break the sum into four parts:
\begin{enumerate}
\item $r_n  < |I|$, 
\item $ |I| \leq r_n  < |J|$, 
\item $ |J| \leq r_n  < 1$, 
\item $1 \leq r_n < \infty$.
\end{enumerate} 
In each case, the series  will be  dominated by a geometric series,
whose sum is dominated by its largest term, and this will be
$O(1)$ in every case.

\textbf{Case 1 ($r_n < |I|$):} 
See Figure \ref{Annuli1}. Suppose $w$ is distance $y$ from the top
 of the tower. The annulus $A_n$ hits at most one of the 
vertical segments in $\partial\Omega$ and this intersection is
contained in a rectangle of height $O(r_n)$ and width $O\left(r e^{-y/|I|}\right)$, 
and thus has area at most $O\left( r_n r e^{-y/|I|}\right)$. 
Also, since $D$ is outside $T(2r)$, if $A_n$ intersects $T(r)$ then  $r_n$ 
must be at least comparable to $ r e^{-y/|I|}$.  
 Thus, the non-zero terms of the sum satisfy
\[
\frac{\area( T(r)\cap A_n)}{\area(A_n)}\leq O\left(\frac{r_n  r e^{-y/|I|} } {r_n^2 }\right)
 =  O\left(\frac {r e^{-y/|I|}}{r_n}\right).
\]
These form a  decreasing geometric series whose 
largest term is $ O(1)$. Thus the sum over Case 1 
annuli is bounded by a  constant independent of  
$w$ and $N$.

\begin{figure}[htb]
\centerline{
\includegraphics[height=2.0in]{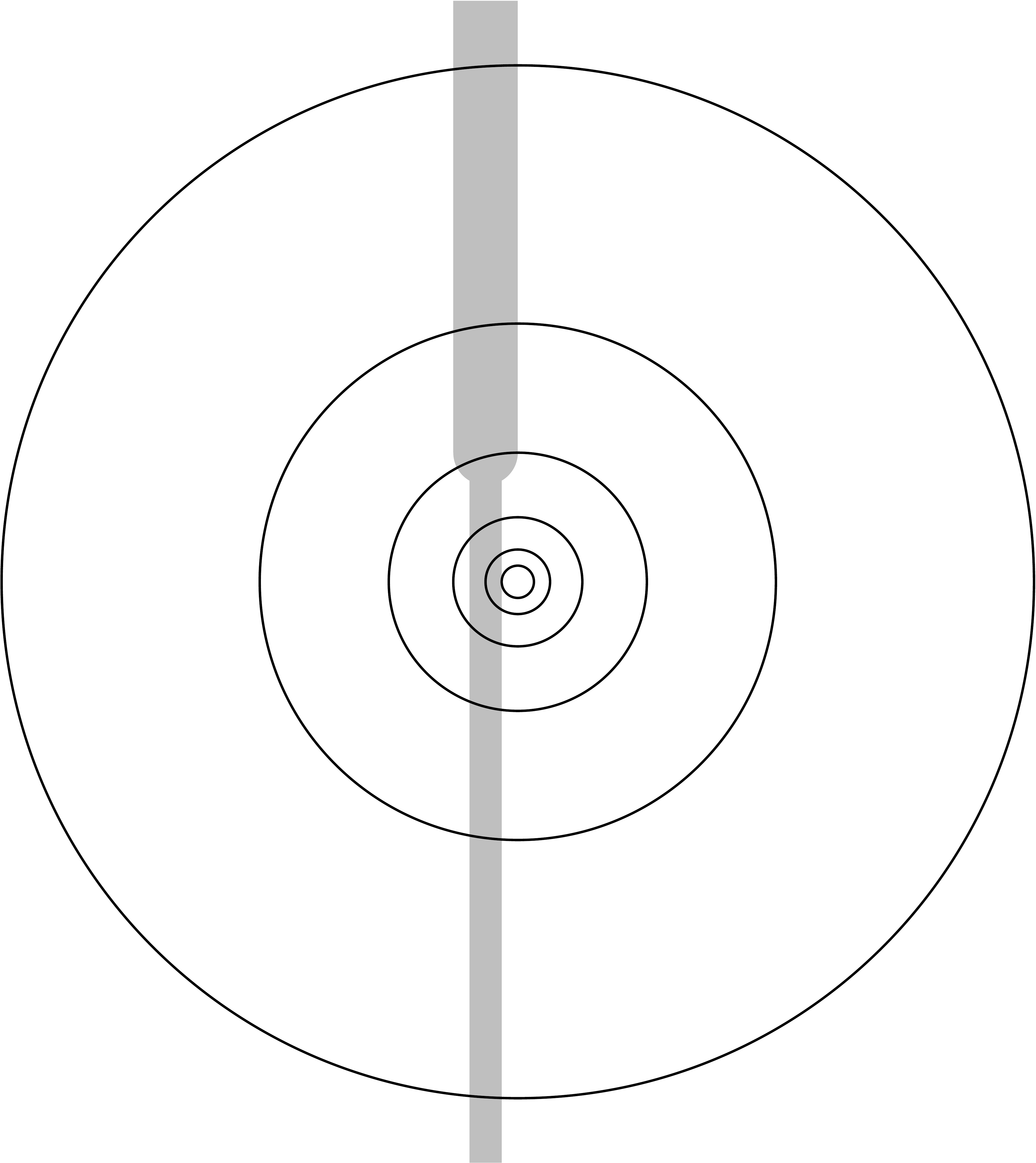}
}
\caption{ \label{Annuli1}
Case 1: annuli hit at most one vertical segment.
}
\end{figure}

\textbf{Case 2  ($|I| \leq r_n < |J|$):}
Consider horizontal strips of width $|I|$ as shown in  
Figure \ref{Annuli2}, and number them $k=1, \dots , m \simeq |J|/|I|$
 starting from the top side of $\Omega$ and going down (this is assuming
 we are working near the top edge of $\Omega$; an identical
 argument works along the bottom edge). Recall that the interval of length
 $|I|$ which is $k|I|$ from the top edge was subdivided into 
$\exp( \pi k)$ pieces. 
By Lemma \ref{vertex placement}, the $\tau$-lengths of all these pieces 
are comparable to $1$. We thus obtain that in the $k^{\text{th}}$ strip
  $S_k$ the width of  $T(r)$ around each vertical segment in $\partial\Omega$
 is $O\left(\exp(-\pi k)\right)$.
 The area of $T(r)$ in each annulus is bounded by
 the area in the concentric axis-parallel square, whose side lengths equal
 the outer diameter of the annulus. In each such square $Q$, the area of 
$T(r) \cap Q \cap S_k$ decays geometrically with $k$ and hence is bounded
 by a multiple of the area in the top strip. Thus, 
\[
\frac{\area(T(r)\cap A_n)}{\area(A_n)}\leq  Me^{-\pi k}
\]
where $k$ is the smallest index such that $A_n$ hits the strip $S_k$.
 Since $r_n \geq |I|$, this index strictly decreases each time we increment 
$n$ and hence summing over all $n$ in this case gives a sub-sum of a 
geometric sum whose largest term (corresponding to $k=1$) is $O(1)$. 
Thus, the Case 2  terms  sum to $O(1)$ with a  uniform constant.

\begin{figure}[htb]
\centerline{
\includegraphics[height=2.5in]{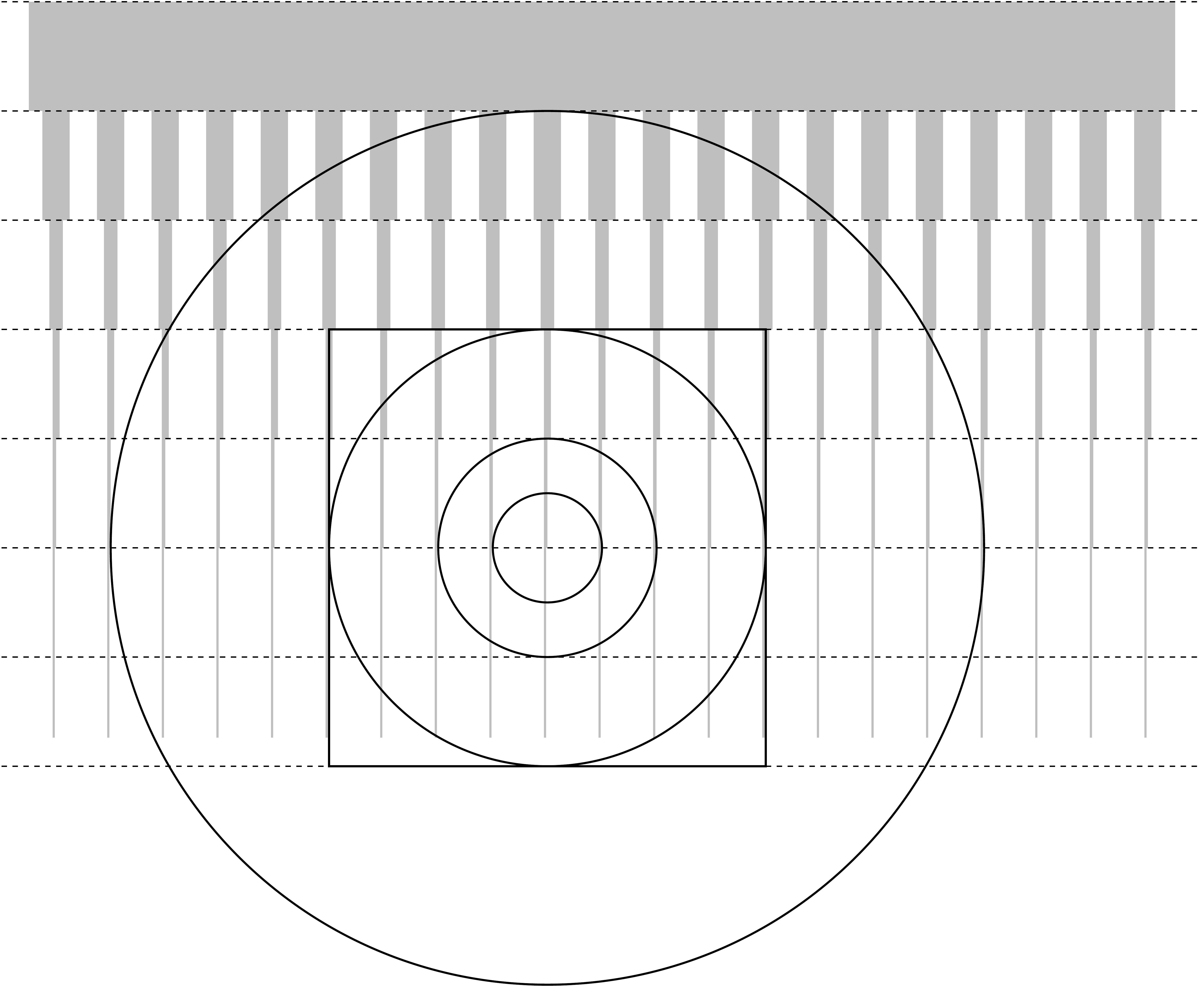}
}
\caption{ \label{Annuli2}
 Case 2: Annuli hit two vertical segments, but not the 
  neighboring L-component.
}
\end{figure}

\textbf{Case 3 ($|J| \leq r_n < 1$):}
As in Case 2, we overestimate the area of $T(r)\cap A_n$ 
by replacing the annulus by a square $Q$, and estimating
the area of $T(r)\cap Q$ by considering the horizontal 
strips $\{S_k\}$.  See Figure \ref{Annuli3}.
Note that annuli in this case have diameters that are at
least comparable to $|I|$ (if $w  \in S_1$, the ``top''
strip, defined in Case 2, then this is true 
since $w \not \in T(2 r)$, and otherwise the annulus
must be large enough to touch the neighboring L-component, and hence
has diameter at least $|I|$).
As before, the area is dominated by a multiple of the
area of $Q \cap S_1$, which is bounded by a 
multiple of $ |I| \diam(Q)$ (there are approximately 
$\diam(Q)/|I|$ towers hitting the annulus and each 
contributes area approximately $|I|^2$). Thus
\[
	\frac{\area(T(r)\cap A_n)}{\area(A_n)}
	\leq\frac{M|I| \diam(Q)}{r_n^2} 
	= O\left( \frac {|I|}{r_n}\right)
\]
In this case,  $r_n$ starts by being at least as large as
$|J|\geq |I|$ (and possibly much larger), and doubles
at each step, so the terms of this sub-series are geometrically
decaying. Thus, the sum is bounded by its first term,
which is at most $O(1)$ with constants again being 
independent of $w$ and  $N$. 

\begin{figure}[htb]
\centerline{
\includegraphics[height=2.5in]{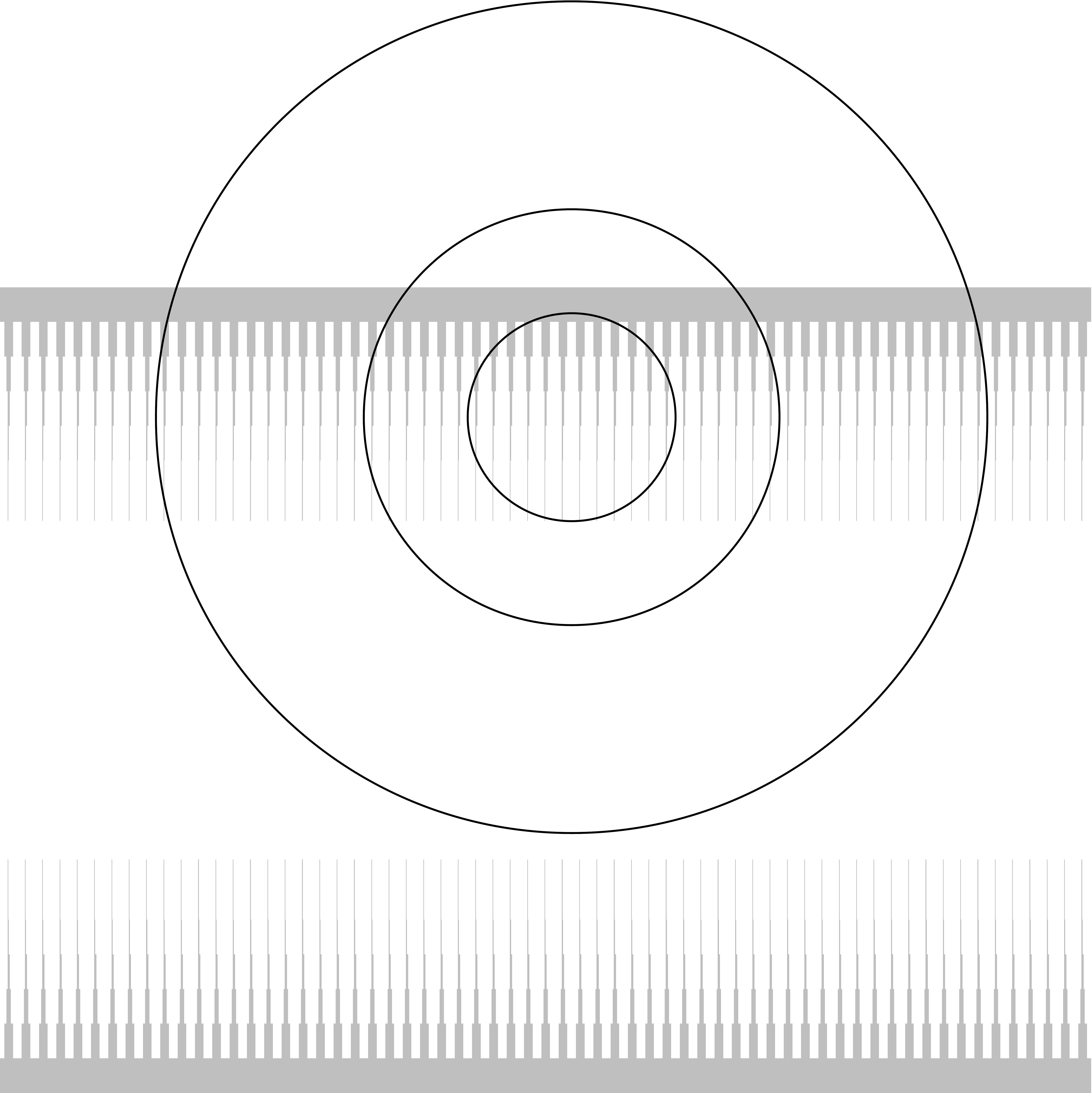}
}
\caption{ \label{Annuli3}
Case 3: annuli hit one, but not two, neighboring 
L-components. 
}
\end{figure}

\textbf{Case 4 ($1\leq r_n <  \infty$):} 
By Lemma \ref{lem:T(r)_has_finite_Lebesgue_area}, 
$\Omega \cap T(r)$ has 
area $O(1)$, independent of $N$, for each $R$-component $\Omega$. 
Because of the way the R-components are arranged
in the plane, a disk of radius $r_0$ can hit 
at most $O(r_0)$ of them. 
Thus the annulus $A_n$  can intersect at most $O(r_n)$ 
different $R$-components, so
$\area(A_n \cap T(r)) = O(r_n)$.  
Since $\area(A_n) \simeq r_n^2$, the ratio of 
these areas is $O(1/r_n)$. 
Since $r_n$ increases geometrically, 
the Case 4 terms are bounded above  by  a geometrically decreasing 
series  with first term of size $O(1)$, see Figure \ref{Case4Fig}.
Thus the sum over this case is also $O(1)$.

\begin{figure}[htb]
\centerline{
\includegraphics[height=2.0in]{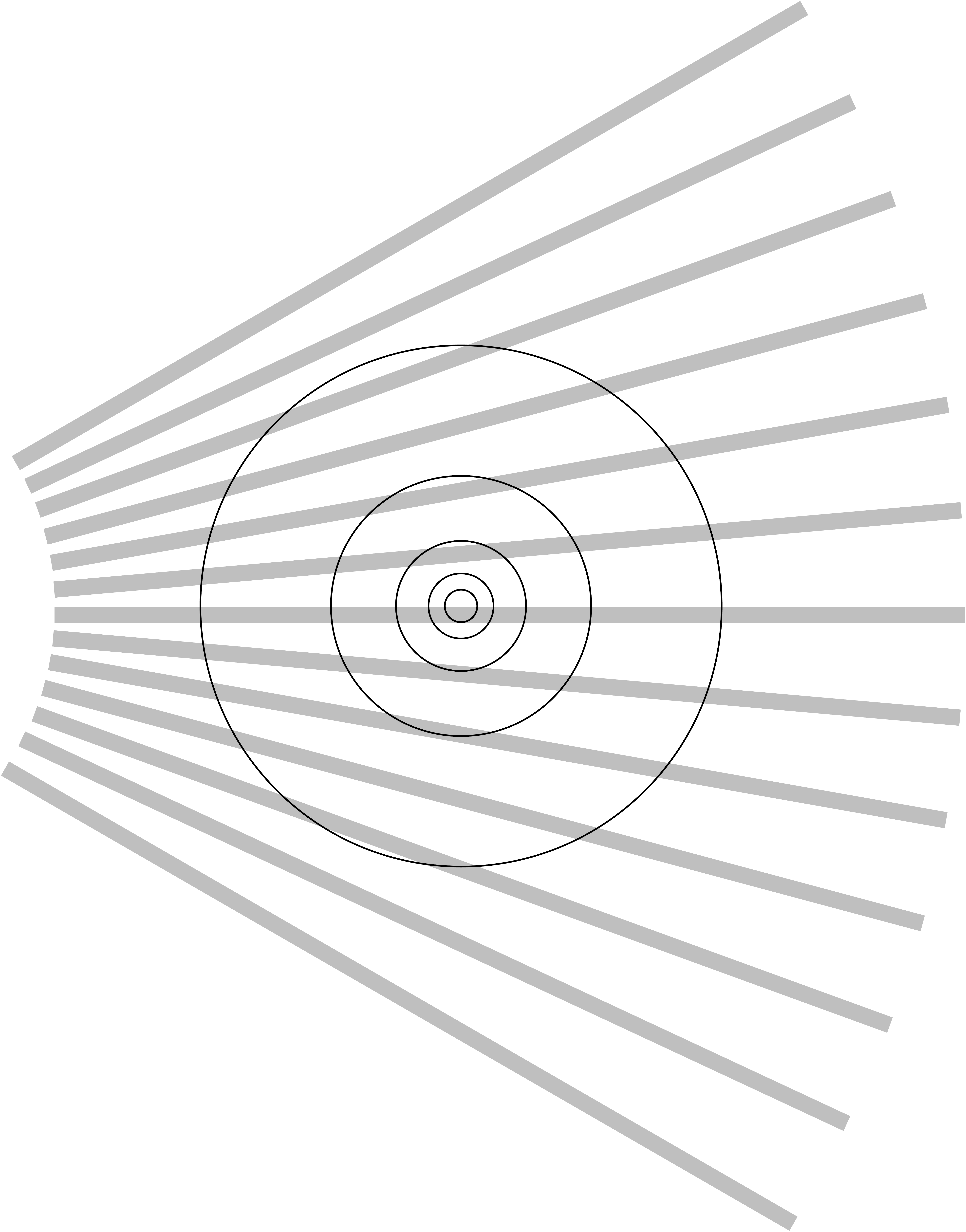}
}
\caption{ \label{Case4Fig}
Case 4:  Annuli with diameters  $r_n \geq 1$
can hit at most  $\max(N, O(r_n))$ different $R$-components, 
each contributing area at most $O(1)$. Thus the Case 4 
terms are bounded by decreasing  geometric series with largest 
term $O(1)$.
}
\end{figure}

The four cases given above describe the proof of the lemma
 when the point $w$ is in one of the towers. If $w$
 is not in a tower, but is in the region between the
 towers and the real axis, then a similar proof by 
 enumerating cases will work. However, it is probably 
easier  to argue as follows. For a $w$ that is not in
 a tower, choose $k$ so that $A_k$ 
 is the smallest annulus of the form
 $A_j(w)$ that hits $T(r)$. Then it is easy to check 
that $A_{k+m}(w)$ contains a point $z$ in a tower and 
outside $T(2r)$ for some number $m$ that is independent
 of $w$. Let $s_0=\dist(z,\partial T(2r))$ and $s_l=s^ls_0$.
 Let $n>k+m+2$. Then there exists some $l$ so that 
$r_n\leq s_l<r_{n+1}$. We then have
\[
A_n(w)\subset A_{l-2}(z)\cup A_{l-1}(z)\cup A_l(z)\cup A_{l+1}(z).
\]
Thus, the series (\ref{annulus sum}) for $w$ is bounded
above by  $O(1)$ terms, each bounded by $1$, plus four
times the corresponding series for $z$. By our previous
argument, the latter is uniformly bounded and hence so 
is the series for $w$.
\end{proof}

\section{A bi-Lipschitz estimate for the correction map}
\label{bi-Lip sec}

Recall from the statement of Theorem \ref{thm:folding} 
that our entire function is of the form 
$f = \sigma \circ \tau \circ \varphi^{-1}$.  Thus inside 
the R-components,  we have 
$ f^{-1} = \varphi \circ \tau^{-1} \circ \log$. 
The logarithm is easy to understand, the map $\tau^{-1}$ 
is conformal on a half-plane and the estimates we need
for it are all well known. Only the quasiconformal map 
$\varphi$ holds some mystery, but we will show that it is
completely harmless, at least for estimating dimension,
because it is bi-Lipschitz
with a uniform constant near the Julia set (however, it need not
be bi-Lipschitz everywhere in $\complex$). 

\begin{lem}\label{lem:bi-Lipschitz}
Suppose $A ,B$ are disjoint, planar sets and
\[
\int_A \frac {dxdy}{|z-w|^2} \leq C  < \infty,
\]
for all $w \in B$. If $\varphi$ is a $K$-quasiconformal map
that is conformal off $A$, then $\varphi$ is $M$-bi-Lipschitz 
on $B$ with $M$ depending only on $C$ and $K$, i.e., 
for all $w,z \in B$,
\[
0 < \frac {1}{M(C,K)} \leq  \frac {|\varphi(z)-\varphi(w)|}{|z-w|} 
\leq M(C,K) < \infty.\]
 
\end{lem}

\begin{proof}
This is more-or-less immediate from results of  
Bojarski, Lehto,  Teichm{\"u}ller and Wittich 
\cite{MR0407275}, \cite{MR0344463}, \cite{Teichmuller38},
\cite{MR0027057} although  we shall give specific 
references to the more recent paper \cite{MR1980177}
which also gives the higher dimensional versions of 
the two dimensional results we will use.

First we prove that $\varphi$ is asymptotically conformal
at $\infty$. Let $w\in B$. Denote by $\mu(z)$ the 
dilatation of $\varphi$. This function is supported
on $A$. Thus, we get
\[
\int_{|z|>|w|}\frac{|\mu(z)|dxdy}{|z|^2}\leq\int_{A\cap\{|z|>|w|\}}\frac{dxdy}{|z|^2}\leq\int_A\frac{4dxdy}{|z-w|^2}\leq 4C.
\]
Hence, (see for example \cite[Chapter V, Theorem 6.1]{MR0344463}) there is a $c\neq 0$ so that
\[
\lim_{z\to\infty}\frac{\varphi(z)}{z}=c.
\]

Now suppose $z,w \in B$ and let $r=|z-w|$. Note that
\[
\{\xi: |\xi-z| = R \} \subset\{ \xi: R-|z| \leq |\xi| \leq R+|z| \}.
\]
If $R$ is large enough,  $\varphi$ maps the round annulus $A(z,r,R) = \{ \xi: r < |\xi-z| < R\}$ to a topological annulus $A'$  whose outer boundary is contained in the  annulus $A(\varphi(z),|c|R/2, 2|c|R)$ and whose inner boundary is a closed Jordan curve $\gamma$. By taking $R$ large enough, we can assume $\gamma$ hits the disk $D(\varphi(z), |c|R/4)$. Therefore,
\[
\mathrm{mod}(A') = \mathrm{mod}(A(\varphi(z),\diam(\gamma), R)) +O(1)
= \log R - \log \diam(\gamma) +O(1).
\]
On the other hand, Corollary 2.10 of \cite{MR1980177} says that 
\begin{align*}
\mathrm{mod}(A')&=\mathrm{mod}(A(z,r,R))+O\left(\int_{r< |\xi-z|< R}\frac{|\mu(\xi)|}{|\xi|^2} dxdy\right)\\
&=\mathrm{mod}(A(z,r,R))+O\left(\int_{A\cap\{r<|\xi-z|< R\}}\frac{1}{|\xi|^2} dxdy\right)\\
&=\log R-\log r+O(1).
\end{align*}
Thus, $ \log \diam(\gamma) = \log r + O(1)$,
or $\diam(\gamma) \simeq r$.
Since $\varphi$ is quasiconformal, the segment $S$
connecting $z$ and $w$ maps to a quasi-arc and hence
satisfies the Ahlfors three-point condition (e.g., 
Theorem II.8.6 of \cite{MR0344463}), so  
$|\varphi(z)-\varphi(w)| \simeq \diam(\varphi(S))$.
Since  quasiconformal maps are quasisymmetric
$\diam(\varphi(S)) \simeq \diam(\gamma)$ (first 
due to Gehring \cite{MR0124488};  see also 
Section 4 of \cite{MR1654771}).
Thus, $ |\varphi(z)-\varphi(w)| \simeq |z-w|, $ as desired.
\end{proof}

\begin{lem} \label{correction map bound}
The correction map $\varphi$ is bi-Lipschitz on $\complex\setminus T(2r)$,
that is,  there exists a constant $M>0$  (independent
of $N$) such that
\[
\frac 1M\leq\frac{|\varphi(z)-\varphi(w)|}{|z-w|}\leq M
\]
for all $z,w\in\complex\setminus T(2r)$.
\end{lem}

\begin{proof}
Define $A=T(r)$ and $B=\complex\setminus T(2r)$. Let $w\in B$.
 By Lemma \ref{area of T}, there exists some $C>0$, independent of $N$ and $w$, so that
\[
\int_A\frac{dxdy}{|z-w|^2}<C<\infty.
\]
Applying Lemma \ref{lem:bi-Lipschitz} proves this lemma.
\end{proof}

\section{Proof of (\ref{initial sum}): the initial covering exists}
\label{prove 1.1}
We have now completed the preliminary estimates. In 
the next two sections we finish the proof of Theorem
\ref{dno thm} by establishing (\ref{initial sum}) and
(\ref{main est}) from the introduction.
One of the main facts we need is the following result.
\begin{lem}\label{lem:diameters of disks}
Let $\Omega\neq\complex$ be simply connected and let
 $\tau:\Omega \to \rhp$ be conformal. Let 
$z=x+iy$ with $x>1$ and let $V\subset\rhp$ 
be a simply connected neigbourhood
 of $z$ with hyperbolic radius bounded by $r$.
If $w=\tau^{-1}(1+iy)$ and $U=\tau^{-1}(V)$, then we  have
\[
\diam (U)=O(|\tau'(w)|^{-1}x\diam (V))
\]
where the constant depends only on $r$ and 
	the diameter is the Euclidean diameter.
\end{lem}

\begin{proof}
The distortion theorem for conformal maps (e.g., 
Theorem I.4.5 of \cite{MR2450237}) says that if $\psi:\disk 
\to  \Omega$ is conformal,  then
\begin{eqnarray} \label{disk case} 
|\psi'(z)| \leq  |\psi'(0)| \frac {1+|z|}{(1-|z|)^3}
\leq  |\psi'(0)| \frac {2}{(1-|z|)^3}.
\end{eqnarray}
Moreover, the derivative of a conformal map has comparable
absolute values at any two points of a compact set, with a 
constant that depends only on the hyperbolic diameter of the set. 
Thus  for any compact $K$ in the disk $\disk$, 
$$ \diam(\psi(K)) = O(\diam(K)|\psi'(z)|),$$ 
where these are Euclidean diameters and 
where  $z$ is any point of $K$. The constant depends
only on the hyperbolic diameter of $K$ and 
	not on $z$ or $\psi$.

If $\Psi$ is a conformal map from the right half-plane
$\rhp$ 
to $\Omega$, then we can write it as a composition
of the M{\"o}bius transformation $\sigma(z) =(z-1)/(z+1)$ 
from the half-plane to the disk, followed by a conformal
map $\psi$ from the disk to $\Omega$. Note that
$\sigma(1) =0$,  $\sigma(\infty) =1$,
and $|\sigma'(x)| = 2/(x+1)^2 \leq 2/x^2$ for $x \in
[1,\infty)$.  
Also note that  $1-\sigma(x) 
 = \frac 2{x+1} \geq  1/x$ for $x \geq 1$.
 Using (\ref{disk case}), these observations, and 
 the chain rule, we deduce that 
\begin{eqnarray*}
| \Psi'(x)| 
&=&   |\psi'(\sigma(x))| \cdot |\sigma'(x)| \\
&\leq &  |\psi'(0)| \frac 2{(1-|\sigma(x)|)^3} \cdot \frac 2{x^2} \\
&\leq &   4\frac {|\Psi'(1)|}{|\sigma'(1)|}  \frac {x^3}{x^2} \\
&= &   8 |\Psi'(1)|x ,
\end{eqnarray*} 
for  $x \geq 1$. By considering a vertical translate
$\tilde{\Psi}(z)=\Psi(z+iy)$ of $\Psi$ and applying
the arguments above to $\tilde{\Psi}$ we see that
\[
| \Psi'(x+iy)| = O( |\Psi'(1+iy)| \cdot x),
\]
for  $x \geq 1$.  
The lemma follows by applying this
discussion to $\Psi = \tau^{-1}$.
\end{proof}

\begin{lem} \label{tracts in strips}
If $N$ is large enough, then 
the Julia set for $f$ is contained in the union of $N$ 
bounded width 
strips passing through the origin. More precisely, 
there are a $C,M < \infty$ (independent of $N$) so that
$$ \Julia(f) \subset
\{ z\in \complex: |z| >  N/(2C),\,\dist(z,X) \leq M \}
$$
where $X = \{ w: w^N \in \reals\}$. 
\end{lem}

\begin{proof} 
Our  entire functions are of the form $f = g \circ \varphi^{-1}$ 
where $g$ is quasiregular and $\varphi$ is a quasiconformal
 homeomorphism of the plane. Both $g$ and $\varphi$
 were chosen to fix the origin and $g$ maps the disk
 $D(0, N)$ into the unit disk. 
Lemma \ref{correction map bound} shows that $\varphi$ 
is bi-Lipschitz  with some constant $C < \infty$ in 
 $D(0,N/2)$, since this disk  is contained inside 
the central D-component but outside $T(2r)$, if $N$
 is large enough. Thus $\varphi^{-1}$ maps $D(0, N/(2C))$
 into $D(0, N/2)$. Hence
  $f$ maps $D(0, N/(2C))$ into the unit 
disk, and hence into itself for $N$ large enough.
 Therefore, $ D = D(0, N/(2C))$  is inside the Fatou
 set of $f$.  In particular,
\[
\Julia(f)\subset\{ z: |f(z)|> N/(2C) \})\subset 
\varphi(  \{ z: |g(z)|> N/(2C)\}).
\]

If $N > 2C$, then by construction $\{ z: |g(z)|>N/(2C) \}
\subset \{ z : |g(z)| > 1 \}$ lies in the union of 
R-components. In particular, $|g(z)|\leq 1$ for all
points outside the R-components. Also by construction,
$g$ maps $X$ to itself and the union of R-components
is contained in a unit neighborhood of $X \setminus D(0, N)$.
Moreover, we claim that 
 $\varphi$ fixes each arm of the set $X$.

One way to verify this is to construct $\varphi$ as follows.
Use the measurable Riemann mapping theorem to find
a quasiconformal  map from the sector $\{w: |\arg(w)| < \pi/2N \}$
with the same dilatation  and fixing $0$ and $\infty$;
then extend the map to the whole plane by reflecting
across the arms of $X$ one at a time. We end with a
quasiconformal map of the plane and the correct 
dilatation, so it must be $\varphi$ (up to a positive 
dilation). Moreover, it preserves the arms of $X$ by its definition,
proving the claim. 

Since $\varphi$  preserves the arms of $X$ and it is
$C$-bi-Lipschitz on the part of the arms of 
$X$ that lie outside $T(2r)$ (this is all but a bounded 
segment), and since it is quasisymmetric on the whole
plane, there is a  $M< \infty$ so that
\[
\varphi(\{ z: \dist(z, X) \leq 1 \}) \subset
	\{ z: \dist(z,X) \leq M\}).
 \]
In particular 
\begin{eqnarray*}
 \Julia(f) 
&\subset & \varphi(\{z:|g(z)|> 1\})  \\
&\subset& \varphi(\{z: |z|> N/2,  \dist(z, X) \leq 1 \})  \\
&\subset& \{z: |z|> N/(2C),  \dist(z, X) \leq M \}.
\qedhere
\end{eqnarray*}
\end{proof}

\begin{proof}[Proof of (\ref{initial sum})]
The  Julia set is contained in   the union of 
$2N$ half-strips of fixed width $M$.
 Due to the symmetries in the construction,
 it will  suffice  to consider  coverings of 
 only one of these, say  $S_0$, the half-strip 
intersecting the positive real axis. 
This half-strip can  be covered by a collection 
 squares $\{Q_n\}$  of side length $M$ and centered at points 
$n=\{\lfloor N/(2M) \rfloor, 1+\lfloor N/(2M) \rfloor, \dots \}$.
 Obviously, the diameters of these squares are 
all about unit size, hence they  are not summable.
However, we shall show that  given $\delta>0$,
 we can choose $N$ large enough so that
\[
\sum \diam(f^{-1}(Q_n))^{1+\delta} < \infty,
\]
where the sum is over all preimages outside $D(0,r_N)$ 
(we need only consider preimages hitting the Julia set 
and this disk is in the Fatou set).

Note that  $f^{-1} = \varphi \circ g^{-1}$. 
There are $2N$ R-components, and preimages under $g$ can 
lie in any of these,  but for the moment we only 
consider $g$-preimages in the R-component that 
intersects the positive real axis.  
Off the set $T(r)$, 
$g(z) = \exp(\tau(z))$  where $\tau:\Omega_0 \to \rhp$ is conformal.
Therefore,  $g^{-1}(Q)$ for 
 $Q\subset\rhp\setminus\tau(T(r))$,  
is given by first taking inverse images under the
 exponential map, then under the conformal map $\tau$. 
Both are easy to understand.

The inverse image of $Q_n$ under $e^z$ is a countable
 union of sets of diameter $O(1/n)$. The sets are all
 vertical translates of each other and there is exactly
 one containing each point of the form 
$z_{n,k}= \log n +  k 2\pi i$. Suppose $Y_{n,k}$ 
is the preimage of $Q_n$  containing $z_{n,k}$. 
Let $P_{n,k} = \tau^{-1}(Y_{n,k})$, 
and $w_{k} = \tau^{-1}(1+ 2\pi k i)$.

\begin{lem}
	$\diam(P_{n,k}) = O(|\tau'(w_{k}) |^{-1}\cdot \frac {\log n}{n})$.
\end{lem} 

\begin{proof}
This follows immediately from Lemma \ref{lem:diameters of disks}, since $Y_{n,k}$ has diameter $O(1/n)$ and hits the vertical line $\{ x = \log n\}$.
\end{proof}

Each of the points $z_{n,k}$ is associated, by horizontal 
projection, to one of the intervals $[2\pi n i, 2 \pi (n+1)i]$
on $\partial \rhp$ and each point $w_{n,k}$
is associated to  an edge of the folded graph $T'$. Each edge of 
$T'$ is associated to a side $I_k$ of the R-component 
and at most $O(1)$  edges of
$T'$ correspond to any such side of the $R$-component (this is 
because the $\tau$-lengths of all sides are  uniformly bounded).
 In particular, for any $0 < \delta \leq 1$, 
\[
|\tau'(w_{n,k})|  \simeq \frac 1{\diam(I_k)},
\]
\[
\diam(P_{n,k}) = O\left( \frac {\diam(I_k)\log n}{n}\right),
\]
\[
\sum_k \sum_n \diam(P_{n,k})^{1+\delta}
\leq C \left(\sum_k \diam(I_k)^{1+\delta}\right)
\left(\sum_n \left(\frac {\log n}{n }\right)^{1+\delta}\right).
\]
The first term in the product is finite by construction (see Corollary \ref{cor:delta_n_powers} and Lemma \ref{lem:sum_edges_in_tower}) and the second is finite by a simple calculus exercise. 
Note that none of these bounds depend on $N$.
This shows that the $g$-preimages of the $\{Q_n\}$ 
have diameters whose $(1+\delta)$-powers have a finite
sum. Lemma \ref{correction map bound} shows that the
images of these preimages under the correction map
have comparable diameters, so the same is true for
the $f$-preimages, if the $g$-preimages 
lie outside $T(2r)$. But this is fulfilled if $N$ is large enough
so that $r_N$ is larger than $\exp(2A)$ where $A$ is the
upper bound on the $\tau$-sizes of the edges obtained
in Lemma \ref{vertex placement}.

In the argument so far we have assumed we started with a covering 
$\{Q_ n\}$ of one tract of $f$ and we only counted $g$-preimages
that were in a single tract of $g$. Since $f$ and $g$ both have 
$2N$ tracts, the sum over all preimages of coverings of all tracts
will be larger by a factor of $4N^{2}$, and this is 
clearly still finite. 
\end{proof}

\section{ Proof of (\ref{main est}): the iterative step}
\label{prove 1.2}

Next, we will prove the iterative step which finishes the proof of Theorem \ref{dno thm}.

\begin{lem}
	Suppose $f =g\circ\varphi^{-1}$ is as above. Let $D=D(w,r)$ 
be a disk with $|w|> r_N+1$ and $r  <1$ (so $D$ is 
disjoint from the closed disk $D(0,r_N)$). Let $\{ D_j\}$
 be the connected components of $ f^{-1}(D)$. Given
 $\var1epsilon >0$ and $\delta >0$, if $N$ is sufficiently large,  then
\[
\sum_j \diam(D_j)^{1+\delta} \leq \var1epsilon\cdot	\diam(D)^{1+\delta}.
\]
\end{lem}

\begin{proof}
We can write the inverse of $f$ by taking a complex
logarithm, followed by the conformal map of the 
right half-plane to an R-component, followed by the
correction map $\varphi$.  
The correction map will only be applied to sets where
it is bi-Lipschitz, so it only adds a uniformly bounded 
multiplicative factor to the sum in the lemma. 
	
First, the logarithm of the disk $D$ is an infinite
collection of disjoint sets of diameter $ O(r/|w|)$
arranged on the vertical line $\{ x = \log |w|\}$.
By Corollary \ref{tau image} we have 
\begin{eqnarray} \label{containment}
\tau(T(2r)) \subset V_M =\{x+iy:  0 < x  \leq    M\}.
\end{eqnarray}
If $N$ so large that $r_N>\exp( M)$, where  
$M$ is the $N$-independent 
 upper bound for the $\tau$-sizes 
of the edges of the graph, we have
\begin{equation}\label{eqn:One}
\log|w|>\log r_N> M.
\end{equation}
Each component  contains one point of the form 
$\log |w| + i (\theta+  2 \pi k),$ where $\theta =\arg(w)$.
By Equation \eqref{eqn:One},
all these sets lie in a sub-half-plane 
$$
\rhp+M= \rhp\setminus V_M = 
\{x+iy: x >  M\}.
$$
By Lemma \ref{lem:diameters of disks} the preimages
have diameters bounded by 
\begin{align*}
O\left(\left|\left(\tau^{-1}\right)' \left(1+ i\left(2 \pi  k+\theta\right) \right)
\right|\frac {r \log |w| }{|w|}\right)
&=O\left(\diam(I_k) \frac {r \log |w|}{|w|}\right).
\end{align*}
Since $|w|\geq r_N>N$ and $x\mapsto\log(x)/x$ is decreasing for $x>e$, this is bounded by 
\[
O\left(\diam(I_k)\frac {r \log |N| }{N}\right).
\]
Sum over  $k$,
use Lemma \ref{lem:sum_edges_in_tower} and 
Corollary \ref{cor:delta_n_powers},
 and recall that there are $2N$ 
$R$-components to consider:
\[
\sum_k\left(\diam D_k\right)^{1+\delta}
 \leq \frac C \delta  \cdot  r^{1+\delta}   \cdot 
            \frac { (\log N)^{1+\delta} N  }{N^{1+\delta} }  
 \leq \frac C \delta  \cdot  r^{1+\delta}   \cdot
            \frac { (\log N)^{1+\delta}  }{N^{\delta} } .
\]
Taking  $N$ large enough gives  the desired estimate
for the quasiregular map $g$.
By  \eqref{eqn:One} and  (\ref{containment}), 
the $g$-preimages  of $D$ lie outside $T(2r)$. 
By Lemma \ref{correction map bound}, the correction map 
expands the diameters  by a bounded factor 
(independent of $N$). Taking $N$ even larger, 
if necessary, proves the lemma. 
\end{proof} 

This completes the proof of Theorem \ref{dno thm}: Speiser
class Julia sets may have dimension as close to $1$ as  desired.

Our examples have three singular values. Is it possible 
to build entire functions with  only two singular values
whose Julia sets have dimensions as close to $1$ as we wish?
Does every dimension in $(1,2]$ occur for some $f \in \classS$?

Two  entire functions $f,g$ are said to be quasiconformally 
equivalent if there are quasiconformal homeomorphisms 
$\psi,\phi$ of the 
plane so that $ f \circ \psi= \phi \circ g$. The set of 
functions QC-equivalent to a given  Speiser class function $f$ 
forms a finite dimensional complex manifold  $M_f$ (the 
complex dimension is $q+2$ where $q$ is the number of 
singular values; see Section 3 of \cite{MR1196102}). 
Basic properties of quasiconformal maps imply that 
$\dim(\Julia(g))$ is a continuous function on 
$M_f$, and also imply that if this function attains the 
value $2$, then it is constant on $M_f$.
Can it ever  be non-constant?
Probably this is common, but it is not even 
clear what happens for the examples 
constructed in this paper.
Can  the dimension  ever be constant on $M_f$ with a 
value other than $2$?  If not,  is 
$\sup\{\dim(\Julia(g)): g \in M_f\}=2$ for every $f$? 

\section{The Julia set is a Cantor Bouquet} \label{topology} 

  A Cantor bouquet is a
closed set in the plane so that there is a homeomorphism 
of the plane mapping it to a ``straight brush'' 
in the sense of Aarts and Oversteegen, see \cite{MR1182980}.
This is a  subset  $B$ of $[0, \infty) \times (\reals \setminus 
\rationals)$ that is closed in the plane and  such that 
\newline 
(1) every  point of $ B$ is contained in a closed 
horizontal ray, called a hair, 
\newline
(2) the horizontal projection of $B$ onto the $y$ axis is 
dense, 
\newline
(3) any point of $B$ can be approached from above or below 
by  endpoints of hairs. 
\newline
Any two such sets are 
homeomorphic, and can even be mapped to each other by 
a homeomorphism of the plane (i.e., they are ambiently 
homeomorphic). See \cite{MR1182980}, 
\cite{MR991691}.

\begin{lem}
The Julia sets 
of our examples are Cantor bouquets.
\end{lem}

\begin{proof}
The construction in this paper
 shows that our examples   have the following
properties: $0$ is an attracting fixed point and all
 three singular values ($0,-1,1$) are contained in the
 basin of attraction of $0$ (which contains a large 
disk around the origin). This means our maps 
are of disjoint type, i.e., the singular set is compact and 
contained in the immediate attracting basin of an
attracting fixed point.

A tract of an entire  function is a connected component 
of $\{z: |f(z)| > R\}$.  These are unbounded Jordan domains.
If a tract $\Omega$  of $f$ does not contain zero, then
consider a connected component  
 $\Omega'$ of $ \log (\Omega)$ for some branch of the logarithm
 (there are
countably many such and any two are related  by a vertical 
translate by an integer multiple of $2 \pi$. We say that 
a curve  $\gamma \subset \Omega'$ 
 from a point $z_0 \in \overline{\Omega'}$ to $\infty$ has 
\emph{bounded wiggling} if    
 there are real constants $A,B$ so that 
for any point $z \in \gamma$
 we have $ \Real (z) >  A \cdot \Real(z_0) -B.  $
The  domain $\Omega'$ is said to have  uniformly
 bounded wiggling if all the 
hyperbolic geodesics from 
$z_0 \in \overline{\Omega'}$ to $\infty$ 
have bounded wiggling (with uniform constants).

Proposition 5.4 of \cite{RRRS} implies  that if $f$
is in the Eremenko-Lyubich class  and the tracts of
$f$ have uniformly bounded wiggling, then $f$ satisfied
a condition called the ``uniform linear head-start
condition''. We won't define this condition, but 
we note that Corollary 6.3 of \cite{Brushing} 
says that if an entire function is disjoint-type and
has the uniform head-start condition, then the Julia
set is a Cantor bouquet. Thus,  \cite{RRRS}
and \cite{Brushing} combined imply that disjoint-type
plus bounded wiggling imply the Julia set is a Cantor
bouquet. Hence, it suffices to  
show that our  examples have bounded wiggling.

The  logarithms of the $R$-components we construct in this
 paper clearly have this property. These  components
 are mapped to the tracts of our entire functions by 
the correction map $\varphi$. Let $R$ be a tract of
 the quasiregular map $g$ and $S=\varphi(R)$ the 
corresponding tract of the entire function $f$. Let
 $R'$ be a component of $\log(R)$ and $S'$ a component
 of $\log(S)$. Choosing the correct branch of the logarithm,
 the map $h(z)=\log(\varphi(\exp(z)))$ maps $R'$ conformally
 onto $S'$. The estimates of this paper show that $\varphi$
 is asymptotically conformal, i.e., there is a constant
 $c\neq 0$ so that
\[
\left|\frac{\varphi(z)}{z}-c\right|=o(1)\quad\text{as}\quad|z|\to\infty
\]
or in other words
$ \varphi(z)=cz+o(|z|)\quad\text{as}\quad|z|\to\infty. $ 
Hence, 
\[
\log(\varphi(z))=\log(z)+\log(c)+\log\left(1+o(1)\right)=\log(z)+O(1)
\]
which yields $|h(z) -z| = O(1)$. This easily implies
that hyperbolic geodesics in $R'$ map to curves
of bounded wiggling in $S'$. This in turn implies 
the hyperbolic  geodesics in $S'$ have bounded wiggling
by \cite[Lemma A.2]{RRRS}.
\end{proof}

For disjoint-type, finite-order maps, Bara\'nski \cite{MR2464786}
proved  that  
the Julia set is a Cantor bouquet  whose endpoints 
 have dimension   2 and the rest of
the bouquet has dimension $1$ (this was proven earlier for 
exponential maps by Karpi\'nska \cite{MR1696203}).
In our example, the non-endpoints are 
escaping (see Theorem 5.1 in   \cite{MR2675603})   and  this 
should imply that they have dimension
$1$, and so  we expect that 
 the dimension of the whole Julia set is concentrated
on the endpoints, but we leave this for future investigation.


\bibliography{dno}
\bibliographystyle{abbrv}

\end{document}